\documentclass[11pt,reqno]{amsart}
\usepackage{amsmath,amssymb,latexsym,esint,cite,mathrsfs}
\usepackage{verbatim,wasysym}
\usepackage[left=1.8cm,right=1.8cm,top=2.4cm,bottom=2.4cm]{geometry}
\usepackage{tikz,enumitem,graphicx, subfig, microtype, color}
\usepackage{epic,eepic}

\usepackage{tikz}
\usetikzlibrary{calc}

\usepackage[colorlinks=true,urlcolor=blue, citecolor=red,linkcolor=blue,
linktocpage,pdfpagelabels, bookmarksnumbered,bookmarksopen]{hyperref}
\usepackage[hyperpageref]{backref}
\usepackage[english]{babel}
\usepackage{appendix}

\numberwithin{equation}{section}

\newtheorem{thm}{Theorem}[section]
\newtheorem{lem}[thm]{Lemma}

\newtheorem{Prop}[thm]{Proposition}

\newtheorem{Rem}[thm]{Remark}

\begin{document}
	\title[Nonlinear elliptic problems with the fractional Laplacian]
	{Asymptotic behavior of solutions for the nonlinear Hartree equation involving the fractional Laplacian}

	\author[N. Borgia]{Natalino Borgia}
\address{\noindent Natalino Borgia \newline
		Dipartimento di Matematica, Universit\`{a} degli Studi di Bari Aldo Moro,\newline
		Via Orabona 4, 70125 Bari, Italy.}\email{natalino.borgia@uniba.it}

	\author[S. Cingolani]{Silvia Cingolani$^\dag$}
	\address{\noindent Silvia Cingolani  \newline
		Dipartimento di Matematica, Universit\`{a} degli Studi di Bari Aldo Moro,\newline
		Via Orabona 4, 70125 Bari, Italy.}\email{silvia.cingolani@uniba.it}
	
	\author[M. Yang]{Minbo Yang$^\ddag$}
	\address{\noindent Minbo Yang  \newline
		School of Mathematical Sciences, Zhejiang Normal University,\newline
		Jinhua 321004, Zhejiang, People's Republic of China.}\email{mbyang@zjnu.edu.cn}
	
	\author[S. Zhao]{Shunneng Zhao$^\S$}
	\address{\noindent Shunneng Zhao  \newline
		Dipartimento di Matematica, Universit\`{a} degli Studi di Bari Aldo Moro,\newline Via Orabona 4, 70125 Bari, Italy.
		\vspace{2mm}
		\newline
School of Mathematical Sciences, Zhejiang Normal University,\newline
		Jinhua 321004, Zhejiang, People's Republic of China.}
	\email{snzhao@zjnu.edu.cn}

\thanks{22 May 2026}

	\thanks{2020 {\em{Mathematics Subject Classification.}} Primary 35J08; 35B40;  Secondly 35J60, 46E35.}
	
	\thanks{{\em{Key words and phrases.}} A uniform $L^{\infty}$ bound; Asymptotic
behavior of solutions; blow-up; Fractional Laplacian.}

\thanks{Natalino Borgia is supported is supported by INdAM-GNAMPA project ''Metodi variazionali e topologici tra Fisica, Geometria e Scienze Applicate'' (CUP E53C25002010001).}

\thanks{$^\dag$Silvia Cingolani is supported by PNRR MUR project PE0000023 NQSTI - National Quantum Science and Technology Institute (CUP H93C22000670006) and partially supported by INdAM-GNAMPA.}

	\thanks{$^\ddag$Minbo Yang was supported by National Natural Science Foundation of China (12471114) and the Natural Science Foundation of Zhejiang Province (LZ26A010002).}

	\thanks{$^\S$Shunneng Zhao was supported by PNRR MUR project PE0000023 NQSTI - National Quantum Science and Technology Institute (CUP H93C22000670006) and National Natural Science Foundation of China (12401146) and Natural Science Foundation of Zhejiang Province (LMS25A010007).}
	
	\allowdisplaybreaks
	
	\begin{abstract}
		{\small
			In this paper, we investigate the fractional Laplacian problem with slightly subcritical exponents
\begin{equation*}
\left\lbrace
\begin{aligned}
    &A_{s} u=(|x|^{-(n-2s)}\ast u^{2_{s}^{\sharp}-1-\epsilon})u^{2_{s}^{\sharp}-2-\epsilon} \quad\quad\hspace{3.5mm} \mbox{in}\hspace{2mm}\Omega,\\
    &u>0\quad\quad \quad\quad\quad\quad\quad\quad\quad\quad\quad\quad\quad\quad\quad\quad\hspace{2mm}\mbox{in}\hspace{2mm}\Omega,\\
    &u=0\quad \quad\quad\quad\quad\quad\quad\quad\quad\quad\quad\quad\quad\quad\quad\quad\hspace{2mm}\mbox{on}\hspace{2mm}\partial\Omega,
   \end{aligned}
\right.
\end{equation*}
where $\Omega$ is a smooth bounded domain in $\mathbb{R}^n$, $0<s<1$, $n\in(2s,\min\{6s,n+2s\})$, $\epsilon>0$ small, $2_{s}^{\sharp}-1=(n+2s)/(n-2s)$ and $A_{s}$ stands for the spectral fractional Laplacian.
For a general domain $\Omega$ or domains with convexity, we first prove a uniform $L^1$ bound away from the boundary and a uniform $L^{\infty}$ bound near the boundary for positive solutions to the general fractional Hartree-type PDEs  by applying the moving planes method and integral estimates for the convolution term.
Among these results, we study the asymptotic behavior of solutions as $\epsilon\rightarrow0$.
These solutions are shown to blow-up at exactly one point $x_0$ and location of this point is characterized. In addition, the shape and exact rates for blowing-up are studied.
Finally, we also establish the corresponding main results for solutions of the fractional Brezis-Nirenberg problem involving critical Hartree-type nonlinearity.}
	\end{abstract}
	
	\vspace{3mm}
	
	\maketitle
	\section{Introduction}
	\subsection{Motivation and main results}

This paper is devoted to the study of the nonlinear nolocal problem involving slightly subcritical exponents
\begin{equation}\label{prondgr}
\left\lbrace
\begin{aligned}
    &A_{s} u=(|x|^{-(n-2s)}\ast u^{2_{s}^{\sharp}-1-\epsilon})u^{2_{s}^{\sharp}-2-\epsilon} \quad\quad\hspace{3.9mm} \mbox{in}\hspace{2mm}\Omega,\\
    &u>0\quad\quad \quad\quad\quad\quad\quad\quad\quad\quad\quad\quad\quad\quad\quad\quad\hspace{0.8mm}\mbox{in}\hspace{2mm}\Omega,\\
    &u=0\quad \quad\quad\quad\quad\quad\quad\quad\quad\quad\quad\quad\quad\quad\quad\quad\hspace{0.8mm}\mbox{on}\hspace{2mm}\partial\Omega,
   \end{aligned}
\right.
\end{equation}
where $\Omega$ is a smooth bounded domain in $\mathbb{R}^n$, $n\in(2s,\min\{6s,n+2s\})$, $s\in(0,1)$, $\ast$ denotes the standard convolution,
$\epsilon>0$ is a small parameter, $2_{s}^{\sharp}=\frac{2n}{n-2s}$ and $A_{s}$ stands for the spectral fractional Laplace operator $(-\Delta)^{s}$ in $\Omega$ with  outside zero Dirichlet boundary condition.
The fractional Laplacian $(-\Delta)^s$ with $0<s<1$ is the infinitesimal generator of symmetric L\'{e}vy stable diffusion processes on
$\mathbb{R}^n$. It appears in models of anomalous diffusion, flame propagation, chemical reactions, population dynamics, geophysical fluid dynamics, and financial option pricing \cite{Applebaum}.
To define $A_{s}$, let $(\phi_k,\lambda_k)$ be the eigenfunctions and eigenvectors of $(-\Delta)$ in $\Omega$ with zero Dirichlet boundary data. Then $(\phi_k,\lambda_k^s)$ are the eigenfunctions and eigenvectors
of $(-\Delta)^{s}$, also with Dirichlet boundary conditions. In fact, the fractional Laplacian $(-\Delta)^s$ is well defined in the space of functions
$$
H_0^{s}(\Omega)=\Big\{u=\sum a_k\phi_k\in L^{2}(\Omega)| \sum a_k^2\lambda_k^s<\infty\Big\},
$$
and, as a consequence,
\begin{equation}\label{awomiga}
A_{s}u=\sum a_k\lambda_k^{s}\phi_k\in H_0^{-s}(\Omega),
\end{equation}
where $H_0^{-s}(\Omega)$ is the dual space of $H_0^{s}(\Omega)$.

For the classical case $s=1$, the Lane-Emden equation
\begin{equation}\label{eq1.1-1}
A_{s}u=u^{2_{s}^{\sharp}-1-\epsilon},\hspace{2mm}u>0\hspace{2mm}\mbox{in}\hspace{2mm}\Omega\hspace{2mm}\mbox{and}\hspace{2mm}u=0\hspace{2mm}\mbox{on}\hspace{2mm}\partial\Omega.
\end{equation}
was extensively studied regarding the existence and asymptotic behavior of its blow-up solutions when $\epsilon\geq0$.
The studies on the asymptotic behavior of radial solutions of \eqref{eq1.1-1} with zero Dirichlet boundary values  were initiated by Atkinson and Peletier in \cite{ATKINSON-1986} by using ODE technique in the unit ball of $\mathbb{R}^3$.
Later, Brezis and Peletier \cite{BP} used the method of PDE to obtain the same results as that in \cite{ATKINSON-1986} for the spherical domains.
Finally, the same kind of results hold for nonspherical domain, which was settled by Han in \cite{HANZCHAO} (independently by Rey in \cite{Rey-1989}).
Moreover, this result was extended in \cite{Musso-Pistoia-2002}, where Musso and Pistoia obtained the existence of multi-peak solutions for certain domains.
Besides, according to Bahri-Li-Rey \cite{B-L-R} and Rey \cite{Rey-1999}, bubbling solutions are studied for $n>4$, concentrating
around nondegenerate critical points of certain objects which involve the Green's and and Robin's function of $\Omega$. See also \cite{LiWZ,Rey-1990,JW0,HL}.
For the asymptotic behavior of solutions of Lane-Emden system, we refer to \cite{Guerra} and \cite{CKL-1}. For the bubble solutions of nonlinear elliptic PDE with the slightly subcritical or supercritical, we refer to \cite{delPino-0, delPino-1, delPino-2, delPino-3, Musso-Pistoia-2010} and references therein.

Coming back to the subcritical problem \eqref{eq1.1-1} with zero Dirichlet boundary conditions, Choi, Kim and Lee \cite{CKL} (cf. \cite{CKL-2} for the elliptic system) developed a nonlocal analog of the results by Han \cite{HANZCHAO}
and Rey \cite{Rey-1989} above mentioned, namely if $u_{\epsilon}$ is a solution of \eqref{eq1.1-1} with zero Dirichlet boundary condition such that
$$
\lim\limits_{\epsilon\rightarrow0}\frac{\int_{\Omega}|A_{s}^{1/2}u_{\epsilon}|^2dx}{(\int_{\Omega}|u_{\epsilon}|^{2_{s}^{\sharp}}dx)^{2/2_{s}^{\sharp}}}=S_{n,s},
$$
exact rates of blow-up were given and the location of blow-up points were characterized.
There $S=S_{n,s}>0$ is a sharp constant of the fractional Sobolev inequality
	\begin{equation}\label{bsic}
		\int_{\mathbb{R}^n}|(-\Delta)^{s/2}u(x)|^2dx\geq S_{n,s}\big(\int_{\mathbb{R}^n}|u(x)|^{2^*_{s}}dx\big)^{\frac{2}{2^*_s}}\quad \mbox{for all}~~u\in \dot{H}^s(\mathbb{R}^n),
	\end{equation}
where the homogeneous Sobolev space $\dot{H}^s(\mathbb{R}^n)$ is defined as the completion of $C^\infty_0(\mathbb{R}^n)$ with respect to the norm	$$\|u\|_{\dot{H}^s(\mathbb{R}^n)}:=\Big(\int_{\mathbb{R}^n}|(-\Delta)^{s/2}u|^2dy\Big)^{1/2}=\Big(\int_{\mathbb{R}^n}|\hat{u}(\xi)|^2|\xi|^{2s}d\xi\Big)^{1/2}.$$
For $s\in(0,\frac{n}{2})$, an equivalent reformulation of the fractional Sobolev inequality \eqref{bsic}, known as the (diagonal) Hardy-Littlewood Sobolev inequality in \cite{H-L-1928,S1963} asserts that:
\begin{equation}\label{hlsi}
		\int_{\mathbb{R}^n}\int_{\mathbb{R}^n}u(x)|x-y|^{-\mu} v(y)dxdy\leq C_{n,r,t,\mu}\|u\|_{L^r(\mathbb{R}^n)}\|v\|_{L^t(\mathbb{R}^n)},
	\end{equation}
where $\mu\in(0,n)$, $1<r,t<\infty$ and $\frac{1}{r}+\frac{1}{t}+\frac{\mu}{n}=2$,
and $C_{n,r,t,\mu}>0$ is a positive constant.
Furthermore Lieb \cite{Lieb83} found the optimal constant and identified that the extremal functions of fractional Sobolev inequality are functions of the form
	\begin{equation}\label{minimizer}
		U[\xi,\lambda](x)=c_{n,s}\Big(\frac{\lambda}{1+\lambda^2|x-\xi|^2}\Big)^{\frac{n-2s}{2}},\hspace{4mm}\lambda\in\mathbb{R}^{+},\hspace{4mm}\xi\in\mathbb{R}^n,
	\end{equation}
	for $x\in\mathbb{R}^n$ and $c_{n,s}:=2^{2s}(\Gamma(\frac{n+2s}{2})/\Gamma(\frac{n-2s}{2}))^{\frac{n-2s}{4s}}$.
	Chen, Li and Ou \cite{Chen-ou}
	classified that the only positive solutions
to the Euler Lagrange equation associated with fractional Sobolev inequality are the bubbles described in \eqref{minimizer}.
In the general diagonal case $t=r=\frac{2n}{2n-\mu}$, Lieb \cite{Lieb83} showed that the extremal functions for the Hardy-Littlewood-Sobolev inequality and determined the optimal constant:
	\begin{equation}\label{defhlsbc}
		C_{n,\mu}=\pi^{\mu/2}\frac{\Gamma((n-\mu)/2)}{\Gamma(n-\mu/2)}\left(\frac{\Gamma(n)}{\Gamma(n/2)}\right)^{1-\frac{\mu}{n}}.
	\end{equation}
	Moreover, the equality holds if and only if
	\begin{equation*}
		u(x)=av(x)=a\big(\frac{1}{1+\lambda^2|x-x_0|^2}\big)^{\frac{2n-\mu}{2}}
	\end{equation*}
	for some $a\in \mathbb{C}$, $\lambda\in \mathbb{R}\backslash\{0\}$ and $x_0\in \mathbb{R}^n$.
According to the fractional Sobolev inequality and
	Hardy-Littlewood-Sobolev inequality, there exists an optimal constant $C_{HLS}>0$ depending only on $n$, $s$ and $\mu$ such that (cf. \cite{L-Y-Z})
	\begin{equation}\label{Prm}
		\int_{\mathbb{R}^n}|(-\Delta)^{\frac{s}{2}} u|^2dx\geq C_{HLS}\left(\int_{\mathbb{R}^n}(|x|^{-\mu} \ast|u|^{2_{\mu,s}^{\ast}})|u|^{2_{\mu,s}^{\ast}}dx\right)^{\frac{1}{2_{\mu,s}^{\ast}}}\hspace{2mm} \mbox{for all}~~u\in \dot{H}^s(\mathbb{R}^n),
	\end{equation}
where $2^{\ast}_{\mu,s}=\frac{2n-\mu}{n-2s}$. It is well-known that the Euler-Lagrange equation associated with \eqref{Prm} is given by
	\begin{equation}\label{ele-1.1-2}
		(-\Delta)^s u=(|x|^{-\mu}\ast |u|^{2_{\mu,s}^{\ast}})|u|^{2_{\mu,s}^{\ast}-2}u\quad \mbox{for all}~~u\in \dot{H}^s(\mathbb{R}^n).
	\end{equation}
	Furthermore, Le \cite{ple} classified all positive solutions of \eqref{ele-1.1-2} are functions of the form
	\begin{equation}\label{defU}	W[\xi,\lambda](x)=\alpha_{n,\mu,s}\big(\frac{\lambda}{1+\lambda^2|x-\xi|^2}\big)^{\frac{n-2s}{2}},\hspace{1mm}\lambda\in\mathbb{R}^{+},\hspace{1mm}\xi\in\mathbb{R}^n,
	\end{equation}
	and the constant $\alpha_{n,\mu,s}$ is given by
\begin{equation}\label{afal}
\alpha_{n,\mu,s}:=\Big(\frac{2^{2s}\Gamma(\frac{n+2s}{2})\Gamma(\frac{2n-\mu}{2})}{\pi^{n/2}\Gamma\big((n-2s)/2\big)\Gamma\big((n-\mu)/2\big)}\Big)^{\frac{n-2s}{2(n+2s-\mu)}}.
\end{equation}
For the case $s=1$ in \eqref{ele-1.1-2}, the authors of \cite{DY19, GHPS19,DAIQIN} independently computed the optimal constant $C_{HLS}>0$ and classified that the extremal functions of \eqref{Prm} are the bubbles $W[\xi,\lambda]$ described in \eqref{defU}.

Note that the fractional elliptic problem involving fractional operators, given by
\begin{equation}\label{gr}
\left\lbrace
\begin{aligned}
    &A_{s} u=f(u) \quad\quad\hspace{3.9mm} \mbox{in}\hspace{2mm}\Omega,\\
    &u=0\quad\quad\quad\quad\quad\hspace{2.5mm}\mbox{on}\hspace{2mm}\partial\Omega,
   \end{aligned}
\right.
\end{equation}
has attracted substantial research attention over the last decade (see, for instance, \cite{Colorado},\cite{CS},\cite{Cabr-1},\cite{LCS},\cite{Davila},\cite{Sugitani}, \cite{TAN-1}, and references therein).
In the case $s=1/2$, positive solutions to the nonlinear problem \eqref{gr} were studied by Cabre and Tan in \cite{CT},
where the authors also established the Gidas-Spruck-type a priori estimates.
Concerning regularity, Cabre and Tan \cite{CT} extended the $s=1$ results of Gidas and Spruck in \cite{Gidas} for problem \eqref{gr} $f(u)=u^{p}$ with $p<2_{s}^{\ast}-1$.
Their proof relies on blow-up analysis combined with nonlinear Liouville-type results in the whole space and half-space,
which are themselves proved via the Kelvin transform and moving planes/spheres methods.
Moreover, the result for $1/2<s<1$ was established by Tan in \cite{TAN}. Later, for more general nonlinearities and $0<s<1$,
a priori estimates for every weak solution were derived by Choi using a different approach based on the Pohozaev identity.
For the regularity of solutions to the Dirichlet problem for the fractional Laplacian, we refer to \cite{CS,Ros-Oton-1,Ros-Oton-2} and references therein.

The first purpose of the present paper is to establish a uniform $L^1$ bound away from the boundary and  a uniform $L^{\infty}$ bound near the boundary for solutions to problem
\begin{equation}\label{ele-100}
		\left\lbrace
\begin{aligned}
    &A_{s} u=(|x|^{-\mu}\ast F(u))f(u) \quad\quad\quad\quad\hspace{2.7mm} \mbox{in}\hspace{2mm}\Omega,\\
    &u>0\quad\quad \quad\quad\quad\quad\quad\quad\quad\quad\quad\quad\hspace{2mm}\quad\mbox{in}\hspace{2mm}\Omega,\\
    &u=0\quad \quad\quad\quad\quad\quad\quad\quad\quad\quad\quad\quad\hspace{2mm}\quad\mbox{on}\hspace{2mm}\partial\Omega,
   \end{aligned}
\right.
\end{equation}
where $n>2s$, $s\in(0,1)$, $F(u)=\int_{0}^{u}f(s)ds$ and $\mu\in(0,\min\{n,4s,\frac{n+2s}{2}\})$.

For each $r>0$, we set
$\mathcal{M}(\Omega,r)=\big\{x\in\Omega:~dist(x,\partial\Omega)\geq r\big\}$
and
$\mathcal{Q}(\Omega,r)=\big\{x\in\Omega:~dist(x,\partial\Omega)< r\big\}$.
We can prove the following result.
\begin{thm}\label{prior-1}
Assume that $\Omega$ is smooth bounded domain.
Let $f:\mathbb{R}_+\rightarrow\mathbb{R}$ be a locally
Lipschitz continuous function. Assume that $f(t)\geq0$ for all $t\geq0$, and
\begin{equation}\label{f1-0}
\liminf\limits_{u\rightarrow\infty}\frac{f(u)}{u}>\lambda_{1}^{s},\hspace{2mm}\mbox{and}\hspace{2mm}\lim\limits_{u\rightarrow\infty}\frac{f(u)}{u^{\delta}}=0,
\end{equation}
where $\lambda_{1}^{s}$ is the first eigenvalue of $A_{s}$ and $\delta=(n+2s-\mu)/(n-2s)$ if $n>2s$,
with one of the following assumptions:
\begin{itemize}
\item[$(f1)$] $\Omega$ is convex.

\item[$(f2)$]
\begin{equation}\label{f1-3}
\lim_{u\rightarrow0}\frac{f(u)}{u^{\sigma}}=0\hspace{2mm}\mbox{for}\hspace{2mm}\sigma\in\big(\max\{\frac{\mu}{n-2s},\frac{2_{s}^{\sharp}}{2},\frac{2s}{n-2s}\},2_{\mu,s}^{\ast}-1\big),
\end{equation}
and
\begin{equation}\label{f1-2}
\mbox{the function}\hspace{2mm}u\mapsto f(u)u^{-\frac{n+2s}{n-2s}}\hspace{2mm}\mbox{is nonincreasing on}\hspace{2mm}(0,\infty).
\end{equation}
\end{itemize}
Then there exist a small number $r>0$ and a constant $C=C(r,\Omega)>0$ such that for any solution $u$ of \eqref{ele-100}, there holds
\begin{equation}\label{boundary-1}
\int_{\mathcal{M}(\Omega,r)}u(x)dx\leq C\hspace{2mm}\mbox{and}\hspace{2mm}\int_{\mathcal{M}(\Omega,r)}\big(\int_{\Omega}\frac{F(u)}{|x-t|^{\mu}}dt\big)f(u)dx\leq C.
\end{equation}
Moreover, there is a constant $C=C(r,\Omega)>0$ such that
\begin{equation}\label{boundary-2}
\sup\limits_{x\in\mathcal{Q}(\Omega,r)}u(x)\leq C.
\end{equation}
\end{thm}
\begin{Rem}\label{Rem1-2}
A similar results for the equation with $A_{s}$ operator in a smooth bounded domain $\Omega$ with outside zero Dirichlet boundary condition can be obtained for the nonlinearity $f(u)=u^{2_{s}^{\sharp}-2-\epsilon}$ under the assumption \eqref{minimi} (see below) when $n>2s$
and for the nonlinearity $(|x|^{-\mu}\ast F(u))f(u)=(|x|^{-\mu}\ast u^{2_{\mu,s}^{\ast}})u^{2_{\mu,s}^{\ast}-1}+\epsilon u$ under the assumption \eqref{1-26} (see below) when $n>4s$.
  \end{Rem}
The proof of theorem \ref{prior-1} follows as a combination of the Kelvin transform, the moving planes method, maximum principle and some Hartree-type integral estimates.

Since $A_{s}$ is a nonlocal operator, we consider an $s$-harmonic extension $w\doteq E_{s}(u)$ to the cylinder $\mathcal{C}$ as the solution to the following problem
\begin{equation*}
\left\lbrace
\begin{aligned}
&-\mbox{div}(y^{1-2s}\nabla w)=0\hspace{4.14mm} \mbox{in}\hspace{2mm} \mathcal{C},\\
&w=0\hspace{12mm}\hspace{8mm}\hspace{9mm} \mbox{on}\hspace{2mm}\partial_{L}\mathcal{C},\\
&w=u\hspace{10mm}\hspace{8.5mm}\hspace{10mm} \mbox{on}\hspace{2mm}\Omega\times\{y=0\},
\end{aligned}
		\right.
\end{equation*}
where the half-cylinder standing on a bounded smooth domain $\Omega$ in $\mathbb{R}^n$ by
$$\mathcal{C}= \Omega\times(0,\infty)\subset R_{+}^{n+1}:=\Big\{z=(x,y)=(x_1,\cdots,x_n,y)\in\mathbb{R}^{n+1}|y>0\Big\},$$
and its lateral boundary by
$$\partial_{L}\mathcal{C}=\partial\Omega\times(0,\infty).$$
The extension function belongs to the space
$$
H_{0,L}^{s}(\mathcal{C}):=\overline{\mathcal{C}_{0}^{\infty}\big(\Omega\times[0,\infty)\big)}^{\|\cdot\|_{H_{0,L}^{s}(\mathcal{C})}}\quad\mbox{with}
\quad\big\|z\big\|_{H_{0,L}^{s}(\mathcal{C})}=\Big(k_s\int_{\mathcal{C}}y^{1-2s}|\nabla z|^2dxdy\Big)^{\frac{1}{2}},
$$
where $k_s$ is a normalization constant. With this constant we have that the extension operator is an isometry between $H_{0}^{s}(\Omega)$ and $H_{0,L}^{s}(\mathcal{C})$.
Moreover, for any function $\psi\in H_{0,L}^{s}(\mathcal{C})$, we have the following trace inequality
$$\big\|\psi(\cdot,0)\big\|_{H_{0}^{s}(\Omega)}\leq\big\|\psi\big\|_{H_{0,L}^{s}(\mathcal{C})}.$$
As shown in \cite{LCS,CS,CT,CDS}, $(-\Delta)^{s}$ can also be characterized by
\begin{equation}\label{FLDN}
-\frac{1}{\kappa_{s}}\lim_{y\rightarrow0^{+}}y^{1-2s}\frac{\partial w}{\partial y}(x,y)=(-\Delta)^{s}u(x)=c_{n,s}P.V.\int_{\mathbb{R}^n}\frac{u(x)-u(y)}{|x-y|^{n+2s}}dy.
\end{equation}
When $\Omega=\mathbb{R}^n$, the $s$-harmonic extension and the fractional Laplacian have explicit expressions in terms of the Poisson and the Riesz kernels, respectively
\begin{equation}\label{FLDN1}
w(x,y)=c_{n,s}y^{2s}\int_{\mathbb{R}^n}\frac{u(t)}{(|x-t|^2+y^2)^{\frac{n+2s}{2}}}dt\hspace{2mm}\mbox{and}\hspace{2mm}(-\Delta)^{s}u(x)=d_{N,s}P.V.\int_{\mathbb{R}^N}\frac{u(x)-u(t)}{|x-t|^{n+2s}}dt,
\end{equation}
In fact the extension technique is developed originally for the fractional Laplacian defined in the whole space \cite{CS}. The constants in \eqref{FLDN}
and \eqref{FLDN1} satisfy the identity $2sc_{N,s}k_{s}=d_{n,s}$. Their explicit value can be consulted for instance in \cite{Colorado}.
Thus we can see that the problem \eqref{ele-100} can be transformed into the following problem
For function $f:[0,\infty)\rightarrow\mathbb{R}$ and let $w$ be the s-harmonic extension of $u$ to half cylinder $\Omega\times[0,\infty)$, that is, $w$ satisfies $tr|_{\Omega\times\{0\}}w=u$ and it solves
\begin{equation}\label{ele-100-1}
\left\lbrace
\begin{aligned}
&-\mbox{div}(y^{1-2s}\nabla w)=0\hspace{4.14mm}\quad\quad\quad \quad\quad\quad \quad  \mbox{in}\hspace{2mm} \mathcal{C},\\
&w=0\quad\quad \quad\quad\quad\quad\hspace{12mm}\hspace{2mm}\hspace{8.9mm}\hspace{10mm} \mbox{on}\hspace{2mm}\partial_{L}\mathcal{C},\\
&\partial_{\nu}^{s}w=\big(|x|^{-(n-2s)}\ast F(w)\big)f(w)\quad\quad\quad\hspace{3.3mm} \mbox{in}\hspace{2mm}\Omega\times\{y=0\}.
\end{aligned}
		\right.
\end{equation}
where
\begin{equation}\label{CFLL-00}
\partial_{\nu}^{s}w(x,0)\doteq-\frac{1}{\kappa_{s}}\lim_{y\rightarrow0^{+}}y^{1-2s}\frac{\partial w}{\partial y}(x,y), \ \ \mbox{for}\ \  x\in\Omega.
\end{equation}
An energy solution to this problem is a function $H_{0,L}^{s}(\mathcal{C})$ such that
$$
\frac{1}{k_s}\int_{\mathcal{C}}y^{1-2s}\big\langle\nabla w,\psi\big\rangle dxdy=\int_{\Omega}\big(|x|^{-(n-2s)}\ast w^{2_{s}^{\sharp}-1-\epsilon}\big)w^{2_{s}^{\sharp}-2-\epsilon}\psi dx,\quad\forall\psi\in H_{0,L}^{s}(\mathcal{C}).
$$
For any energy solution $w\in H_{0,L}^{s}(\mathcal{C})$ to \eqref{ele-100-1},
the function $u=w(\cdot,0)$, defined in the sense of traces, belongs to the space $H_0^{s}(\Omega)$ and is an energy solution to problem \eqref{ele-100}.

Once one establishes a uniform bound for $u$ near the blow-up point, one of the next natural questions is to examine the shape and exact
rates for blowing-up of solutions. The aim of the latter part of the present work is to consider the asymptotic behavior of solutions for problem \eqref{prondgr} involving fractional order operators with an almost critical exponent.
Its proof involves the method of blow-up argument, the moving planes method combined with the Kelvin transform, Morse iteration technique and integral inequality involving convolution terms.

In this line of research, we would like to mention some related results to our problem.
Recently, Deng and Luo in \cite{DL-2} used the finite
dimensional reduction method to consider the slightly subcritical problem
\begin{equation}\label{112ele}
    \left\lbrace
\begin{aligned}
    &(-\Delta)^su=(|x|^{-\mu}\ast u^{2^{\ast}_{\mu,s}-\epsilon})u^{2^{\ast}_{\mu,s}-1-\epsilon}\quad \hspace{1.6mm}\mbox{in}\hspace{2mm}\Omega, \\
    &u>0\quad\quad \quad\quad\quad\quad\quad\quad\quad\quad\quad\quad\quad\quad\hspace{2mm}\mbox{in}\hspace{2mm}\Omega,\\
    &u=0\quad \quad\quad\quad\quad\quad\quad\quad\quad\quad\quad\quad\quad\quad\hspace{2mm}\mbox{on}\hspace{2mm}\partial\Omega,
   \end{aligned}
\right.
\end{equation}
They proved that if $n>2s$, $s\in(0,\min\{1,\frac{3n-2}{6}\})$ and $\mu\in(0,\min\{4s,n\})$, as $\epsilon\rightarrow0$
problem \eqref{112ele} admits solution, which blows up and concentrates around the local minimum point of the Robin
function $H(x,x)$, where $H(x,y)$ is the regular part of the Green function $G(x,y)$, i.e.
   \begin{equation}\label{Robin}
	H(x,y)=\frac{\gamma_{n,s}}{|x-y|^{n-2s}}-G(x,y)\hspace{2mm}\mbox{with}\hspace{2mm}
\gamma_{n,s}=\frac{1}{|S^{n-1}|}\cdot\frac{2^{1-2s}\Gamma(\frac{n-2s}{2})}{\Gamma(\frac{n}{2})\Gamma(s)},
\end{equation}
where $(-\Delta)^sG(\cdot,y)=\delta_y$ in $\Omega$ and $G(\cdot,y)=0$ on $\partial\Omega$.
The weak solution theory (regularity, existence, symmetry, decay) for the fractional critical Choquard problem
with $p\in(\frac{2n-\mu}{n},\frac{2n-\mu}{n-2s})$ was developed by d'Avenia et al. in \cite{Avenia}. Extending further,
Mukherjee and Sreenadh in \cite{MTS} addressed a fractional Choquard-type Br\'{e}zis-Nirenberg problem.
A problem analogous to \eqref{112ele} was considered by Ghimenti et al. \cite{Ghimenti-3} using variational techniques.
They applied Lusternik-Schnirelmann category theory and nonlocal compactness arguments to obtain multiple positive solutions.
When $s=1$,	Chen and Wang \cite{cw}
studied the asymptotic behavior of solutions for \eqref{112ele} by employing the finite dimensional reduction method.
Later on, the latter three authors of this work in \cite{Cingolani} studied that the exact rates of blow-up and the location of blow-up points of the solutions for problem \eqref{112ele}.
For the studied of nonlocal Hartree problem, we refer to \cite{CCYS,Cingolani-2,DL-1, Ghimenti-1,Ghimenti-2,cw-1,V-Moroz} and references therein.

 We will now state the other main results of this paper. Let us start with a asymptotic behavior of single blow-up solutions.
\begin{thm}\label{1-prondgr}
Assume that $2s<n<\min\{6s,n+2s\}$, $0<s<1$ and $\epsilon$ is sufficiently small. Let $u_\varepsilon$ be a solution of \eqref{prondgr} such that
\begin{equation}\label{minimi}
\frac{\int_{\Omega}|A_{s}^{1/2}  u_{\epsilon}|^2dx}{\left[\int_{\Omega}(|x|^{-(n-2s)} \ast|u_{\epsilon}|^{2_{s}^{\sharp}-1-\epsilon})|u_{\epsilon}|^{2_{s}^{\sharp}-1-\epsilon} dx\right]^{\frac{1}{2_{s}^{\sharp}-1-\epsilon}}}=C_{HLS}+o(1)\quad\mbox{as}\quad\epsilon\rightarrow0.
\end{equation}
 Then
there exists $x_0\in\Omega$ such that, after passing to a subsequence, we have
\begin{equation*}
u_\epsilon\rightarrow0\hspace{2mm}\mbox{in}\hspace{2mm}
\left\lbrace
\begin{aligned}
&C_{loc}^\alpha(\Omega\setminus\{x_0\})\hspace{2mm}\mbox{for all}\hspace{2mm}\alpha\in(0,2s)\hspace{2mm}\mbox{if}\hspace{2mm}s\in(0,1/2],\hspace{2mm}\mbox{as}\hspace{2mm}\epsilon\rightarrow0,
\\&C_{loc}^{1,\alpha}(\Omega\setminus\{x_0\})\hspace{2mm}\mbox{for all}\hspace{2mm}\alpha\in(0,2s-1)\hspace{2mm}\mbox{if}\hspace{2mm}s\in(1/2,1),\hspace{2mm}\mbox{as}\hspace{2mm}\epsilon\rightarrow0.
   \end{aligned}
\right.
\end{equation*}
	\end{thm}

\begin{thm}\label{prondgr-1}
Let the assumptions of Theorem \ref{1-prondgr} be satisfied and $x_0$ is a point given by Theorem \ref{prondgr}. Then $x_0$ is a critical of $\phi(x):=H(x,x)$ (Robin's function of $\Omega$) for $x\in\Omega$.
\end{thm}

The following the behavior of solutions in terms of the Green's function.
\begin{thm}\label{consequence}
 Let the assumptions of Theorem \ref{1-prondgr} be satisfied. Then
\begin{equation*}
\begin{split}
\lim\limits_{\epsilon\rightarrow0^{+}}&\|u_{\epsilon}\|_{L^{\infty}(\Omega)}u_{\epsilon}(x)=b_{n,s}G(x,x_{0})\\&
\mbox{in}\hspace{2mm}
\left\lbrace
\begin{aligned}
&C_{loc}^\alpha(\Omega\setminus\{x_0\})\hspace{2mm}\mbox{for all}\hspace{2mm}\alpha\in(0,2s)\hspace{2mm}\mbox{if}\hspace{2mm}s\in(0,1/2],\hspace{2mm}\mbox{as}\hspace{2mm}\epsilon\rightarrow0,
\\&C_{loc}^{1,\alpha}(\Omega\setminus\{x_0\})\hspace{2mm}\mbox{for all}\hspace{2mm}\alpha\in(0,2s-1)\hspace{2mm}\mbox{if}\hspace{2mm}s\in(1/2,1),\hspace{2mm}\mbox{as}\hspace{2mm}\epsilon\rightarrow0,
   \end{aligned}
\right.
\end{split}
\end{equation*}
where
$$
b_{n,s}:=\frac{|S^{n-1}|}{2}\frac{\Gamma(s)\Gamma(n/2)}{\Gamma((n+2s)/2)}\alpha^{2_{s}^{\sharp}}_{n,s}\widetilde{\beta}_{n,s}
$$
with $\widetilde{\beta}_{n,s}=\frac{\pi^{n/2}\Gamma(s)}{\Gamma(\frac{n+2s}{2})}\sqrt{\frac{2^{2s}\Gamma(\frac{n+2s}{2})\Gamma(\frac{n+2s}{2})}{\pi^{n/2}\Gamma(\frac{n-2s}{2})\Gamma(s)}}$ and $\alpha_{n,s}:=\alpha_{n,n-2s,s}$ is defined in \eqref{afal}.
\end{thm}

We give the rate of blow-up of the maximum of the solutions.
	\begin{thm}\label{remainder terms}
		Let the assumptions of Theorem \ref{1-prondgr} be satisfied. Then
\begin{equation*}
\lim\limits_{\epsilon\rightarrow0}\frac{(n-2s)^2}{2[n+2s-\epsilon(n-2s)]}\epsilon\big\|u_{\epsilon}\big\|_{L^{\infty}(\Omega)}^{2}
=\frac{(n-2s)^2\gamma_{n,s}b_{n,s}^2}{2\kappa_{s}\alpha^2_{n,s}\widetilde{\beta}_{n,s}B_{n,s}}M_{n,s}|\phi(x_0)|,
\end{equation*}
where
$$B_{n,s}=\sigma_n\int_{0}^{\infty}\frac{r^{n-1}}{(1+r^2)^n}dr\hspace{2mm}\mbox{and}\hspace{2mm}M_{n,s}=\sigma_n\int_{0}^{1}\frac{r^{n-1}}{(1-r^2)^s}dr,
$$
where $\sigma_n$ is the area of the unit sphere in $\mathbb{R}^n$.
	\end{thm}

Also, we study the Brezis-Nirenberg problem involving critical Hartree nonlinearity
\begin{equation}\label{ele-2}
    \left\lbrace
\begin{aligned}
    &A_{s} u=(|x|^{-\mu}\ast u^{2^{\ast}_{\mu,s}})u^{2^{\ast}_{\mu,s}-1}+\epsilon u \quad\quad\hspace{2.7mm} \mbox{in}\hspace{2mm}\Omega, \\
    &u>0\quad\quad \quad\quad\quad\quad\quad\quad\quad\quad\quad\quad\quad\quad\hspace{2mm}\mbox{in}\hspace{2mm}\Omega,\\
    &u=0\quad \quad\quad\quad\quad\quad\quad\quad\quad\quad\quad\quad\quad\quad\hspace{2mm}\mbox{on}\hspace{2mm}\partial\Omega,
   \end{aligned}
\right.
\end{equation}
where $\Omega$ is a smooth bounded domain in $\mathbb{R}^n$, $n>2s$, $s\in(0,1)$, $\ast$ denotes the standard convolution, $\epsilon>0$ is a small parameter and $2_{\mu,s}^{\ast}=\frac{2n-\mu}{n-2s}$, $\mu\in(0,\min\{n,4s\})$ and $\epsilon>0$ is a small parameter
and $\Omega$ is a smooth bounded domain of $\mathbb{R}^n$.
The existence if positive solutions for this problem can be found in \cite{MTS} and related topics, we refer the readers to \cite{GY18,Ghimenti-1,Ghimenti-2} and reference therein.
In, \cite{DL-1} Deng et al. proved that \eqref{ele-2} has a family of solutions which is concentrate around the
critical points of the Robin function as $\epsilon\rightarrow0$ for $n>4s$, $\mu\in(2s,4s)$ and $s\in(0,1)$.
For the local case, asymptotic behavior of solutions are proved in \cite{HANZCHAO, Rey-1989,Rey-1990}.
To the best of our knowledge, the problem of the asymptotic behaviors of $\{u_{\epsilon}\}_{\epsilon>0}$ and exact rates of single blow-up solutions for problem \eqref{ele-2} has not been investigated so far. We shall address this gap by proving the following results.
	\begin{thm}\label{emm}
		Assume that $n>4s$ and $0<s<1$, $2_{\mu,s}^{\ast}=\frac{2n-\mu}{n-2s}$, $\mu\in(0,\min\{4s,\frac{n+2s}{2}\})$ and $\epsilon$ is sufficiently small. Let $u_{\epsilon}$ be a solution of \eqref{ele-2}
and satisfying
\begin{equation}\label{1-26}
\frac{\int_{\Omega}|A_{s}^{1/2}  u_{\epsilon}|^2dx}{\left[\int_{\Omega}(|x|^{-(n-2s)} \ast|u_{\epsilon}|^{2_{\mu,s}^{\ast}})|u_{\epsilon}|^{2_{\mu,s}^{\ast}} dx\right]^{\frac{1}{2_{\mu,s}^{\ast}}}}=C_{HLS}+o(1)\quad\mbox{as}\quad\epsilon\rightarrow0.
\end{equation}
Then
there exists $x_0\in\Omega$ such that, after passing to a subsequence, we have
\begin{equation*}
u_\epsilon\rightarrow0\hspace{2mm}\mbox{in}\hspace{2mm}
\left\lbrace
\begin{aligned}
&C_{loc}^\alpha(\Omega\setminus\{x_0\})\hspace{2mm}\mbox{for all}\hspace{2mm}\alpha\in(0,2s)\hspace{2mm}\mbox{if}\hspace{2mm}s\in(0,1/2],\hspace{2mm}\mbox{as}\hspace{2mm}\epsilon\rightarrow0,
\\&C_{loc}^{1,\alpha}(\Omega\setminus\{x_0\})\hspace{2mm}\mbox{for all}\hspace{2mm}\alpha\in(0,2s-1)\hspace{2mm}\mbox{if}\hspace{2mm}s\in(1/2,1),\hspace{2mm}\mbox{as}\hspace{2mm}\epsilon\rightarrow0,
   \end{aligned}
\right.
\end{equation*}
and
\begin{equation*}
\begin{split}
\lim\limits_{\epsilon\rightarrow0^{+}}&\|u_{\epsilon}\|_{L^{\infty}(\Omega)}u_{\epsilon}(x)=d_{n,s}G(x,x_{0})\\&
\hspace{2mm}\mbox{in}\hspace{2mm}
\left\lbrace
\begin{aligned}
&C_{loc}^\alpha(\Omega\setminus\{x_0\})\hspace{2mm}\mbox{for all}\hspace{2mm}\alpha\in(0,2s)\hspace{2mm}\mbox{if}\hspace{2mm}s\in(0,1/2],\hspace{2mm}\mbox{as}\hspace{2mm}\epsilon\rightarrow0,
\\&C_{loc}^{1,\alpha}(\Omega\setminus\{x_0\})\hspace{2mm}\mbox{for all}\hspace{2mm}\alpha\in(0,2s-1)\hspace{2mm}\mbox{if}\hspace{2mm}s\in(1/2,1),\hspace{2mm}\mbox{as}\hspace{2mm}\epsilon\rightarrow0.
   \end{aligned}
\right.
\end{split}
\end{equation*}
where
$$
d_{n,s}:=\frac{|S^{n-1}|}{2}\frac{\Gamma(s)\Gamma(n/2)}{\Gamma((n+2s)/2)}\alpha^{2_{s}^{\sharp}}_{n,\mu,s}\widetilde{\beta}_{n,\mu,s}
$$
with $\widetilde{\beta}_{n,\mu,s}=\frac{\pi^{n/2}\Gamma\big(\frac{n-\mu}{2}\big)}{\Gamma(\frac{2n-\mu}{2})}\Big(\frac{2^{2s}\Gamma(\frac{n+2s}{2})\Gamma(\frac{2n-\mu}{2})}{\pi^{n/2}\Gamma\big(\frac{n-2s}{2}\big)\Gamma\big(\frac{n-\mu}{2}\big)}\Big)^{\frac{n-\mu}{n+2s-\mu}}$ and $\alpha_{n,\mu,s}$ is defined in \eqref{afal}.
\end{thm}

\begin{thm}\label{emm-1}
		Let the assumptions of Theorem \ref{emm} be satisfied. Then
\begin{itemize}
\item[$(a)$] $x_0$ is a critical of $\phi(x)$ for $x\in\Omega$.
\item[$(b)$] It holds that
\begin{equation*}
\lim\limits_{\epsilon\rightarrow0}\epsilon\|u_\epsilon\|_{L^{\infty}(\Omega)}^\frac{2n-8s}{n-2s}=\frac{(n-2s)^2\gamma_{n,s}d_{n,s}^2}{2sk_{s}F_{n,s}}M_{n,s}|\phi(x_0)|
\end{equation*}
where $F_{n,s}=\sigma_n\int_{0}^{\infty}\frac{r^{n-1}}{(1+r^2)^{n-2s}}dr$ and $M_{n,s}$ is found in Theorem \ref{remainder terms}.
\end{itemize}
	\end{thm}

\subsection{Structure of the paper}
The paper is organized as follows. In section \ref{sec:weaksolution}, we prove Theorem \ref{prior-1} for the strictly convex domain and the general domain.
Based on this result and an adequate decay estimate of solutions from Lemma \ref{cWU}, which be established in section \ref{sec:sobolev}, we then complete the proof of Theorems \ref{1-prondgr} and \ref{consequence} in  the section \ref{consequence00}.
In section \ref{remainder}, we prove Theorems \ref{prondgr-1} and \ref{remainder terms} by exploiting various identities which is related to local Hartree-type Pohozaev identity
 and blow-up argument. Finally, in section \ref{section7}, we conclude that Theorems \ref{emm} and \ref{emm-1}
 	
	Throughout this paper, $c$ and $C$ are indiscriminately used to denote various absolutely positive constants.

\section{A uniform $L^{\infty}$ bound near the boundary for solutions}\label{sec:weaksolution}
In this section, we are devote to establish a uniform $L^1$ bound away from the boundary and a uniform $L^{\infty}$ bound near the boundary for positive solutions by applying the moving planes method and maximum principle.
\begin{proof}[Proof of Theorems \ref{prior-1}]
First we give the proof of \eqref{boundary-1}. By the virtue of the symmetry,
\eqref{awomiga} and \eqref{ele-100}, we have
\begin{equation}\label{lamata2}
\int_{\Omega}\lambda_{1}^s\phi_1u(x)dx=\int_{\Omega}(A_{\Omega}^{s}\phi_1)u(x)dx=\int_{\Omega}(|x|^{-\mu}\ast F(u))f(u)\phi_1(x)dx.
\end{equation}
$\bullet$ We prove now that, there exists a constant $C>0$ such that
\begin{equation}\label{fanzen-1}
\int_{\Omega}\phi_1u(x)dx\leq C\int_{\Omega}\phi_1(x)dx\leq C.
\end{equation}
We proceed by contradiction to prove \eqref{fanzen-1}. Assume that
there are a sequence $\{u_{k}\}$ which are solution of \eqref{ele-100}, up to subsequences, and such that
\begin{equation}\label{fanzen-2}
\int_{\Omega}u_k(x)\phi_1(x)dx\rightarrow\infty\hspace{2mm}\mbox{as}\hspace{2mm}k\rightarrow\infty.
\end{equation}
According to \eqref{f1-0}, there exist a small $\delta>0$ and a large $M>0$ such that $f(u)>(\lambda_{1}^{s}+\delta)u-M$ for all $u>0$, and we deduce from this that
$$
F(u)>\frac{1}{2}(\lambda_{1}^{s}+\delta)u^2-C_{M}.
$$
We set $R=diam(\Omega)$. Then, for any $x,y\in\Omega$, we have
$$
\int_{\Omega}\frac{F(u_k)}{|x-y|^{\mu}}dy\geq\frac{1}{R^{\mu}}C_1\int_{\Omega}u_k^2(x)dx-C_2.
$$
Thus, there exists a constant $\rho_0\geq1$ such that for the sufficiently large $k$, there holds $\int_{\Omega}|x-y|^{-\mu}F(u_k)dy\geq\rho_0$. Combining this bound with \eqref{lamata2},
we get
$$
\int_{\Omega}\lambda_{1}^s\phi_1u_k(x)dx=\int_{\Omega}(|x|^{-\mu}\ast F(u_k))f(u_k)\phi_1(x)dx\geq\delta_0\int_{\Omega}[(\lambda_{1}^{s}+\delta)u_k-M]\phi_1(x)dx.
$$
This leads to
$$
\big(\rho_0(\lambda_{1}^s+\delta)-\lambda_{1}^s\big)\int_{\Omega}\phi_1u_k(x)dx\leq C\int_{\Omega}\phi_1(x)dx\leq C,
$$
for the sufficiently large $k$, which is a contradiction.

On the other hand,  there exists a constant $C=C(\Omega,r)$ such that $\phi_1\geq C$ for all $x\in \mathcal{M}(\Omega,r)$. It follows from \eqref{fanzen-1} that
\begin{equation}\label{fanzen-3}
\int_{\mathcal{M}(\Omega,r)}u(x)dx\leq C\int_{\Omega}u(x)\phi_1(x)dx\leq C.
\end{equation}
Coming back to \eqref{lamata2}, \eqref{boundary-1} follows immediately from \eqref{fanzen-3}.

We next prove that \eqref{boundary-2}.
We consider two cases, depending whether $\Omega$ is strictly convex, or not.

\textbf{Case 1.} For the general domains.

We denote by $w$ the $s$-harmonic extension of $u$, and $\mathcal{C}^{\ast}=T(\mathcal{C})$ and making a
change of variables as \eqref{WWTI} in section \ref{sec:sobolev},
$$w^{\ast}(z)=|z|^{-(n-2s)}w\big(T(z)\big).
$$
Then by \eqref{FLDN}, we have
\begin{equation*}
\begin{split}
\partial_{\nu}^{s}w^{\ast}&=-\frac{1}{\kappa_{s}}\lim_{y\rightarrow0^{+}}y^{1-2s}\frac{\partial }{\partial y}\big[|z|^{-(n-2s)}w\big(T(z)\big)\big]\\&
=|z|^{-(n+2s)}\lim\limits_{y\rightarrow0}\big[(\frac{y}{|z|^{2}})^{1-2s}\frac{\partial}{\partial y}w\big(T(z)\big)\big].
\end{split}
\end{equation*}
Thus, combining a identity
$$T(x)-T(y^{\prime})=|x|^{-1}|y^{\prime}|^{-1}|x-y^{\prime}|$$
for $x,y^{\prime}\in\mathbb{R}^n\setminus\{0\}$, we get
\begin{equation*}
\left\lbrace
\begin{aligned}
&-\mbox{div}(y^{1-2s}\nabla w^{\ast})=0\hspace{4.14mm}\hspace{19.3mm}\quad\quad\quad\quad\quad \quad \quad \quad \quad \quad \quad  \quad\hspace{5mm}\mbox{in}\hspace{2mm} \mathcal{C}^{\ast},\\
&w^{\ast}=0\hspace{19.3mm}\quad\quad\quad\quad \quad\quad\quad \quad \quad \quad \quad \hspace{12.5mm}\hspace{3.5mm}\hspace{8.9mm}\hspace{10mm}\hspace{3mm} \mbox{on}\hspace{2mm}\partial_{L}\mathcal{C}^{\ast},\\
&w^{\ast}>0\hspace{19.3mm}\quad\quad\quad\quad\quad\quad \quad \quad \quad \quad \quad \quad \hspace{10mm}\hspace{10mm}\hspace{10.5mm}\hspace{3mm}\hspace{1mm}\mbox{in}\hspace{2mm}\mathcal{C}^{\ast},\\
&\partial_{\nu}^{s}w^{\ast}=\frac{1}{|x|^{n+2s}}\Big(\int_{T(\Omega\times\{0\})}\frac{F(|y^{\prime}|^{n-2s}w^{\ast})}{|x-y^{\prime}|^{\mu}|y^{\prime}|^{2n}} dy^{\prime}\Big)f(|x|^{n-2s}w^{\ast})\hspace{3mm} \quad\quad\quad \mbox{in}\hspace{2mm}T(\Omega\times\{0\}).
\end{aligned}
		\right.
\end{equation*}

Thanks to $\Omega$ is a smooth bounded domain in $\mathbb{R}^N$, and let $x_0\in\partial\Omega$ be a boundary point. Then, there exists a ball externally tangent to $\partial\Omega$ at $x_0$.
By translation and scaling, we may assume without loss of generality that
$x_0=(1,0,\cdots,0)$ and this tangent ball is the unit ball $B(0,1)$, so that
$\Omega\subset\mathbb{R}^N\setminus{\overline{B(0,1)}}$.
Recalling the extended cylinder $\mathcal{C}$, which satisfies $\mathcal{C}\cap B(0,1)=\emptyset$. By the Kelvin transform, $T(\mathcal{C})\subset B(0,1)$ and lies near the rightmost point $(1,0,\cdots,0)$.
Let $$T_{\lambda}(x)=\{z=(x_1,x_2,\cdots,x_{n+1})\in\mathbb{R}^{n+1}_{+}|x_1=1-\lambda\}.$$ This is a plane perpendicular to x1-axis and the plane that we will move with. Denoting
\begin{equation*}
\begin{split}
\Sigma_{\lambda}=T(\mathcal{C})\cap\{z\in\mathbb{R}_+^{n+1}: |z|\leq1, z_1>1-\lambda\},\hspace{4mm}\partial_{bd}\Sigma_{\lambda}=\Sigma_{\lambda}\cap\partial \mathbb{R}^{n+1}_{+}.
\end{split}
\end{equation*}
Let
\begin{equation*}
x^{\lambda}=(2-2\lambda-x_1,x_2,\cdots,x_{n+1}),
\end{equation*}
the reflection of the point $(x_1,x_2,\cdots,x_{n+1})$ about $T_{\lambda}$. In the following we always write
$$w_{\lambda}^{\ast}=w^{\ast}(x^{\lambda}),\hspace{2mm}\mbox{and}\hspace{2mm}v_{\lambda}(x)=w_{\lambda}^{\ast}-w^{\ast}(x)\hspace{2mm}\mbox{on}\hspace{2mm} \Sigma_{\lambda}.$$
We want to show that $v^{\lambda}\geq0$ as $\epsilon>0$ sufficiently small.
we generally go through the following two steps. To this end, we set $v^{-}_{\lambda}=\max\{0,-v_{\lambda}\}$.

Step 1. We first show that for $\epsilon$ sufficiently small, we have
\begin{equation}\label{v}
v^{-}_{\lambda}\equiv0\hspace{2mm}\mbox{for}\hspace{2mm}x\in\partial_{bd}\Sigma_{\lambda}.
\end{equation}
We note that
\begin{equation*}
\mbox{div}(y^{1-2s}\nabla v_{\lambda})=\mbox{div}(y^{1-2s}\nabla w^{\ast}(y^{\lambda}))-\mbox{div}(y^{1-2s}\nabla w^{\ast})=0\hspace{2mm}\mbox{on}\hspace{2mm}\partial\Sigma_{\lambda}.
\end{equation*}
Hence
\begin{equation}\label{bd}
\begin{split}
\int_{\partial_{bd}\Sigma_{\lambda}}\partial_{\nu}^{s}v_{\lambda}v^{-}_{\lambda}dx+\int_{\Sigma_{\lambda}}y^{1-2s}|\nabla v^{-}_{\lambda}|^2dxdy=0.
\end{split}
\end{equation}
In what follows, we denote
\begin{equation*}
\mathcal{K}(x,w^{\ast}):=\Big(\int_{T(\Omega\times\{0\})}\frac{F(|y^{\prime}|^{n-2s}w^{\ast})}{|x-y^{\prime}|^{\mu}|y^{\prime}|^{2n}} dy^{\prime}\Big)\frac{f(|x|^{n-2s}w^{\ast})}{|x|^{n+2s}}.
\end{equation*}
Then we infer that
\begin{equation}\label{bdel}
\begin{split}
\int_{\partial_{bd}\Sigma_{\lambda}}\partial_{\nu}^{s}v_{\lambda}(-v^{-}_{\lambda})dx
&=\int_{\partial_{bd}\Sigma_{\lambda}}\big[\mathcal{K}(x^{\lambda},w_{\lambda}^{\ast})-\mathcal{K}(x,w^{\ast})\big](-v^{-}_{\lambda})dx\\&
=\int_{\partial_{bd}\Sigma_{\lambda}\cap\{w_{\lambda}^{\ast}\leq w^{\ast}\}}\big[\mathcal{K}(x,w^{\ast})-\mathcal{K}(x^{\lambda},w_{\lambda}^{\ast})\big]
(w^{\ast}-w_{\lambda}^{\ast})dx.
\end{split}
\end{equation}

$\bullet$ We claim that $\mathcal{K}(x,w_{\lambda}^{\ast})\leq\mathcal{K}(x^{\lambda},w_{\lambda}^{\ast})$ for $x\in\partial_{bd}\Sigma_{\lambda}$.

Indeed, observe that \eqref{f1-2} and $|x|\geq|x^{\lambda}|$ for $x\in\partial_{bd}\Sigma_{\lambda}$, so
\begin{equation}\label{kk-1}
\frac{f(|x|^{n-2s}w_{\lambda}^{\ast})}{|x|^{n+2s}}\leq\frac{f(|x^{\lambda}|^{n-2s}w_{\lambda}^{\ast})}{|x^{\lambda}|^{n+2s}}.
\end{equation}
Then, in order to conclude the proof of claim, it is sufficient to obtain
\begin{equation}\label{kk-2}
\int_{T(\Omega\times\{0\})}\frac{F(|y^{\prime}|^{n-2s}w_{\lambda}^{\ast})}{|y^{\prime}|^{2n}}\Big[\frac{1}{|x-y^{\prime}|^{\mu}}-\frac{1}{|x^{\lambda}-y^{\prime}|^{\mu}}\Big] dy^{\prime}\leq0.
\end{equation}
First, we divide the integral in the right-hand side of \eqref{kk-2} as
\begin{equation}\label{kk-3}
A_1+A_2+A_3:=\Big(\int_{\Sigma_{\lambda}}+\int_{T_{\lambda}(\Sigma_{\lambda})}+\int_{T(\Omega\times\{0\})\setminus (\Sigma_{\lambda}\cup T_{\lambda}(\Sigma_{\lambda}))}\Big)\frac{F(|y^{\prime}|^{n-2s}w_{\lambda}^{\ast})}{|y^{\prime}|^{2n}}\Big[\frac{1}{|x-y^{\prime}|^{\mu}}-\frac{1}{|x^{\lambda}-y^{\prime}|^{\mu}}\Big]dy^{\prime}.
\end{equation}
We will estimate each of them. Observe that for $y^{\prime}\in T(\Omega\times\{0\})\setminus (\Sigma_{\lambda}\cup T_{\lambda}(\Sigma_{\lambda}))$, we have $|x-y^{\prime}|\geq|x^{\lambda}-y^{\prime}|$ for $x\in\in\partial_{bd}\Sigma_{\lambda}$, so that $A_3\leq0$.
It remains to verify $A_1+A_2\leq0$.
If $y^{\prime}\in T_{\lambda}(\Sigma_{\lambda})$, then $(y^{\prime})^{\lambda}\in\Sigma_{\lambda}$, and by applying the translational and reflection symmetry, we have
$$
\frac{1}{|x-y^{\prime}|^{\mu}}=\frac{1}{|x^{\lambda}-(y^{\prime})^{\lambda}|^{\mu}}\hspace{2mm}\mbox{and}\hspace{2mm}\frac{1}{|x-(y^{\prime})^{\lambda}|^{\mu}}=\frac{1}{|x^{\lambda}-y^{\prime}|^{\mu}}.
$$
Thus, we deduce that
$$
A_2=-\int_{\Sigma_{\lambda}}\frac{F(|(y^{\prime})^{\lambda}|^{n-2s}w_{\lambda}^{\ast})}{|(y^{\prime})^{\lambda}|^{2n}}\Big[\frac{1}{|x-y^{\prime}|^{\mu}}-\frac{1}{|x^{\lambda}-y^{\prime}|^{\mu}}\Big]dy^{\prime}.
$$
By combining this with $A_1$, we find that
\begin{equation}\label{kk-7}
A_1+A_2=\int_{\Sigma_{\lambda}}\Big[\frac{F(|y^{\prime}|^{n-2s}w_{\lambda}^{\ast})}{|y^{\prime}|^{2n}}-\frac{F(|(y^{\prime})^{\lambda}|^{n-2s}w_{\lambda}^{\ast})}{|(y^{\prime})^{\lambda}|^{2n}}\Big]\Big[\frac{1}{|x-y^{\prime}|^{\mu}}-\frac{1}{|x^{\lambda}-y^{\prime}|^{\mu}}\Big]dy^{\prime}.
\end{equation}
In view of $|x-y^{\prime}|\leq|x^{\lambda}-y^{\prime}|$ for $y^{\prime}\in\Sigma_{\lambda}$, it is seen that
\begin{equation}\label{kk-4}
\frac{1}{|x-y^{\prime}|^{\mu}}-\frac{1}{|x^{\lambda}-y^{\prime}|^{\mu}}\geq0 \hspace{2mm}\mbox{on}\hspace{2mm}\Sigma_{\lambda}.
\end{equation}
Hence, in order to conclude the proofs of \eqref{kk-2}, we establish now the
following key estimate:
\begin{equation}\label{kk-5}
\frac{F(|y^{\prime}|^{n-2s}w_{\lambda}^{\ast})}{|y^{\prime}|^{2n}}\leq\frac{F(|(y^{\prime})^{\lambda}|^{n-2s}w_{\lambda}^{\ast})}{|(y^{\prime})^{\lambda}|^{2n}}.
\end{equation}
Since
$$|y^{\prime}|^2-|(y^{\prime})^{\lambda}|^2=(y_1^{\prime})^2-(2-2\lambda-y_{1}^{\prime})^2\geq0,$$
we have $|y^{\prime}|\geq|(y^{\prime})^{\lambda}|$. Then, it is enough to check that the function
\begin{equation}\label{kk-6}
\Theta(r):=\frac{F(r^{n-2s}w_{\lambda}^{\ast})}{r^{2n}}\hspace{2mm}\mbox{is non-increasing for}\hspace{2mm}r>0.
\end{equation}
Let us denote $u:=r^{n-2s}w_{\lambda}^{\ast}$. \eqref{kk-6} is equivalent to
$$
\mathcal{Y}(u):=2nF(u)-(n-2s)uf(u)\geq0.
$$
Observe that, due to \eqref{f1-2},
$$
(n+2s)f(u)-(n-2s)uf^{\prime}(u)\geq0.
$$
Thus differentiating $\mathcal{Y}$ with respect to $u$ gives
$$\mathcal{Y}^{\prime}(u)=n+2s)f(u)-(n-2s)uf^{\prime}(u)\geq0.$$
This leads to $\mathcal{Y}(u)\geq0$ for $u\geq0$. Combining these estimates with
\eqref{kk-7}, \eqref{kk-4} and \eqref{kk-5}, we conclude that $A_1+A_2\leq0$, and thus we complete the proof.

Combining the above claim with \eqref{bd}-\eqref{bdel} and $w^{\ast}\geq w_{\lambda}^{\ast}$, we deduce that
\begin{equation}\label{boun}
\begin{split}
\int_{\Sigma_{\lambda}}y^{1-2s}|\nabla v^{-}_{\lambda}|^2dxdy&\leq\int_{\partial_{bd}\Sigma_{\lambda}\cap\{w_{\lambda}^{\ast}\leq w^{\ast}\}}\big[\mathcal{K}(x,w^{\ast})-\mathcal{K}(x,w_{\lambda}^{\ast})\big](w^{\ast}-w_{\lambda}^{\ast})dx.\\&
\leq \int_{\partial_{bd}\Sigma_{\lambda}\cap\{w_{\lambda}^{\ast}\leq w^{\ast}\}}\bigg[
\Big(\int_{T(\Omega\times\{0\})}\frac{F(|y^{\prime}|^{n-2s}w^{\ast})}{|x-y^{\prime}|^{\mu}|y^{\prime}|^{2n}} dy^{\prime}\Big)\frac{f(|x|^{n-2s}w^{\ast})}{|x|^{n+2s}}\\&\hspace{3.5mm}
-\Big(\int_{T(\Omega\times\{0\})}\frac{F(|y^{\prime}|^{n-2s}w_{\lambda}^{\ast})}{|x-y^{\prime}|^{\mu}|y^{\prime}|^{2n}} dy^{\prime}\Big)\frac{f(|x|^{n-2s}w_{\lambda}^{\ast})}{|x|^{n+2s}}\bigg](w^{\ast}-w_{\lambda}^{\ast})dx
\\&
\leq \int_{\partial_{bd}\Sigma_{\lambda}\cap\{w_{\lambda}^{\ast}\leq w^{\ast}\}}\big[\mathcal{E}_1(x,w^{\ast},w_{\lambda}^{\ast})(v_{\lambda}^{-})^2+\mathcal{E}_2(x,w^{\ast},w_{\lambda}^{\ast})(w^{\ast}-w_{\lambda}^{\ast})\big]dx,
\end{split}
\end{equation}
where
\begin{equation*}
\mathcal{E}_1(x,w^{\ast},w_{\lambda}^{\ast})=\frac{1}{w^{\ast}-w_{\lambda}^{\ast}}
\Big(\int_{T(\Omega\times\{0\})}\frac{F(|y^{\prime}|^{n-2s}w^{\ast})}{|x-y^{\prime}|^{\mu}|y^{\prime}|^{2n}} dy^{\prime}\Big)\bigg[\frac{f(|x|^{n-2s}w^{\ast})-f(|x|^{n-2s}w_{\lambda}^{\ast})}{|x|^{n+2s}}\bigg],
\end{equation*}
\begin{equation*}
\mathcal{E}_2(x,w^{\ast},w_{\lambda}^{\ast})=\bigg[\int_{k(\Omega\times\{0\})}\frac{F(|y^{\prime}|^{n-2s}w^{\ast})-F(|y^{\prime}|^{n-2s}w_{\lambda}^{\ast})}{|x-y^{\prime}|^{\mu}|y^{\prime}|^{2n}} dy^{\prime}\bigg]\frac{f(|x|^{n-2s}w_{\lambda}^{\ast})}{|x|^{n+2s}}.
\end{equation*}

Owing to the appearance of Haree-type nonlinearity, $\mathcal{E}_{1}$ and $\mathcal{E}_{2}$ on the RHS of \eqref{boun} are a bit more difficult to estimate directly. Hence we split the proof into multiple cases.

$\bullet$ We now claim that $\mathcal{E}_1(x,w^{\ast},w_{\lambda}^{\ast})$ is bounded, that is, there exists a positive constant $C>0$ such that
\begin{equation}\label{Epuwx-1}
\mathcal{E}_1(x,w^{\ast},w_{\lambda}^{\ast})\leq C.
\end{equation}
Indeed,
since $f$ is locally Lipschitz continuous, it follows that
\begin{equation}\label{Lipschitz-1}
\frac{1}{(w^{\ast}-w_{\lambda}^{\ast})}\bigg[\frac{f(|x|^{n-2s}w^{\ast})-f(|x|^{n-2s}w_{\lambda}^{\ast})}{|x|^{n+2s}}\bigg]\leq C\sup_{\partial_{bd}\Sigma_{\lambda}}\{|w^{\ast}|+|w_{\lambda}^{\ast}|\}
\end{equation}
for $x\in\partial_{bd}\Sigma_{\lambda}$. For $\mathcal{E}_1(x,w^{\ast},w_{\lambda}^{\ast})$, we consider the remaining part of the function.
The assumptions \eqref{f1-0} and \eqref{f1-3} guarantee the existence of $C_{1},C_{2}>0$ such that
\begin{equation}\label{PQ-1}
F(u)\leq C_1 u^{\delta+1}+C_{2}u^{\sigma+1}.
\end{equation}
Hence
\begin{equation*}
\int_{T(\Omega\times\{0\})}\frac{F(|y^{\prime}|^{n-2s}w^{\ast})}{|x-y^{\prime}|^{\mu}|y^{\prime}|^{2n}} dy^{\prime}
\leq C_1\int_{T(\Omega\times\{0\})}\frac{|w^{\ast}|^{\delta+1}}{|x-y^{\prime}|^{\mu}|y^{\prime}|^{\mu}} dy^{\prime}+
C_1\int_{T(\Omega\times\{0\})}\frac{|w^{\ast}|^{\sigma+1}}{|x-y^{\prime}|^{\mu}|y^{\prime}|^{2n-(\sigma+1)(n-2s)}} dy^{\prime}.
\end{equation*}
To calculate the two integrals, we split its domain into
$$T(\Omega\times\{0\})=D_1\cup D_2\cup D_3,$$
where
$$D_1=B(0,1/4)\cap T(\Omega\times\{0\}),\hspace{2mm}D_2=B(x,1/4)\cap T(\Omega\times\{0\})\hspace{2mm}\mbox{and}\hspace{2mm}D_3=T(\Omega\times\{0\})\setminus(D_1\cup D_2).$$
By $x\in\partial_{bd}\Sigma_{\lambda}$ for $\lambda$ sufficiently small, we may assume that $\lambda<1/4$. Then we note that
\begin{equation*}
\begin{split}
\begin{cases}
|x-y^{\prime}|\geq|x|-|y^{\prime}|>\frac{1}{2} \quad\quad\quad\hspace{11.2mm}\mbox{ if } y^{\prime}\in D_1,\\
|y^{\prime}|\geq|x|-|x-y^{\prime}|>\frac{1}{2} \hspace{11.2mm}\quad\quad\quad\mbox{ if } y^{\prime}\in D_2,\\
|y^{\prime}|>\frac{1}{4}\hspace{2mm}\mbox{and}\hspace{2mm}|x-y^{\prime}|>\frac{1}{4}\quad\quad\hspace{4mm}\quad\quad \mbox{ if } y^{\prime}\in D_3.
\end{cases}
\end{split}
\end{equation*}
Using this, we can compute the integral over $D_1$ as
\begin{equation*}
\begin{split}
\int_{D_1}\frac{|w^{\ast}|^{\delta+1}}{|x-y^{\prime}|^{\mu}|y^{\prime}|^{\mu}} dy^{\prime}
\leq C \int_{B(0,\frac{1}{4})}\frac{1}{|y^{\prime}|^{\mu}}dy^{\prime}\leq C,
\end{split}
\end{equation*}
\begin{equation*}
\begin{split}
\int_{D_2}\frac{|w^{\ast}|^{\delta+1}}{|x-y^{\prime}|^{\mu}|y^{\prime}|^{\mu}} dy^{\prime}
\leq C \int_{B(x,\frac{1}{4})}\frac{1}{|x-y^{\prime}|^{\mu}}dy^{\prime}\leq C,
\end{split}
\end{equation*}
and
\begin{equation*}
\begin{split}
\int_{D_3}\frac{|w^{\ast}|^{\delta+1}}{|x-y^{\prime}|^{\mu}|y^{\prime}|^{\mu}} dy^{\prime}
\leq C.
\end{split}
\end{equation*}
Similarly, we estimate the second of RHS over $D_1$, $D_2$ and $D_3$ as
\begin{equation*}
\begin{split}
\Big(\int_{D_1}+\int_{D_2}+\int_{D_3}\Big)\frac{|w^{\ast}|^{\sigma+1}}{|x-y^{\prime}|^{\mu}|y^{\prime}|^{2n-(\sigma+1)(n-2s)}} dy^{\prime}
\leq C,
\end{split}
\end{equation*}
where $\sigma>2s/(n-2s)$. Consequently,
\begin{equation}\label{Lipschitz-5}
\int_{T(\Omega\times\{0\})}\frac{F(|y^{\prime}|^{n-2s}w^{\ast})}{|x-y^{\prime}|^{\mu}|y^{\prime}|^{2n}} dy^{\prime}\leq C.
\end{equation}
Estimate \eqref{Epuwx-1} now follows from \eqref{Lipschitz-1} and \eqref{Lipschitz-5}.

$\bullet$ We claim that
\begin{equation}\label{Lipschitz-6}
\begin{split}
&\int_{\partial_{bd}\Sigma_{\lambda}\cap\{w_{\lambda}^{\ast}\leq w^{\ast}\}}\mathcal{E}_2(x,w^{\ast},w_{\lambda}^{\ast})(w^{\ast}-w_{\lambda}^{\ast})dx
\\&\leq C\Big[\big|\partial_{bd}\Sigma_{\lambda}\cap\{w_{\lambda}^{\ast}\leq w^{\ast}\}\big|^{\frac{1}{q}-\frac{n-2s}{2n}}+
\big|\partial_{bd}\Sigma_{\lambda}\cap\{w_{\lambda}^{\ast}\leq w^{\ast}\}\big|^{\frac{1}{q_0}-\frac{n-2s}{2n}}\Big]\big(\int_{\Omega}|v_{\lambda}^{-}(\cdot,0)|^{2_{s}^{\sharp}}dx\big)^{\frac{2}{2_{s}^{\sharp}}},
\end{split}
\end{equation}
where
$$\frac{1}{q}=1-\frac{\mu}{n}-\varepsilon_0,\hspace{2mm}\mbox{and}\hspace{2mm}\frac{1}{q_0}=1-\frac{2n-(n-2s)(\sigma+1)}{n}-\varepsilon_0.$$
Indeed, due to \eqref{f1-0}, there exist $C_{1},C_2>0$ such that
\begin{equation}\label{delta}
f(u)\leq C_1u^{\delta}+C_{2}u^{\sigma}.
\end{equation}
Thus the mean value theorem and \eqref{delta} imply that
\begin{equation}\label{Lipschitz-7}
\begin{split}
F(|y^{\prime}|^{n-2s}w^{\ast})-F(|y^{\prime}|^{n-2s}w_{\lambda}^{\ast})&=f(\xi(y^{\prime}))(y^{\prime})^{n-2s}(w^{\ast}-w_{\lambda}^{\ast})
\leq C_1(\max\{w^{\ast},w^{\ast}_{\lambda}\})^{\delta}|y^{\prime}|^{(\delta+1)(n-2s)}(w^{\ast}-w_{\lambda}^{\ast})\\&\hspace{3.5mm}+C_2(\max\{w^{\ast},w^{\ast}_{\lambda}\})^{\sigma}|y^{\prime}|^{(\sigma+1)(n-2s)}(w^{\ast}-w_{\lambda}^{\ast}).
\end{split}
\end{equation}
And for $x\in\partial_{bd}\Sigma_{\lambda}$, we have $|x|>\rho>0$. Then
\begin{equation}\label{Lipschitz-8}
\frac{f(|x|^{n-2s}w_{\lambda}^{\ast})}{|x|^{n+2s}}\leq C(\sup_{\partial_{bd}\Sigma_{\lambda}}\{1+|w_{\lambda}^{\ast}|\}).
\end{equation}
Then \eqref{Lipschitz-7} and \eqref{Lipschitz-8} imply that
\begin{equation}\label{Lipschitz-9}
\begin{split}
&\int_{\partial_{bd}\Sigma_{\lambda}\cap\{w_{\lambda}^{\ast}\leq w^{\ast}\}}\mathcal{E}_2(x,w^{\ast},w_{\lambda}^{\ast})(w^{\ast}-w_{\lambda}^{\ast})dx\\&\leq
C\int_{\partial_{bd}\Sigma_{\lambda}}\int_{T(\Omega\times\{0\})}\frac{F(|y^{\prime}|^{n-2s}w^{\ast})-F(|y^{\prime}|^{n-2s}w_{\lambda}^{\ast})}{|x-y^{\prime}|^{\mu}|y^{\prime}|^{2n}}(w^{\ast}-w_{\lambda}^{\ast})dxdy^{\prime}\\&
\leq C\int_{\partial_{bd}\Sigma_{\lambda}\cap\{w_{\lambda}^{\ast}\leq w^{\ast}\}}\bigg[\Big(\int_{T(\Omega\times\{0\})}\frac{w^{\ast}-w_{\lambda}^{\ast}}{|x-y^{\prime}|^{\mu}|y^{\prime}|^{\mu}}dy^{\prime}\Big)+\Big(\int_{T(\Omega\times\{0\})}\frac{w^{\ast}-w_{\lambda}^{\ast}}{|x-y^{\prime}|^{\mu}|y^{\prime}|^{2n-(n-2s)(\sigma+1)}}dy^{\prime}\bigg](w^{\ast}-w_{\lambda}^{\ast})dx\\&
=B_1+B_2.
\end{split}
\end{equation}
Then, since for any $\varepsilon_0>$ small,
$$\frac{1}{p}=1-\frac{\mu}{n}+\varepsilon_0,\hspace{2mm}\frac{1}{q}=1-\frac{\mu}{n}-\varepsilon_0,$$
The weighted Hardy-Littlewood-Sobolev inequality (see \cite{stein-weiss}) gives that
\begin{equation}\label{Lipschitz-10}
\begin{split}
B_1\leq C\big\|w^{\ast}-w_{\lambda}^{\ast}\big\|_{L^{p}(T(\Omega\times\{0\}))}\big\|w^{\ast}-w_{\lambda}^{\ast}\big\|_{L^{q}(\partial_{bd}\Sigma_{\lambda}\cap\{w_{\lambda}^{\ast}\leq w^{\ast}\})}.
\end{split}
\end{equation}
Since $\mu<(n+2s)/2$, we have $\frac{1}{p}>\frac{n-2s}{2n}$ and $\frac{1}{q}>\frac{n-2s}{2n}$ for $\varepsilon_0$ small. Hence using H\"{o}lder inequality, we obtain
\begin{equation}\label{Lipschitz-11}
\big\|w^{\ast}-w_{\lambda}^{\ast}\big\|_{L^{p}(T(\Omega\times\{0\}))}\leq\big|\Omega\big|^{\frac{1}{p}-\frac{n-2s}{2n}}\big\|v_{\lambda}^{-}(\cdot,0)\big\|_{L^{\frac{2n}{n-2s}}(\Omega))},
\end{equation}
\begin{equation}\label{Lipschitz-12}
\big\|w^{\ast}-w_{\lambda}^{\ast}\big\|_{L^{q}(\partial_{bd}\Sigma_{\lambda}\cap\{w_{\lambda}^{\ast}\leq w^{\ast}\})}\leq\big|\partial_{bd}\Sigma_{\lambda}\cap\{w_{\lambda}^{\ast}\leq w^{\ast}\}\big|^{\frac{1}{q}-\frac{n-2s}{2n}}\big\|v_{\lambda}^{-}(\cdot,0)\big\|_{L^{\frac{2n}{n-2s}}(\Omega))}.
\end{equation}
Then \eqref{Lipschitz-10} and \eqref{Lipschitz-11}-\eqref{Lipschitz-12} imply that
\begin{equation}\label{Lipschitz-13}
B_1\leq C\big|\partial_{bd}\Sigma_{\lambda}\cap\{w_{\lambda}^{\ast}\leq w^{\ast}\}\big|^{\frac{1}{q}-\frac{n-2s}{2n}}\big|\Omega\big|^{\frac{1}{p}-\frac{n-2s}{2n}}\big(\int_{\Omega}|v_{\lambda}^{-}(\cdot,0)|^{2_{s}^{\sharp}}dx\big)^{\frac{2}{2_{s}^{\sharp}}}.
\end{equation}
Similarly, due to $\sigma>\max\{\frac{\mu}{n-2s},\frac{n+2s}{2(n-2s)}\}$, we can take
$$\frac{1}{p_0}=1-\frac{\mu}{n}+\varepsilon_0,\hspace{2mm}\frac{1}{q_0}=1-\frac{2n-(n-2s)(\sigma+1)}{n}-\varepsilon_0,$$
and so that $\frac{1}{p_0}>\frac{n-2s}{2n}$ and $\frac{1}{q_0}>\frac{n-2s}{2n}$ for $\varepsilon_0$ small.
Hence combining the weighted Hardy-Littlewood-Sobolev inequality, we have
\begin{equation*}
\begin{split}
B_2&\leq C\big\|w^{\ast}-w_{\lambda}^{\ast}\big\|_{L^{p_0}(T(\Omega\times\{0\}))}\big\|w^{\ast}-w_{\lambda}^{\ast}\big\|_{L^{q_0}(\partial_{bd}\Sigma_{\lambda}\cap\{w_{\lambda}^{\ast}\leq w^{\ast}\})}\\&
\leq C\big|\partial_{bd}\Sigma_{\lambda}\cap\{w_{\lambda}^{\ast}\leq w^{\ast}\}\big|^{\frac{1}{q_0}-\frac{n-2s}{2n}}\big|\Omega\big|^{\frac{1}{p_0}-\frac{n-2s}{2n}}\big(\int_{\Omega}|v_{\lambda}^{-}(\cdot,0)|^{2_{s}^{\sharp}}dx\big)^{\frac{2}{2_{s}^{\sharp}}}.
\end{split}
\end{equation*}
By combining this with \eqref{Lipschitz-9} and \eqref{Lipschitz-13}, we finally derive
\begin{equation*}
\begin{split}
\int_{\partial_{bd}\Sigma_{\lambda}\cap\{w_{\lambda}^{\ast}\leq w^{\ast}\}}\mathcal{E}_2(x,w^{\ast},w_{\lambda}^{\ast})(w^{\ast}-w_{\lambda}^{\ast})dx
&\leq C\Big[\big|\partial_{bd}\Sigma_{\lambda}\cap\{w_{\lambda}^{\ast}\leq w^{\ast}\}\big|^{\frac{1}{q}-\frac{n-2s}{2n}}\big|\Omega\big|^{\frac{1}{p}-\frac{n-2s}{2n}}\\&+
\big|\partial_{bd}\Sigma_{\lambda}\cap\{w_{\lambda}^{\ast}\leq w^{\ast}\}\big|^{\frac{1}{q_0}-\frac{n-2s}{2n}}\big|\Omega\big|^{\frac{1}{p_0}-\frac{n-2s}{2n}}\Big]\big(\int_{\Omega}|v_{\lambda}^{-}(\cdot,0)|^{2_{s}^{\sharp}}dx\big)^{\frac{2}{2_{s}^{\sharp}}},
\end{split}
\end{equation*}
as claimed.

As a consequence, combined with \eqref{boun}-\eqref{Epuwx-1}, \eqref{Lipschitz-6} and H\"{o}lder inequality, we conclude that
\begin{equation}\label{I245}
\begin{split}
&\int_{\Sigma_{\lambda}}y^{1-2s}|\nabla v_{\lambda}^{-}|^2dxdy\leq C_1\int_{\partial_{bd}\Sigma_{\lambda}\cap\{w_{\lambda}^{\ast}\leq w^{\ast}\}}(v_{\lambda}^{-})^2dx\\&
+C_1\Big[\big|\partial_{bd}\Sigma_{\lambda}\cap\{w_{\lambda}^{\ast}\leq w^{\ast}\}\big|^{\frac{1}{q}-\frac{n-2s}{2n}}+
\big|\partial_{bd}\Sigma_{\lambda}\cap\{w_{\lambda}^{\ast}\leq w^{\ast}\}\big|^{\frac{1}{q_0}-\frac{n-2s}{2n}}\Big]\big(\int_{\Omega}|v_{\lambda}^{-}(\cdot,0)|^{2_{s}^{\sharp}}dx\big)^{\frac{2}{2_{s}^{\sharp}}}\\&
\leq C_1\Big[\big|\partial_{bd}\Sigma_{\lambda}\cap\{w_{\lambda}^{\ast}\leq w^{\ast}\}\big|^{\frac{2s}{n}}+\big|\partial_{bd}\Sigma_{\lambda}\cap\{w_{\lambda}^{\ast}\leq w^{\ast}\}\big|^{\frac{1}{q}-\frac{n-2s}{2n}}+
\big|\partial_{bd}\Sigma_{\lambda}\cap\{w_{\lambda}^{\ast}\leq w^{\ast}\}\big|^{\frac{1}{q_0}-\frac{n-2s}{2n}}\Big]\big(\int_{\Omega}|v_{\lambda}^{-}|^{2_{s}^{\sharp}}dx\big)^{\frac{2}{2_{s}^{\sharp}}},
\end{split}
\end{equation}
where $C_1$ denotes different constants. Using this bound and the weighted Sobolev inequality, and taking $\lambda>0$ sufficiently small such that $\lambda<\epsilon_0:=C_1^{-1}$,
there holds
$$
\Big[\big|\partial_{bd}\Sigma_{\lambda}\cap\{w_{\lambda}^{\ast}\leq w^{\ast}\}\big|^{\frac{2s}{n}}+\big|\partial_{bd}\Sigma_{\lambda}\cap\{w_{\lambda}^{\ast}\leq w^{\ast}\}\big|^{\frac{1}{q}-\frac{n-2s}{2n}}+
\big|\partial_{bd}\Sigma_{\lambda}\cap\{w_{\lambda}^{\ast}\leq w^{\ast}\}\big|^{\frac{1}{q_0}-\frac{n-2s}{2n}}\Big]<\varepsilon_0.
$$
This leads to
\begin{equation*}
\begin{split}
\big(\int_{\Omega}|v_{\lambda}^{-}|^{2_{s}^{\sharp}}dx\big)^{\frac{2}{2_{s}^{\sharp}}}<\big(\int_{\Omega}|v_{\lambda}^{-}|^{2_{s}^{\sharp}}dx\big)^{\frac{2}{2_{s}^{\sharp}}}
\end{split}
\end{equation*}
for $\lambda>0$ sufficiently small.
This implies $v_{\lambda}^{-}\equiv0$ for $\lambda>0$ sufficiently small.
Then we are able to start off from this neighborhood of $z_1=1$, and move
the plane $T_{\lambda}$ along the $z_1$ direction to the left as long as the inequality \eqref{v} holds.

Step 2. We continuously move the plane this way up to its limiting position.
More precisely, we define
$$
\lambda_0=\sup\{\lambda>0: T_{\lambda}(\Sigma_{\lambda})\subset T(\mathcal{C})\},
$$
and
$$E:=\{\lambda\in(0,\frac{\lambda_0}{2}]:~v_{\lambda}\geq0\hspace{2mm}\mbox{for all}\hspace{2mm}x\in\Sigma_{\lambda}\}\cup\{0\}.$$
We next need to show that $E=[0,\frac{\lambda_0}{2}]$, that is $v_{\lambda}(x)\geq0$ for all $x\in\Sigma_{\lambda_0}$.
Since $E$ is closed and nonempty, it is enough to prove that $E$ is open in $[0,\frac{\lambda_0}{2}]$. The proof is essentially the same as that of Lemma 4.1 in \cite{Woocheol,QS}. For the reader's convenience, we include the details.
For any $\lambda_1\in E\cap [0,\frac{\lambda_0}{2}]$, we have $v_{\lambda_1}\geq0$ on $\Sigma_{\lambda_1}$. Since $v_{\lambda_1}>0$ on $T(\partial\Omega\times[0,\infty))\cap \Sigma_{\lambda_1}$ and satisfies the homogeneous weighted harmonic equation, the strong maximum principle (\cite[Section 4]{CS})
implies that $v_{\lambda_1}>0$ in $\Sigma_{\lambda_1}$. Thus there exists $c>0$ such that the set $\Sigma_{\lambda_1,c}=\{x\in \Sigma_{\lambda_1}: v_{\lambda_1}>c\}$ has measure at least $|\Sigma_{\lambda_1}|-\frac{\varrho}{2}$.
By continuity, for all $\lambda\in[\lambda_1,\lambda_1+\epsilon)$ with sufficiently small $\epsilon>0$, $v_{\lambda}>\frac{c}{2}$ on $\Sigma_{\lambda_1}$ and $|\Sigma_{\lambda}\setminus\Sigma_{\lambda_1}|<\frac{\varrho}{2}$.
This yields that the measure of $\{x\in \Sigma_{\lambda}: v_{\lambda}\leq0\}$ is less than $\varrho$.
By inequality \eqref{I245}, we conclude that $v_{\lambda}\geq0$ in $\Sigma_{\lambda}$ for all such $\lambda$, which proves that $E$ is open. Combined with the closedness and non-emptiness of $E$,
we obtain $[0,\frac{\lambda_0}{2}]$. Therefore, $w^{\ast}$ is increasing along any line starting from a boundary point in $\Omega$.

Choose any $x\in\partial\Omega$, there exist the constants $\alpha_0>0$ and $\gamma>0$ such that there exists an open connected set $\mathcal{O}_x\in S^{n-1}$ satisfying $|\mathcal{O}_x|>\gamma$ with
$$I_x=\{x+sw^{\ast}|0\leq s\leq\alpha_0,~w^{\ast}\in \mathcal{O}_x\}.$$
Then we guarantee a weaker version of the monotonicity
$$
u(x+s_1w^{\ast})\leq Cu(x+s_2w^{\ast}), \hspace{2mm}\forall0\leq s_1\leq s_2\leq\alpha_0\hspace{2mm}\mbox{and}\hspace{2mm}w^{\ast}\in \mathcal{O}_x.
$$
Combining this property and \eqref{boundary-1}, we get
\begin{equation*}
\begin{split}
\gamma u(x)\leq\frac{1}{|I_x|}\int_{I_x}u(y)dy\leq\frac{C}{\gamma}\int_{\mathcal{M}(\Omega,\alpha_0)}u(y)dy\leq C
\end{split}
\end{equation*}
for any $x\in\mathcal{Q}(\Omega,\alpha_0)$. This proves \eqref{boundary-2}.

\textbf{Case 2.} $\Omega$ is strictly convex. Then for each $3$-tuple $(x,s,v)$ such that $x+sv\in I_x$, let $P_x$ be a plane which passes through $x+sv$ and has $v$ as its normal vector. Assume that it divides $\Omega$ into two parts $\Omega_1$ and $\Omega_2$ where $x\in\partial\Omega_1$. Then the reflection of $\Omega_1$ with respect to the plane $P_x$ is contained in $\Omega_2$. We can take $T_{\lambda}$ and $\Sigma_{\lambda}$ replaced by $P_x$ and $\Omega_1$, and we have $T_{\lambda}(\Omega_1)\subset\Omega_2\subset\Omega$. Hence by applying the maximum principle, $v_{\lambda}=u(T_{\lambda}(y))-u(y)\geq0$ on $\Sigma_\lambda$. Therefore we deduce that a stronger monotonicity for $u$,
$$
u(x+s_1w^{\ast})\leq u(x+s_2w^{\ast}), \hspace{2mm}\forall0\leq s_1\leq s_2\leq\alpha_0\hspace{2mm}\mbox{and}\hspace{2mm}w^{\ast}\in \mathcal{O}_x.
$$
The rest of the proof is then analogous to the case $1$. We finally get a uniform bound of $u(x)$ on $\mathcal{Q}(\Omega,\alpha_0)$.
 Concluding the proof.
\end{proof}

\section{Decay estimate of rescaled solutions}\label{sec:sobolev}
In this section, we are devoted to prove that an adequate decay estimate for least energy solutions to problem \eqref{prondgr}.
For the simplicity of notations, we write $\|\cdot\|_{\infty}$ instead of the norm of $\|\cdot\|_{L^{\infty}(\Omega)}$
in the sequel.
Since $u_\epsilon$ become unbounded as $\epsilon\rightarrow0$ then we choose $x_\epsilon\in\Omega$ and the number $\mu_{\epsilon}>0$ is given by
\begin{equation}\label{miu}
\alpha_{n,s}\mu_{\epsilon}=\|u_{\epsilon}\|_{\infty}=u_\epsilon(x_\epsilon),\hspace{2mm}\mbox{where}\hspace{2mm}\alpha_{n,s}:=\alpha_{n,n-2s,s}.
\end{equation}
We define a family of rescaled functions
$$v_{\epsilon}(x)=\mu_{\epsilon}^{-1}u_{\epsilon}(\mu_{\epsilon}^{-\frac{2_{s}^{\sharp}-2-\epsilon}{2s}}x+x_{\epsilon})\quad\mbox{for}\quad x\in\Omega_{\epsilon}:=\mu_{\epsilon}^{\frac{4s-(n-2s)\epsilon}{2s(n-2s)}}(\Omega-x_{\epsilon}).$$
Then we have
\begin{equation}\label{rescaled}
(-\Delta)^s v_{\epsilon}(x)=\mu_{\epsilon}^{-\big(1+2s\big)}(-\Delta)^s u_{\epsilon}\big(\mu_{\epsilon}^{-\frac{2_{s}^{\sharp}-2-\epsilon}{2s}}x+x_{\epsilon}\big)=\big(|x|^{-{(n-2s)}}\ast v_\epsilon^{2_{s}^{\sharp}-1-\epsilon}\big)v_\epsilon^{2_{s}^{\sharp}-2-\epsilon}\hspace{3mm}\mbox{in}\hspace{2mm} x\in\Omega_{\epsilon}.
\end{equation}
Now, we consider the following change of variable:
$$
T:~\mathbb{R}^n\setminus\{0\}\rightarrow\mathbb{R}^n\setminus\{0\},\hspace{2mm}T(x)=\frac{x}{|x|^2}.$$
Let $\Omega_{\epsilon}^{\ast}:=T(\Omega_{\epsilon})\subset \mathbb{R}^n\setminus B(0,1/(c_{0}\mu_{\epsilon}^{\frac{4s-(n-2s)\epsilon}{2s(n-2s)}}))$. Then we have that $\Omega_{\epsilon}^{\ast}\subset\mathbb{R}^n\setminus B^{n}(0,c_0\lambda_{\epsilon})$ for some $c_0>0$, and let us consider the Kelvin transform of $V_{\epsilon}$ defined by
$$V_{\epsilon}(x)=\frac{1}{|x|^{n-2s}}v_{\epsilon}\big(\frac{x}{|x|^2}\big)\quad\mbox{in}\quad\Omega_{\epsilon}^{\ast}:=\big\{x:\frac{x}{|x|^{2}}\in\Omega_{\epsilon}\big\}$$
satisfies
\begin{equation}\label{Kelvin}
\begin{split}
(-\Delta)^{s} V_{\epsilon}&=\frac{1}{|x|^{n+2s}}(-\Delta)^{s}  v_{\epsilon}\big(\frac{x}{|x|^2}\big)\\&
=\frac{1}{|x|^{{\epsilon(n-2s)}}}\Big(\int_{\Omega_{\epsilon}^{\ast}}\frac{V_{\epsilon}^{2_{s}^{\sharp}-1-\epsilon}(y)}{|x-y|^{{(n-2s)}}|y|^{{\epsilon(n-2s)}}} dy\Big)V_{\epsilon}^{2_{s}^{\sharp}-2-\epsilon}\hspace{4mm}\mbox{in}\hspace{2mm} x\in\Omega_{\epsilon}^{\ast}.
\end{split}
\end{equation}

\begin{lem}\label{infinite}
Let $\mu_{\epsilon}>0$ be the number introduced in \eqref{miu}. If $u_\epsilon$ is a minimizing sequence to \eqref{minimi}, then
\begin{equation*}
\mu_{\epsilon}=\|u_{\epsilon}\|_{\infty}\rightarrow\infty \quad\mbox{as}\quad\epsilon\rightarrow0.
\end{equation*}
\end{lem}
\begin{proof}
 Assume that $\{u_{\epsilon}\}$ has a bounded subsequence. Then by Arzela-Ascoli Theorem, up to a sub-sequence, we obtain that for $\epsilon>0$ small enough, there exists a $\{u_{\epsilon}\}$ converging to some $u_{0}\not\equiv0$ uniformly on any compact set in $\mathbb{R}^n$. Recalling $w_{\epsilon}$ the s-harmonic extension of $u_\epsilon$ and we denote $W_0$ the s-harmonic extension of $u_{0}$. Then we have $\nabla w_{\epsilon}(x,y)\rightarrow\nabla W_0(x,y)$ for any $(x,y)\in\mathcal{C}$ as $\epsilon\rightarrow0$. Therefore,
 \begin{equation}\label{bound-1}
 \begin{split}
 \int_{\mathcal{C}}y^{1-2s}|\nabla W_0|^2dxdy&=\int_{\mathcal{C}}y^{1-2s}\liminf\limits_{\epsilon\rightarrow0}|\nabla w_\epsilon(x,y)|^2dxdy\leq\liminf\limits_{\epsilon\rightarrow0}\int_{\mathcal{C}}y^{1-2s}|\nabla w_\epsilon(x,y)|^2dxdy\\&
 =\liminf\limits_{\epsilon\rightarrow0}\kappa_{s}\int_{\Omega\times\{0\}}A_{\Omega}^su_{\epsilon}(x)u_{\epsilon}(x)dx
 \\&
 =\liminf\limits_{\epsilon\rightarrow0}\kappa_s\int_{\Omega}\big(|x|^{-(n-2s)}\ast u_{\epsilon}^{2_{s}^{\sharp}-1-\epsilon}\big)u_{\epsilon}^{2_{s}^{\sharp}-1-\epsilon}dx\\&
 =\kappa_s\int_{\Omega}\big(|x|^{-(n-2s)}\ast u_{0}^{2_{s}^{\sharp}-1}\big)u_{0}^{2_{s}^{\sharp}-1}dx.
 \end{split}
 \end{equation}
Moreover, by virtue of \eqref{minimi}, we know that
\begin{equation}\label{SHL}
\frac{\Big[\displaystyle\int_{\Omega}(|x|^{-(n-2s)} \ast|w_\epsilon(x,0)|^{2_{s}^{\sharp}-1-\epsilon})|w_\epsilon(x,0)|^{2_{s}^{\sharp}-1-\epsilon} dx\Big]^{\frac{1}{2_{s}^{\sharp}-1-\epsilon}}}{\displaystyle\int_{\mathcal{C}}y^{1-2s}|\nabla w_\epsilon(x,y)|^2dxdy}=\frac{1}{\kappa_{s}C_{HLS}}+o(1)
\end{equation}
as $\epsilon\rightarrow0$. Then combining \eqref{SHL} and the identity
\begin{equation*}
\int_{\mathcal{C}}y^{1-2s}|\nabla w_\epsilon(x,y)|^2dxdy
 =\kappa_s\int_{\Omega}\big(|x|^{-(n-2s)}\ast u_{\epsilon}^{2_{s}^{\sharp}-1-\epsilon}\big)u_{\epsilon}^{2_{s}^{\sharp}-1-\epsilon}dx,
\end{equation*}
we obtain
\begin{equation}\label{chls}
\lim\limits_{\epsilon\rightarrow0}\int_{\Omega}(|x|^{-(n-2s)} \ast|w_\epsilon(x,0)|^{2_{s}^{\sharp}-1-\epsilon})|w_\epsilon(x,0)|^{2_{s}^{\sharp}-1-\epsilon} dx=C_{HLS}^{\frac{n+2s}{4s}}.
\end{equation}
In view of \eqref{bound-1} and \eqref{chls}, we get that
\begin{equation*}
\int_{\mathcal{C}}y^{1-2s}|\nabla W_0|^2dxdy\leq \kappa_{s}C_{HLS}\Big[\int_{\Omega}\int_{\Omega} \frac{W_{0}^{2_{s}^{\sharp}-1}(x,0)W_{0}^{2_{s}^{\sharp}-1}(t,0)}{|x-t|^{n-2s}}dxdt\Big]^{\frac{1}{2_{s}^{\sharp}-1}}.
\end{equation*}
Thus, we find that $W_{0}$ achieves the best Sobolev constant in the nonlocal Sobolev-type trace inequality
\begin{equation}\label{trace}
\left(\int_{\mathbb{R}^n}\int_{\mathbb{R}^n} \frac{|w(x,0)|^{2_{s}^{\sharp}-1}|w(t,0)|^{2_{s}^{\sharp}-1}}{|x-t|^{n-2s}}dxdt\right)^{\frac{1}{2_{s}^{\sharp}-1}}
\leq C\int_{0}^{\infty}\int_{\mathbb{R}^n}y^{1-2s}|\nabla w(x,y)|^2dxdy.
\end{equation}
so that $W_{0}=CW[\xi,\lambda]$, for some $C,\lambda\in\mathbb{R}^{+}$ and $\xi\in\mathbb{R}^n$.
A contradicts to the well know fact that the nonlocal Sobolev-type trace inequality is never achieved on a bounded domain, see \cite{ple,MTS}.
\end{proof}

In virtue of \eqref{ele-100-1}, for a given least energy solution $u_{\epsilon}$ for problem \eqref{prondgr}, define its $s$-harmonic extension $w_{\epsilon}\in\mathcal{D}^{1,2}(\mathcal{C};y^{1-2s})$ as the solution of the elliptic equation
\begin{equation}\label{CF}
\left\lbrace
\begin{aligned}
&-\mbox{div}(y^{1-2s}\nabla w_{\epsilon})=0\hspace{4.14mm}\quad\quad \quad \quad \quad \quad \quad \quad  \mbox{in}\hspace{2mm} \mathcal{C},\\
&w_{\epsilon}=0\quad \quad \quad \quad \quad \quad \quad \hspace{12mm}\hspace{2mm}\hspace{8.9mm}\hspace{10mm} \mbox{on}\hspace{2mm}\partial_{L}\mathcal{C},\\
&w_{\epsilon}>0\quad \quad \quad \quad \quad \quad \quad \quad \hspace{10mm}\hspace{8.5mm}\hspace{10mm}\hspace{1mm}\mbox{in}\hspace{2mm}\mathcal{C},\\
&\partial_{\nu}^{s}w_{\epsilon}=\big(|x|^{-(n-2s)}\ast w_{\epsilon}^{2_{s}^{\sharp}-1-\epsilon}\big)w_{\epsilon}^{2_{s}^{\sharp}-2-\epsilon}\hspace{7.8mm} \mbox{in}\hspace{2mm}\Omega\times\{y=0\},
\end{aligned}
		\right.
\end{equation}
where $\mathcal{C}=\Omega\times(0,\infty)$ and $\partial_{L}\mathcal{C}=\partial\Omega\times(0,\infty)$.
We also define a family of rescaled functions
$$
\widetilde{W}_{\epsilon}(z):=\mu_{\epsilon}^{-1}w_{\epsilon}(\mu_{\epsilon}^{-\frac{2_{s}^{\sharp}-2-\epsilon}{2s}}z+x_{\epsilon})\quad\mbox{for}\quad z\in\mathcal{C}_{\epsilon}:=\mu_{\epsilon}^{\frac{4s-(n-2s)\epsilon}{2s(n-2s)}}(\mathcal{C}-x_{\epsilon}).
$$
Then $\widetilde{W}_{\epsilon}$ satisfies
\begin{equation}\label{CF101}
\left\lbrace
\begin{aligned}
&-\mbox{div}(y^{1-2s}\nabla \widetilde{W}_{\epsilon})=0\hspace{4.14mm}\quad\quad \quad \quad \quad \quad \quad \quad  \hspace{3mm}\mbox{in}\hspace{2mm} \mathcal{C}_{\epsilon},\\
&\widetilde{W}_{\epsilon}=0\quad \quad \quad \quad \quad \quad \quad \hspace{12mm}\hspace{2mm}\hspace{8.9mm}\hspace{10mm}\hspace{3mm} \mbox{on}\hspace{2mm}\partial_{L}\mathcal{C}_{\epsilon},\\
&\widetilde{W}_{\epsilon}>0\quad \quad \quad \quad \quad \quad \quad \quad \hspace{10mm}\hspace{8.5mm}\hspace{10mm}\hspace{3mm}\hspace{1mm}\mbox{in}\hspace{2mm}\mathcal{C}_{\epsilon},\\
&\partial_{\nu}^{s}\widetilde{W}_{\epsilon}=\big(|x|^{-(n-2s)}\ast \widetilde{W}_{\epsilon}^{2_{s}^{\sharp}-1-\epsilon}\big)\widetilde{W}_{\epsilon}^{2_{s}^{\sharp}-2-\epsilon}\hspace{7.8mm} \mbox{in}\hspace{2mm}\Omega_{\epsilon}\times\{y=0\},
\end{aligned}
		\right.
\end{equation}
Now, we consider the following change of variable:
$$
T:~\mathbb{R}_{+}^{n+1}\setminus\{(0,0)\}\rightarrow\mathbb{R}_{+}^{n+1}\setminus\{(0,0)\},\hspace{2mm}T(z)=\frac{z}{|z|^2}.$$
Let $\mathcal{C}_{\epsilon}^{\ast}:=T(\mathcal{C}_{\epsilon})$
and the Kelvin transform $\widetilde{W}_{\epsilon}^{\ast}$ of $\widetilde{W}_{\epsilon}$ defined by
\begin{equation}\label{WWTI}
\widetilde{W}_{\epsilon}^{\ast}(z)=\frac{1}{|z|^{n-2s}}\widetilde{W}_{\epsilon}\big(\frac{z}{|z|^2}\big),
\end{equation}
then
\begin{equation}\label{CF102}
\left\lbrace
\begin{aligned}
&-\mbox{div}(y^{1-2s}\nabla \widetilde{W}_{\epsilon}^{\ast})=0\hspace{4.14mm}\hspace{19.3mm}\quad\quad\quad\quad\quad \quad \quad \quad \quad \quad \quad  \quad\quad\quad\hspace{5mm}\mbox{in}\hspace{2mm} \mathcal{C}_{\epsilon}^{\ast},\\
&\widetilde{W}_{\epsilon}^{\ast}=0\hspace{19.3mm}\quad\quad\quad\quad \quad\quad\quad\quad \quad \quad \quad \quad \quad \hspace{12.5mm}\hspace{3.5mm}\hspace{8.9mm}\hspace{10mm}\hspace{3mm} \mbox{on}\hspace{2mm}\partial_{L}\mathcal{C}_{\epsilon}^{\ast},\\
&\widetilde{W}_{\epsilon}^{\ast}>0\hspace{19.3mm}\quad\quad\quad\quad\quad\quad\quad \quad \quad \quad \quad \quad \quad \quad \hspace{10mm}\hspace{10mm}\hspace{10.5mm}\hspace{3mm}\hspace{1mm}\mbox{in}\hspace{2mm}\mathcal{C}_{\epsilon}^{\ast},\\
&\partial_{\nu}^{s}\widetilde{W}_{\epsilon}^{\ast}=\frac{1}{|x|^{{\epsilon(n-2s)}}}\Big(\int_{T(\Omega_{\epsilon}\times\{0\})}\frac{(\widetilde{W}_{\epsilon}^{\ast})^{2_{s}^{\sharp}-1-\epsilon}(y^{\prime})}{|x-y^{\prime}|^{{(n-2s)}}|y^{\prime}|^{{\epsilon(n-2s)}}} dy^{\prime}\Big)(\widetilde{W}_{\epsilon}^{\ast})^{2_{s}^{\sharp}-2-\epsilon}\hspace{4mm} \mbox{in}\hspace{2mm}T(\Omega_{\epsilon}\times\{0\}),
\end{aligned}
		\right.
\end{equation}

Using \eqref{CF101}, we have the following result.
\begin{lem}\label{finite}
Up to the subsequence, the function $v_{\epsilon}$ converges to $W$ uniformly on any compact set as $\epsilon\rightarrow0$.
\end{lem}
\begin{proof}
Note that $v_{\epsilon}(0)=\alpha_{n,s}$ and $0\leq v_{\epsilon}\leq\alpha_{n,s}$.
Let $\overline{W}$ denote the weak limit of $\{\widetilde{W}_{\epsilon}\}_{\epsilon>0}$ in $H_{0,L}^s(\mathcal{C})$ as $\epsilon\rightarrow0$, and define its trace on the boundary as $\bar{v}:=\mbox{tr}_{\Omega\times\{0\}}\widetilde{W}_{\epsilon}$. Then $\bar{v}(0)=\max_{x\in\mathbb{R}^n}\bar{v}=\alpha_{n,s}$ and
$\overline{W}$ satisfies
\begin{equation*}
\left\lbrace
\begin{aligned}
&-\mbox{div}(y^{1-2s}\nabla \overline{W})=0\hspace{4.14mm}\quad\quad \quad \quad \quad \quad \quad \quad  \hspace{3mm}\mbox{in}\hspace{2mm} \mathbb{R}_{+}^{n+1},\\
&\overline{W}>0\quad \quad \quad \quad \quad \quad \quad \quad \hspace{10mm}\hspace{8.5mm}\hspace{10mm}\hspace{3mm}\hspace{1mm}\mbox{in}\hspace{2mm}\mathbb{R}_{+}^{n+1},\\
&\partial_{\nu}^{s}\overline{W}=\big(|x|^{-(n-2s)}\ast \overline{W}^{2_{s}^{\sharp}-1}\big)\overline{W}^{2_{s}^{\sharp}-2}\hspace{16mm} \mbox{in}\hspace{2mm}\mathbb{R}^{n}\times\{0\},
\end{aligned}
		\right.
\end{equation*}
Moreover, \cite{ple,MTS} identified that the extremal functions of the nonlocal Sobolev-type trace inequality are the bubbles $W(x)$ described in \eqref{defU}.
Hence $\overline{W}(x,y)=V_1(x,y)$, where we denote $V_1[\xi,\lambda]\in H^{s}(\mathbb{R}_{+}^{n+1})$ as the $s$-harmonic extension of $W[\xi,\lambda]$.
On the other hand, by \eqref{CF101}, we find that
$$v_{\epsilon}=\int_{\Omega}G(x,t)\big(|x|^{-(n-2s)}\ast v_{\epsilon}^{2_{s}^{\sharp}-1-\epsilon}\big)v_{\epsilon}^{2_{s}^{\sharp}-2-\epsilon}dt.$$
Then a direct computation yields that $\sup_{\epsilon>0}\sup_{x\in\Omega}|v_{\epsilon}|<\infty$. Thus, $\{v_{\epsilon}: \epsilon>0\}$ are equicontinuous on any compact set by \cite[Lemma 2.5]{CKL}.
Applying the Arzela-Ascoli theorem, we extract a subsequence converging uniformly on compact sets to some continuous function $v_0$. By construction,
$v_0$ coincides with the weak limit $W$. The proof is thus complete.
\end{proof}

We first state the following lemma which is useful in our analysis (see \cite{HANZCHAO,GT} for the proof).
\begin{lem}\label{regular1}
Suppose that $U$ is a bound solution of the equations
 \begin{equation*}
\left\lbrace
\begin{aligned}
&-\mbox{div}(y^{1-2s}\nabla U)=0\quad\quad\quad\quad\hspace{4.8mm} \mbox{in}\hspace{2mm} \mathcal{C}_{\ast},\\
&U>0\quad\quad\quad\quad\hspace{12.6mm}\hspace{8mm}\hspace{9mm} \mbox{on}\hspace{2mm}\mathcal{C}_{\ast},
\\
&U=0\quad\quad\quad\quad\hspace{12mm}\hspace{8.6mm}\hspace{9mm} \mbox{on}\hspace{2mm}\partial_{L}\mathcal{C}_{\ast},\\
&\partial_{\nu}^{s}U(x,0)=g(x)U(x,0)\hspace{5mm}\hspace{10mm} \mbox{on}\hspace{2mm}T(\Omega_{\ast}\times\{y=0\}).
\end{aligned}
		\right.
\end{equation*}
Then there exist a constant $r>0$ and fixed $\beta\in(1,\infty)$ such that if
$$
\big\|g\big\|_{L^{\frac{n}{2s}}(T(\Omega_{\epsilon}\times\{0\}))\cap B_n(0,2r)}\leq C_1\hspace{2mm}\mbox{and}\hspace{2mm}\int_{B_{n+1}(0,2r)}y^{1-2s}U^{\beta+1}(x,y)dxdy\leq C_2
$$
then
\begin{align}\label{cU1}
\int_{B_n(0,\frac{r}{2})}U^{\frac{2_{s}^{\sharp}(\beta+1)}{2}}(x,0)\leq C_3.
\end{align}
with $C_3>0$ depending on $\beta$, $r$, $C_1$ and $C_2$. Furthermore, if there exist the constants $D_1$ and $D_2>0$, fixed $\alpha\in(1,\infty)$ such that
$$
\int_{B_{n+1}(0,r)}y^{1-2s}U(x,y)^{\alpha+1}dxdy+\int_{B_{n}(0,r)}U(x,0)^{\alpha+1}dx\leq D_1,
$$
and
$$
\int_{T(\Omega_{\ast}\times\{y=0\})\cap B_{n}(0,r)}|g(x)|^qdx\leq D_2$$
for some $r>0$ and $q>\frac{n}{2s}$, then
\begin{align}\label{cU1}
\big\|U(x,0)\big\|_{L^{\infty}(B_n(0,\frac{r}{2}))}\leq D_3,
\end{align}
where $D_3>0$ depends only on $\alpha$, $r$, $D_1$ and $D_2$.
\end{lem}
\begin{proof}
See the proof in \cite{CKL}.
\end{proof}

\begin{lem}\label{exists}
There exists a constant $C>0$, such that
$$
|x|^{-\epsilon(n-2s)}\leq C\hspace{2mm}\mbox{for all}\hspace{2mm}\epsilon>0\hspace{2mm}\mbox{and}\hspace{2mm}x\in T(\Omega_{\epsilon}).
$$
\end{lem}
\begin{proof}
In virtue of Lemma \ref{finite} and Fatou's lemma, we have
\begin{equation*}
\begin{split}
&\int_{B_n(0,1)}\int_{B_n(0,1)}\frac{v_\epsilon^{2_{s}^{\sharp}-1-\epsilon}(x)v_\epsilon^{2_{s}^{\sharp}-1-\epsilon}(y)}{|x-y|^{n-2s}}dxdy
\\&\geq\int_{B_n(0,1)}\Big[\int_{B_n(0,1)}\frac{1}{|x-y|^{n-2s}}\big(\frac{1}{1+|y|^2}\big)^{\frac{(n+2s)}{2}}dy\Big]\big(\frac{1}{1+|x|^2}\big)^{\frac{(n+2s)}{2}}dx\geq C.
\end{split}
\end{equation*}
On the other hand, we have
$$
\int_{B_n(0,1)}\int_{B_n(0,1)}\frac{v_\epsilon^{2_{s}^{\sharp}-1-\epsilon}(x)v_\epsilon^{2_{s}^{\sharp}-1-\epsilon}(y)}{|x-y|^{n-2s}}dxdy
\leq\mu_{\epsilon}^{-\frac{n-2s}{2s}}\int_{\Omega}\int_{\Omega}\frac{u_\epsilon^{2_{s}^{\sharp}-1-\epsilon}(y)u_\epsilon^{2_{s}^{\sharp}-1-\epsilon}(x)}{|x-y|^{n-2s}} dxdy\leq C\mu_{\epsilon}^{-\frac{n-2s}{2s}}.
$$
 Moreover, from $\Omega_{\epsilon}^{\ast}=T(\Omega_{\epsilon})\subset \mathbb{R}^n\setminus B(0,1/(c_{0}\mu_{\epsilon}^{\frac{4s-(n-2s)\epsilon}{2s(n-2s)}}))$, we have that $|x|\geq C\mu_{\epsilon}^{-\frac{2_{s}^{\sharp}-2-\epsilon}{2s}}$ for $x\in\Omega_{\epsilon}^{\ast}$. Combining all this together, we get that
$$
|x|^{-\epsilon(n-2s)}\leq\mu_{\epsilon}^{\frac{(2_{s}^{\sharp}-2-\epsilon)(n-2s)\epsilon}{2s}}\leq C \hspace{3mm}\mbox{for all}\hspace{3mm}x\in T(\Omega_{\epsilon}),
$$
as desired.
\end{proof}
Using the Moser iteration technique combined with integral estimates, we establish the following result, which will be crucial for the subsequent analysis.
\begin{lem}\label{cWU}
There exists a positive constant $c$, independent of $\varepsilon$, such that
\begin{align}\label{cU}
v_\varepsilon(x)\leq cW(x),
\end{align}
where $W(x)=W[0,1](x)$. By rescaling, it can be shown that it is equivalent to
\begin{align}\label{00cU}
u_\varepsilon(x)\leq cW[\mu_{\epsilon}^{\frac{2_{s}^{\sharp}-2-\epsilon}{2s}},x_{\epsilon}](x).
\end{align}
\end{lem}
\begin{proof}
It is easy to see that \eqref{cU} is equivalent to
\begin{equation}\label{cU1}
V_{\epsilon}(x)\leq B,\hspace{2mm}\mbox{for}\hspace{2mm}x\in T(\Omega_{\epsilon})
\end{equation}
for some $B$. In order to obtain the result, it is enough to prove that there exist $C>0$ and a radius $r>0$, such that for any $\epsilon>0$,
\begin{equation}\label{cU1}
\sup\limits_{T(\Omega_{\epsilon})\cap B(0,r)}V_{\epsilon}(x)\leq C.
\end{equation}
Recalling \eqref{CF102}, letting $g(x)$ be given by
\begin{equation*}
g(x):=\frac{1}{|x|^{{\epsilon(n-2s)}}}\Big(\int_{T(\Omega_{\epsilon}\times\{0\})}\frac{(\widetilde{W}_{\epsilon}^{\ast})^{2_{s}^{\sharp}-1-\epsilon}(y^{\prime})}{|x-y^{\prime}|^{{(n-2s)}}|y^{\prime}|^{{\epsilon(n-2s)}}} dy^{\prime}\Big)(\widetilde{W}_{\epsilon}^{\ast})^{2_{s}^{\sharp}-3-\epsilon}:=b(x)(\widetilde{W}_{\epsilon}^{\ast})^{2_{s}^{\sharp}-3-\epsilon}.
\end{equation*}
We make use of Lemma \ref{exists}, H\"{o}lder inequality, Hardy-Littlewood-Sobolev inequaity and the following identity
\begin{equation}\label{VEPUSILONG}
V_{\epsilon}^{2_{s}^{\sharp}-1-\epsilon}=V_{\epsilon}^{2_{s}^{\sharp}-1}-\epsilon W^{2_{s}^{\sharp}-1}\log(V_{\epsilon})+O(\epsilon^2V_{\epsilon}^{2_{s}^{\sharp}-1}),
\end{equation}
$$V_{\epsilon}^{2_{s}^{\sharp}-3-\epsilon}=V_{\epsilon}^{2_{s}^{\sharp}-3}-\epsilon V_{\epsilon}^{2_{s}^{\sharp}-3}\log(V_{\epsilon})+O(\epsilon^2V_{\epsilon}^{2_{s}^{\sharp}-3})$$
to obtain
\begin{equation}\label{q1pie-1}
\begin{split}
\big\|g\big\|_{L^{\frac{n}{2s}}(T(\Omega_{\epsilon}\times\{0\}))\cap B_n(0,2r)}&\leq C\Big\|\Big(\int_{T(\Omega_{\epsilon}\times\{0\})}\frac{(\widetilde{W}_{\epsilon}^{\ast})^{2_{s}^{\sharp}-1-\epsilon}(y^{\prime})}{|x-y^{\prime}|^{{(n-2s)}}} dy^{\prime}\Big)(\widetilde{W}_{\epsilon}^{\ast})^{2_{s}^{\sharp}-3-\epsilon}\Big\|_{L^{\frac{n}{2s}}(T(\Omega_{\epsilon}\times\{0\}))\cap B_n(0,2r)}\\&
\leq C \Big\|\Big(\int_{T(\Omega_{\epsilon})}\frac{V_{\epsilon}^{2_{s}^{\sharp}-1}(y^{\prime})}{|x-y^{\prime}|^{{(n-2s)}}} dy^{\prime}\Big)V_{\epsilon}^{2_{s}^{\sharp}-3}\Big\|_{L^{\frac{n}{2s}}(T(\Omega_{\epsilon}))}
\\&
\leq C\Big\|\frac{1}{|x|^{n-2s}}\ast V_{\epsilon}^{2_{s}^{\sharp}-1}\Big\|_{L^{\frac{n}{2s}\cdot\theta}(T(\Omega_{\epsilon}))}
\Big(\int_{T(\Omega_{\epsilon})}V_{\epsilon}^{2_{s}^{\sharp}-1}dx\Big)^{\frac{2s}{n\gamma}}\\&
\leq C\Big(\int_{T(\Omega_{\epsilon})}V_{\epsilon}^{2_{s}^{\sharp}-1}dx\Big)^{\frac{n+2s}{2n}+\frac{2s}{n\gamma}},
\end{split}
\end{equation}
where $\theta=\frac{4s}{n-2s}$ and $\gamma=\frac{4s}{6s-n}$.
Recalling the definition of $V_1[\xi,\lambda]\in H^{s}(\mathbb{R}_{+}^{n+1})$ in Lemma \ref{finite}. Then, by Lemma \ref{finite}, we obtain
\begin{equation}\label{1-q1pie-1}
\lim\limits_{\epsilon\rightarrow0}\int_{\mathcal{C}_{\epsilon}}y^{1-2s}\big|\nabla(\widetilde{W}_{\epsilon}-V_1)\big|^2dxdy=0.
\end{equation}
Combining the Sobolev trace embedding theorem, we get
\begin{equation}\label{BRO-00}
\lim\limits_{\epsilon\rightarrow0}\int_{\Omega_{\epsilon}}\big|v_{\epsilon}(x)-W(x)\big|^{2_{s}^{\sharp}}dx=0.
\end{equation}
It follows from \eqref{q1pie-1} and \eqref{BRO-00} that
\begin{equation}\label{BRO-000}
\big\|g\big\|_{L^{\frac{n}{2s}}(T(\Omega_{\epsilon}\times\{0\}))\cap B_n(0,2r)}\leq C \Big(\int_{\mathbb{R}^n}W^{2_{s}^{\sharp}}dx\Big)^{\frac{n+2s}{2n}+\frac{2s}{n\gamma}}\leq C.
\end{equation}

Let $\eta\in C_{0}^{\infty}(B_{N+1}(0,1))$ be a smooth cut-off function such that
\begin{equation*}
\begin{split}
&\quad \quad \eta\big|_{B_{N+1}(0,\frac{1}{2})}=1 \hspace{3mm}\mbox{and}\hspace{3mm} 0\leq\eta\leq1, \hspace{3mm}\mbox{and}\big|\nabla\eta\big|\leq C(N).
\end{split}
\end{equation*}
Taking $\phi=\eta^2 |\widetilde{W}_{\epsilon}^{\ast}(x,y)-V_1(x,y)|^{2(\beta-1)}(\widetilde{W}_{\epsilon}^{\ast}(x,y)-V_1(x,y))$,
then due to \mbox{div}$(y^{1-2s}|\widetilde{W}_{\epsilon}^{\ast}(x,y)-V_1(x,y)|)=0$, we get
\begin{equation*}
\begin{split}
(2\beta-1)&\int_{B_{n+1}(0,1)}y^{1-2s}|\widetilde{W}_{\epsilon}^{\ast}-V_1|^{2(\beta-1)}|\nabla(\widetilde{W}_{\epsilon}^{\ast}-V_1)|^2\eta^2dxdy\\&+2\int_{B_{n+1}(0,1)}y^{1-2s}|\widetilde{W}_{\epsilon}^{\ast}-V_1|^{2\beta-1}\nabla(\widetilde{W}_{\epsilon}^{\ast}-V_1)\eta\nabla\eta dxdy=0.
\end{split}
\end{equation*}
With the help of the Young's inequality, we see that
\begin{equation*}
\begin{split}
&2\int_{B_{n+1}(0,1)}y^{1-2s}|\widetilde{W}_{\epsilon}^{\ast}-V_1|^{2\beta-1}\nabla(\widetilde{W}_{\epsilon}^{\ast}-V_1)\eta\nabla\eta dz\\&\leq\frac{2\beta-1}{2}\int_{B_{n+1}(0,1)}y^{1-2s}|\widetilde{W}_{\epsilon}^{\ast}-V_1|^{2(\beta-1)}|\nabla(\widetilde{W}_{\epsilon}^{\ast}-V_1)|^2\eta^2dz+\frac{2}{2\beta-1}
\int_{B_{n+1}(0,1)}y^{1-2s}|\widetilde{W}_{\epsilon}^{\ast}-V_1|^{2\beta}|\nabla\eta|^2dz.
\end{split}
\end{equation*}
Combining this information, we deduce that
\begin{equation}\label{WWW-0}
\begin{split}
\int_{B_{n+1}(0,1)}y^{1-2s}|\widetilde{W}_{\epsilon}^{\ast}-V_1|^{2(\beta-1)}|\nabla(\widetilde{W}_{\epsilon}^{\ast}-V_1)|^2\eta^2dz\leq C(\beta)
\int_{B_{n+1}(0,1)}y^{1-2s}|\widetilde{W}_{\epsilon}^{\ast}-V_1|^{2\beta}|\nabla\eta|^2dz.
\end{split}
\end{equation}
Let us denote $\xi_{\beta}=|\widetilde{W}_{\epsilon}^{\ast}-V_1|^{\beta}\eta$. Then we estimate
\begin{equation}\label{WWW-1}
|\xi_{\beta}|\leq 2\beta^{2}|\widetilde{W}_{\epsilon}^{\ast}-V_1|^{2(\beta-1)}|\nabla(\widetilde{W}_{\epsilon}^{\ast}-V_1)|^2\eta^2+2|\widetilde{W}_{\epsilon}^{\ast}-V_1|^{2\beta}|\nabla\eta|^2.
\end{equation}
According to the Sobolev inequality on weighted spaces (see \cite[Theorem 1.3]{})
\begin{equation}\label{WWW-2}
\Big(\int_{\Omega}y^{1-2s}|w(x,y)|^{\frac{2(n+1)}{n}}dxdy\Big)^{\frac{n}{2(n+1)}}\leq C\Big(\int_{\Omega}y^{1-2s}|\nabla w(x,y)|^{2}dxdy\Big)^{\frac{1}{2}},
\end{equation}
for any function $w$ and \eqref{WWW-0}-\eqref{WWW-1}, with the definition of cut-off function $\eta$, a direct computation yields
\begin{equation}\label{bata}
\Big(\int_{B_{n+1}(0,\frac{1}{2})}y^{1-2s}|\widetilde{W}_{\epsilon}^{\ast}-V_1|^{\frac{2(n+1)}{n}\beta}dz\Big)^{\frac{n}{2(n+1)\beta}}\leq C(C(\beta))^{\frac{1}{2\beta}}
\Big(\int_{B_{n+1}(0,1)}y^{1-2s}|\widetilde{W}_{\epsilon}^{\ast}-V_1|^{2\beta}dz\Big)^{\frac{1}{2\beta}}
\end{equation}
where $C$ changing from line to line, but independent of $\beta$. Thus, we can
take $\beta_1=\frac{n+1}{n}$. From now on, we can follow exactly the iteration argument. That is, we define $\beta_{m+1}$ for $m\geq1$, so that
$$2\beta_{m+1}=\frac{2(n+1)}{n}\beta_m.$$
Hence
$$
\beta_{m+1}=(\frac{n+1}{n})^{m}\beta_1,
$$
and replacing in equality \eqref{bata}, we find that
 \begin{equation*}
\Big(\int_{B_{n+1}(0,\frac{1}{2})}y^{1-2s}|\widetilde{W}_{\epsilon}^{\ast}-V_1|^{2\beta_{m+1}}dz\Big)^{\frac{1}{2\beta_{m+1}}}\leq C(C(\beta_m))^{\frac{1}{2\beta_m}}
\Big(\int_{B_{n+1}(0,1)}y^{1-2s}|\widetilde{W}_{\epsilon}^{\ast}-V_1|^{2\beta_{m}}dz\Big)^{\frac{1}{2\beta_{m}}}
\end{equation*}
By defining
$$A_m=\Big(\int_{B_{n+1}(0,1)}y^{1-2s}|\widetilde{W}_{\epsilon}^{\ast}-V_1|^{2\beta_{m}}dz\Big)^{\frac{1}{2\beta_{m}}},$$
we obtain that there exists a constant $C_1>0$ independently of $m$, such that
\begin{equation}\label{cO-1}
A_{m+1}\leq\prod_{k=1}^{m}(CC(\beta_k))^{\frac{1}{2\beta_k}}A_1\leq C_1A_1.
\end{equation}
Moreover, by \eqref{1-q1pie-1}, we find easily that
\begin{equation}\label{cO}
\lim\limits_{\epsilon\rightarrow0}\int_{T(\Omega_{\epsilon})}y^{1-2s}\big|\nabla(\widetilde{W}_{\epsilon}^{\ast}-V_1)(x,y)\big|^2dxdy=0,
\end{equation}
where we make use of the identity $W(x)=|x|^{-(n-2s)}W(T(x))$ and the following inequality
\begin{equation*}
\int_{\mathbb{R}^{n+1}}y^{1-2s}\big|\nabla H(z)\big|^2dz\geq C\int_{\mathbb{R}^{n+1}}y^{1-2s}\big|\nabla [|z|^{-(n-2s)}H(T(z))]\big|^2dz
\end{equation*}
for any functions $H:\mathbb{R}^{n+1}\rightarrow\mathbb{R}$.
Consequently, by the Sobolev inequality \eqref{WWW-2} on weighted spaces and \eqref{cO-1}-\eqref{cO}, for $m\geq1$ sufficiently large such that $2<\beta_0+1\leq2\beta_{m+1}$, we derive
\begin{equation*}
\lim\limits_{\epsilon\rightarrow0}\int_{B_{n+1}(0,1)}y^{1-2s}\big|\widetilde{W}_{\epsilon}^{\ast}-V_1)(x,y)\big|^{\beta_0+1}dxdy=0.
\end{equation*}
Thus, we apply Lemma \ref{regular1} to obtain
\begin{equation}\label{cO}
\int_{T(\Omega_{\epsilon}\times\{0\})\cap B_{n}(0,\rho)}V_{\epsilon}^{\frac{(\beta_0+1)2_{s}^{\sharp}}{2}}dx\leq C
\end{equation}
for some $\rho>0$ and $\beta_0>1$.

We next claim that $\|b(x)\|_{L^\infty(T(\Omega_{\epsilon}))}\leq C_0$ where $C_0$ independent of $\epsilon>0$ provided $\epsilon$ is sufficiently small.
We argue by some elementary inequalities and the estimates of integration, similarly to \cite{Cingolani}. Therefore, we will just sketch it.

Observe that
$$|x|^{\epsilon(n-2s)}b(x)=\int_{T(\Omega_{\epsilon}\times\{0\})}\frac{(\widetilde{W}_{\epsilon}^{\ast})^{2_{s}^{\sharp}-1-\epsilon}(y^{\prime})}{|x-y^{\prime}|^{{(n-2s)}}|y^{\prime}|^{{\epsilon(n-2s)}}} dy^{\prime}.$$
Then we have, for any $r>0$,
$x\in B(0,r)$, by H\"{o}lder inequality
\begin{equation*}
\begin{split}
\int_{B(0,\frac{r}{2})}\frac{|\widetilde{W}_{\epsilon}^{\ast}(y)|^{2_{s}^{\sharp}-1-\epsilon}}{|x-y|^{{n-2s}}|y|^{{\epsilon(n-2s)}}} dy&\leq\int_{B(0,\frac{r}{2})}\frac{|\widetilde{W}_{\epsilon}^{\ast}(y)|^{2_{s}^{\sharp}-1-\epsilon}}{|y|^{{(\epsilon+1)(n-2s)}}} dy+\int_{B(x,\frac{3}{2}r)}\frac{|\widetilde{W}_{\epsilon}^{\ast}(y)|^{2_{s}^{\sharp}-1-\epsilon}}{|x-y|^{{(\epsilon+1)(n-2s)}}} dy<\infty.
\end{split}
\end{equation*}
For any $x\in T(\Omega_{\epsilon})\setminus B(0,r)$, it holds that
\begin{equation*}
\begin{split}
\int_{B(0,\frac{r}{2})}\frac{|\widetilde{W}_{\epsilon}^{\ast}(y)|^{2_{s}^{\sharp}-1-\epsilon}}{|x-y|^{{n-2s}}|y|^{{\epsilon(n-2s)}}} dy&\leq\int_{B(0,\frac{r}{2})}\frac{|\widetilde{W}_{\epsilon}^{\ast}(y)|^{2_{s}^{\sharp}-1-\epsilon}}{|y|^{{(\epsilon+1)(n-2s)}}} dy<\infty,
\end{split}
\end{equation*}
On the other hand, we know that
\begin{equation*}
\begin{split}
&\int_{T(\Omega_{\epsilon})\setminus B(0,\frac{r}{2})}\frac{|\widetilde{W}_{\epsilon}^{\ast}(y)|^{2_{s}^{\sharp}-1-\epsilon}}{|x-y|^{{n-2s}}|y|^{{\epsilon(n-2s)}}} dy\leq\int_{\mathbb{R}^n\setminus B(0,\frac{r}{2})}\frac{|\widetilde{W}_{\epsilon}^{\ast}(y)|^{2_{s}^{\sharp}-1-\epsilon}}{|x-y|^{{n-2s}}|y|^{{\epsilon(n-2s)}}} dy\\&
=\int_{(\mathbb{R}^n\setminus B(0,\frac{r}{2}))\cap B_{\frac{r}{2}}(x)}\frac{|\widetilde{W}_{\epsilon}^{\ast}(y)|^{2_{s}^{\sharp}-1-\epsilon}}{|x-y|^{{n-2s}}|y|^{{\epsilon(n-2s)}}} dy+\int_{(\mathbb{R}^n\setminus B(0,\frac{r}{2}))\cap( \mathbb{R}^n\setminus B_{\frac{r}{2}}(x))}\frac{|\widetilde{W}_{\epsilon}^{\ast}(y)|^{2_{s}^{\sharp}-1-\epsilon}}{|x-y|^{{n-2s}}|y|^{{\epsilon(n-2s)}}} dy.
\end{split}
\end{equation*}
A direct computation shows that
\begin{equation}\label{q-1}
\begin{split}
\int_{(\mathbb{R}^n\setminus B(0,\frac{r}{2}))\cap B_{\frac{r}{2}}(x)}\frac{|\widetilde{W}_{\epsilon}^{\ast}(y)|^{2_{s}^{\sharp}-1-\epsilon}}{|x-y|^{{n-2s}}|y|^{{\epsilon(n-2s)}}} dy&\leq\frac{c}{r^{\epsilon(n-2s)}}\int_{(\mathbb{R}^n\setminus B(0,\frac{r}{2}))\cap B_{\frac{r}{2}}(x)}\frac{|\widetilde{W}_{\epsilon}^{\ast}(y)|^{2_{s}^{\sharp}-1-\epsilon}}{|x-y|^{{n-2s}}}dy<\infty.
\end{split}
\end{equation}
Furthermore, H\"{o}lder inequality gives that
\begin{equation}\label{q1pie}
\begin{split}
\int_{(\mathbb{R}^n\setminus B(0,\frac{r}{2}))\cap( \mathbb{R}^n\setminus B_{\frac{r}{2}}(x))}&\frac{|\widetilde{W}_{\epsilon}^{\ast}(y)|^{2_{s}^{\sharp}-1-\epsilon}}{|x-y|^{n-2s}|y|^{{\epsilon(n-2s)}}} dy\\&\leq\frac{c}{r^{n-2s}}\big\|\widetilde{W}_{\epsilon}^{\ast}\big\|^{2_{s}^{\sharp}-1-\epsilon}_{L^{\frac{(2_{s}^{\sharp}-1-\epsilon)q^{\prime}}{q^{\prime}-1}}(\mathbb{R}^n\setminus B(0,\frac{r}{2}))}\Big\|\frac{1}{|y|^{{\epsilon(n-2s)}}}\Big\|_{L^{q^{\prime}}(\mathbb{R}^n\setminus B(0,\frac{r}{2}))}<\infty,
\end{split}
\end{equation}
where $q^{\prime}>\frac{n}{\epsilon(n-2s)}$. Combining these estimates, we conclude that
$$|x|^{\epsilon(n-2s)}b(x)\in L^\infty(T(\Omega_{\epsilon})),$$
where the bound independent of $\epsilon>0$ provided $\epsilon$ is sufficiently small. Hence, due to this bound and Lemma \ref{exists}, the claim follow and where $c$ independent of $\epsilon>0$. Hence, combining this with \eqref{cO}, there exists $q=\frac{(\beta_0+1)n}{6s-n}>\frac{n}{2s}$ such that for small $\epsilon>0$, there holds that
\begin{equation*}
\begin{split}
\big\|g(x)\big\|_{L^{q}(T(\Omega_{\epsilon}\times\{0\})\cap B_{n}(0,\rho))}&\leq C\big\|V_{\epsilon}^{2_{s}^{\sharp}-3}(x)\big\|_{L^{q}(T(\Omega_{\epsilon}\times\{0\})\cap B_{n}(0,\rho))}\\&=C\big\|V_{\epsilon}(x)\big\|_{L^{\frac{(\beta_0+1)2_{s}^{\sharp}}{2}}(T(\Omega_{\epsilon}\times\{0\})\cap B_{n}(0,\rho))}\leq C.
\end{split}
\end{equation*}
Combining this estimate and Lemma \ref{regular1}, we finally derive
$$
\big\|V_{\epsilon}\big\|_{L^{\infty}(B_n(0,\frac{\rho}{2}))}\leq C,
$$
as desired.
\end{proof}

\section{Proof of Theorems \ref{prondgr} and \ref{consequence}}\label{consequence00}
The whole section is dedicated to the proof of Theorems \ref{prondgr} and \ref{consequence}, which is established by the following some key Lemmas with different right hand sides.

\subsection{Estimates for $u_\varepsilon$ near the blow-up point}
This subsection is devoted to the derivation of the next result.
\begin{thm}\label{blow-up}
Let $\Omega$ be a smooth bounded domain in $\mathbb{R}^n$.
Suppose that $\{u_\epsilon\}_{\epsilon>0}$ be a family of least energy solutions to \eqref{prondgr}. Then for any $r>0$ there exists a number $C(r,\Omega)>0$ such that
\begin{equation}\label{boundary-1-1}
\int_{\mathcal{M}(\Omega,r)}udx\leq C(r,\Omega).
\end{equation}
Moreover, there is a constant $C(r,\Omega)>0$ such that
\begin{equation}\label{boundary-2-1}
\sup\limits_{x\in\mathcal{Q}(\Omega,r)}u(x)\leq C(r,\Omega).
\end{equation}
\end{thm}
\begin{proof}
As the statement of in Remark \ref{Rem1-2}, we note that $f(u_{\epsilon})=u_{\epsilon}^{2_{s}^{\sharp}-2-\epsilon}$ with $2_{s}^{\sharp}=2n/(n-2s)$ for small $\epsilon>0$ satisfies the assumptions $(f_1)$ and $(f_2)$.
Hence the derivation of \eqref{boundary-1-1} and \eqref{boundary-2-1}, see the proof of Theorem \ref{prior-1}.
\end{proof}

\subsection{$\mu_{\epsilon}^{\epsilon}$ estimate}
In this subsection, we shall use a
estimate of $\epsilon$ we have derived to prove that the limit of $\mu_{\epsilon}^{\epsilon}$ tends to $1$ as $\epsilon\rightarrow0$. To this end, our purpose in what follows is to prove the following results.
\begin{lem}\label{thm:uniquenessofweaksolution}
 There exists a constant $M>0$ such that for $\epsilon>0$ small, there holds
\begin{equation}\label{eq:energyestimateu}
\epsilon\leq M \mu_{\epsilon}^{-2-\frac{[4s-(n-2s)\epsilon]\epsilon}{s}}.
\end{equation}
\end{lem}
\begin{Rem}
In the Laplacian operator problem $-\triangle u_{\epsilon}=(|x|^{-(n-2)}\ast u_{\epsilon}^{p-\epsilon})u_{\epsilon}^{p-1-\epsilon}$ in $\Omega$, $u=0$ on $\partial\Omega$, with $p=(n+2)/(n-2)$, given the estimate of rescaled solution, one can use the regularity theory (see Lemma 2 in \cite{HANZCHAO} or A.2 in \cite{Cingolani}) and the Pohozaev identity
to get the estimate of $\epsilon$. However, this approach is not easily applied to our problem \eqref{prondgr} because the Pohozaev identity is given
on the domain $\Omega$
\begin{equation*}
\begin{split}
\frac{1}{2k_{s}}\int_{\partial_{L}\mathcal{C}}y^{1-2s}\big|\nabla w_{\epsilon}(z)\big|^2\langle\nu,x\rangle dS_x=\frac{(n-2s)^2}{n+2s-\epsilon(n-2s)}\cdot\epsilon\cdot\int_{\Omega}\int_{\Omega} \frac{w_{\epsilon}^{2_{s}^{\sharp}-1-\epsilon}(z,0)w_{\epsilon}^{2_{s}^{\sharp}-1-\epsilon}(x,0)}{|x-z|^{n-2s}}dzdx,
\end{split}
\end{equation*}
where $\nu=\nu(x)$ denotes the unit outward normal to the boundary $\partial_{L}\mathcal{C}$.
There is a notable difference in the regularity of the left-hand side solution between this case (we refer to \cite{LS-2016,CDS} for $s\in(0,\frac{1}{2}]$ and $s\in(\frac{1}{2},1)$) and the case with $s=1$, which renders the methods adopted in \cite{HANZCHAO} and \cite{Cingolani} inapplicable. For this reason,
the integral estimates from the following Proposition \ref{prosition-identity} are essential to the preceding proof for Lemma \ref{thm:uniquenessofweaksolution}.
\end{Rem}

We now recall the Green's function and the Robin function for the fractional Laplacian $A_{s}$.
Note that Green's function $G_{\mathcal{C}}=G_{\mathcal{C}}(z,x)(z\in\mathcal{C},x\in\Omega)$ of the problem
\begin{equation}\label{CFLL-2}
\left\lbrace
\begin{aligned}
&-\mbox{div}(y^{1-2s}\nabla w)=0\hspace{4.14mm}\quad \quad \quad  \mbox{in}\hspace{2mm} \mathcal{C},\\
&w=0\quad\quad \hspace{12mm}\hspace{2mm}\hspace{8.9mm}\hspace{10mm} \mbox{on}\hspace{2mm}\partial_{L}\mathcal{C},\\
&\partial_{\nu}^{s}w=h\quad\quad\quad\quad\quad\quad\quad\hspace{9.5mm} \mbox{in}\hspace{2mm}\Omega\times\{y=0\},
\end{aligned}
		\right.
\end{equation}
for some function $f$ on $\Omega\times\{y=0\}$, then we can see that $w$ has the expression
$$
w(z)=\int_{\Omega}G_{\mathcal{C}}(z,t)h(t)dt=\int_{\Omega}G_{\mathcal{C}}(z,t)(-\Delta )^su(t)dt, \hspace{2mm}\mbox{where}\hspace{2mm}z\in\mathcal{C}\hspace{2mm}\mbox{and}\hspace{2mm}u=tr|_{\Omega\times\{t=0\}}w.
$$
We have the following formula (see \cite{CKL})
\begin{equation}\label{1-Green}
G_{\mathcal{C}}\big((x,y),t\big)=G_{\mathbb{R}^{n+1}_{+}}\big((x,y),t\big)-H_{\mathcal{C}}\big((x,y),t\big),
\end{equation}
where
$$G_{\mathbb{R}^{n+1}_{+}}\big((x,y),t\big):=\frac{\alpha_{n,s}}{|(x-t,y)|^{n-2s}},\hspace{2mm}\alpha_{n,s}=\frac{1}{|S^{n-1}|}\cdot\frac{2^{1-2s}\Gamma(\frac{n-2s}{2})}{\Gamma(\frac{n}{2})\Gamma(s)},$$
and the regular part $H_{\mathcal{C}}:\mathcal{C}\rightarrow\mathbb{R}^n$ satisfies
\begin{equation*}
\left\lbrace
\begin{aligned}
&-\mbox{div}\big(y^{1-2s}\nabla_{(x,y)} H_{\mathcal{C}}\big((x,y),t\big)\big)=0\hspace{4.14mm}\quad \quad \quad  \mbox{in}\hspace{2mm} \mathcal{C},\\
&H_{\mathcal{C}}\big((x,y),t\big)=\frac{\alpha_{n,s}}{|(x-t,y)|^{n-2s}}\quad \hspace{6mm}\hspace{2.5mm}\hspace{6mm}\hspace{6mm} \mbox{on}\hspace{2mm}\partial_{L}\mathcal{C},\\
&\partial_{\nu}^{s}H_{\mathcal{C}}\big((x,y),t\big)=0\quad\quad\quad\quad\quad\quad\quad\quad\hspace{4mm}\hspace{9mm} \mbox{in}\hspace{2mm}\Omega\times\{y=0\},
\end{aligned}
		\right.
\end{equation*}
Then we can denote the Robin function $\phi(x):=H_{\mathcal{C}}\big((x,0),t\big)$.

We need the following integral inequality which plays a crucial role in the proof of Lemma \ref{thm:uniquenessofweaksolution}.
\begin{Prop}\label{prosition-identity}
Assume that $w\in H_{0,L}^{s}(\mathcal{C})$ is a solution of problem
\begin{equation*}
\left\lbrace
\begin{aligned}
&-\mbox{div}(y^{1-2s}\nabla w)=0\hspace{6.14mm}\hspace{10mm} \quad\quad \mbox{in}\hspace{2mm} \mathcal{C},\\
&w>0\hspace{10mm}\quad\hspace{10.5mm}\hspace{10mm} \quad\hspace{10mm} \mbox{on}\hspace{2mm}\mathcal{C},
\\
&w=0\hspace{12mm}\quad\quad\hspace{9.5mm}\hspace{10mm} \hspace{9mm} \mbox{on}\hspace{2mm}\partial_{L}\mathcal{C},
\\& \partial_{\nu}^{s}w=\big(|x|^{-\mu}\ast w^{p}\big)w^{p-1}\hspace{8mm}\hspace{10mm} \mbox{on}\hspace{2mm}\Omega\times\{y=0\}
\end{aligned}
		\right.
\end{equation*}
Then, for each $\rho>0$ and $q>\frac{n}{s}$, there is a constant $C=C(\rho,q)>0$ such that
\begin{equation}\label{idengity0}
\begin{split}
&\min\limits_{r\in[\rho,2\rho]}\Big|\Big(\frac{n}{p}-\frac{n-2s}{2}\Big)\int_{\mathcal{M}(\Omega,r/2)}\big(\int_{\Omega} \frac{w^{p}(t,0)}{|x-t|^{\mu}}dt\big)w^{p}(x,0)dx\Big|\\
\leq&C\bigg[\Big(\int_{\mathcal{Q}(\Omega,2\rho)}\Big|\big(\int_{\Omega}\frac{w^p(t,0)}{|x-t|^{\mu}}dt\big)w^{p-1}(x,0)\Big|^qdx\Big)^{\frac{2}{q}}+\Big(\int_{\mathcal{M}(\Omega,\rho/2)}\Big|\big(\int_{\Omega}\frac{w^p(t,0)}{|x-t|^{\mu}}dt\big)w^{p-1}(x,0)\Big|dx\Big)^{2}\\&
+\int_{\mathcal{Q}(\Omega,2\rho)}\Big|\big(\int_{\Omega}\frac{w^p(t,0)}{|x-t|^{\mu}}dt\big)w^{p}(x,0)\Big|dx+\int_{\mathcal{M}(\Omega,r/2)}\Big(\int_{\Omega}x(x-t)\frac{w^p(t,0)}{|x-t|^{\mu+2}} dt\Big)w^p(x,0)dx\bigg].
\end{split}
\end{equation}
\end{Prop}
\begin{proof}
Let us define the sets
$$
\mathcal{D}_r=\big\{z\in\mathbb{R}_{+}^{N+1}:~dist\big(z,\mathcal{M}(\Omega,r)\times\{0\}\big)\leq\frac{r}{2}\big\},
$$
$$
\mathcal{D}_r^{+}=\partial\mathcal{D}_r\cap\big\{(x,y)\in\mathbb{R}^{N+1}:~y>0\big\},\hspace{2mm}\mbox{and}\hspace{2mm}\mathcal{E}_{\rho}=\bigcup_{r=\rho}^{2\rho}\partial\mathcal{D}_r^{+},
$$
and it is noticing that $\partial\mathcal{D}_r=\partial\mathcal{D}_r^{+}\cup(\mathcal{M}(\Omega,r/2)\times\{0\})$.
Moreover, we have the Pohozaev identity
\begin{equation}\label{local-po}
\mbox{div}\Big\{y^{1-2s}\langle z, \nabla w\rangle\nabla w-y^{1-2s}\frac{|\nabla w|^2}{2}z\Big\}+\frac{n-2s}{2}y^{1-2s}|\nabla w|^2=0.
\end{equation}
Integrating by parts, we get
\begin{equation}\label{pohozaev2}
\begin{split}
&\int_{\partial\mathcal{D}_r^{+}}y^{1-2s}\Big\langle\langle z, \nabla w\rangle\nabla w-\frac{|\nabla w|^2}{2}z,\nu\Big\rangle dS_x+\kappa_s\int_{\mathcal{M}(\Omega,r/2)\times\{0\}}\langle x, \nabla w\rangle\partial_{\nu}^{s}wdx\\&=-\frac{N-2s}{2}\int_{\mathcal{D}_r}y^{1-2s}|\nabla w|^2dxdy.
\end{split}
\end{equation}
Furthermore, by applying $\partial_{\nu}^{s}w=(|x|^{-\mu}\ast w^p)w^{p-1}$ and performing a further integration by parts, we obtain
\begin{equation*}
\begin{split}
\int_{\mathcal{M}(\Omega,r/2)\times\{0\}}\langle x, \nabla w\rangle\partial_{\nu}^{s}wdx&=\int_{\mathcal{M}(\Omega,r/2)}\langle x, \nabla w\rangle \Big(\int_{\Omega}\frac{w^p(t,0)}{|x-t|^{\mu}}dt\Big)w^{p-1}(x,0)dx\\&
=-\int_{\mathcal{M}(\Omega,r/2)}w(x,0)\nabla\Big(x \int_{\Omega}\frac{w^p(t,0)}{|x-t|^{\mu}}dtw^{p-1}(x,0)\Big)dx\\&
\hspace{3.5mm}+\int_{\partial\mathcal{M}(\Omega,r/2)} \int_{\Omega}\frac{w^p(t,0)}{|x-t|^{\mu}}dtw^{p}(x,0)\langle x,\nu\rangle dS_x.
\end{split}
\end{equation*}
A direct calculation yields
\begin{equation}\label{pohozaev3}
\begin{split}
\int_{\mathcal{M}(\Omega,r/2)\times\{0\}}\langle x, \nabla w\rangle\partial_{\nu}^{s}wdx&
=-\frac{n}{p}\int_{\mathcal{M}(\Omega,r/2)} \int_{\Omega}\frac{w^p(t,0)}{|x-t|^{\mu}}dtw^{p}(x,0)dx\\&
\hspace{3.5mm}+\frac{\mu}{p}\int_{\mathcal{M}(\Omega,r/2)} \int_{\Omega}x(x-t)\frac{w^p(t,0)}{|x-t|^{\mu+2}}dtw^{p}(x,0)dx\\&
\hspace{3.5mm}+\frac{1}{p}\int_{\partial\mathcal{M}(\Omega,r/2)} \int_{\Omega}\frac{w^p(t,0)}{|x-t|^{\mu}}dtw^{p}(x,0)\langle x,\nu\rangle dS_x.
\end{split}
\end{equation}
Moreover, the integration by parts arguments give
\begin{equation}\label{pohozaev4}
\int_{\mathcal{D}_r}y^{1-2s}|\nabla w|^2dxdy=\kappa_s\int_{\mathcal{M}(\Omega,r/2)} \int_{\Omega}\frac{w^p(t,0)}{|x-t|^{\mu}}dtw^{p}(x,0)dx+\int_{\partial\mathcal{D}_r^{+}}y^{1-2s}w \frac{\partial w}{\partial\nu}dS_x.
\end{equation}
Summarizing, by plugging \eqref{pohozaev3} and \eqref{pohozaev4} in \eqref{pohozaev2}, we get that
\begin{equation}\label{Qx}
\begin{split}
&\kappa_{s}\Big(\frac{n}{p}-\frac{n-2s}{2}\Big)\int_{\mathcal{M}(\Omega,r/2)} \int_{\Omega}\frac{w^p(t,0)}{|x-t|^{\mu}}dtw^{p}(x,0)dx\\&
=\int_{\partial\mathcal{D}_r^{+}}y^{1-2s}\Big\langle\langle z, \nabla w\rangle\nabla w-\frac{|\nabla w|^2}{2}z,\nu\Big\rangle dS_x+\frac{N-2s}{2}\int_{\partial\mathcal{D}_r^{+}}y^{1-2s}w \frac{\partial w}{\partial\nu}dS_x\\&
\hspace{3.5mm}+\frac{\mu}{p}\int_{\mathcal{M}(\Omega,r/2)} \int_{\Omega}x(x-t)\frac{w^p(t,0)}{|x-t|^{\mu+2}}dtw^{p}(x,0)dx
+\frac{1}{p}\int_{\partial\mathcal{M}(\Omega,r/2)} \int_{\Omega}\frac{w^p(t,0)}{|x-t|^{\mu}}dtw^{p}(x,0)\langle x,\nu\rangle dS_x
\end{split}
\end{equation}
Integrating the above identity with respect to $r$ over the interval $[\rho,2\rho]$, and using Poincar\'{e} inequality, an easy calculation shows that
\begin{equation}\label{A1A22}
\begin{split}
&\min\Big|\kappa_{s}\Big(\frac{n}{p}-\frac{n-2s}{2}\Big)\int_{\mathcal{M}(\Omega,r/2)} \int_{\Omega}\frac{w^p(t,0)}{|x-t|^{\mu}}dtw^{p}(x,0)dx\Big|\\&
\leq C\int_{\mathcal{E}_{\rho}}y^{1-2s}|\nabla w|^2dz
+C\int_{\mathcal{M}(\Omega,r/2)} \int_{\Omega}x(x-t)\frac{w^p(t,0)}{|x-t|^{\mu+2}}dtw^{p}(x,0)dx\\&
\hspace{3.5mm}+C\int_{\mathcal{Q}(\Omega,\rho)} \int_{\Omega}\frac{w^p(t,0)}{|x-t|^{\mu}}dtw^{p}(x,0)dx.
\end{split}
\end{equation}
In order to estimate the right-hand side, observe that there holds
$$
\big|\nabla w\big|^2=\bigg|\int_{\Omega\times\{0\}}\nabla_z\frac{\alpha_{n,s}}{|(x-t,y)|^{n-2s}}\big(|x|^{-\mu}\ast w^{p}\big)w^{p-1}dt-\int_{\Omega\times\{0\}}\nabla_zH_{\mathcal{C}}\big((x,y),t\big)\big(|x|^{-\mu}\ast w^{p}\big)w^{p-1}dt\bigg|^2.
$$

We next proceed similarly to \cite{CKL}. From \cite{CKL}, we find that
$$
\sup\limits_{t\in\Omega}\int_{\mathcal{E}_{\rho}}y^{1-2s}|\nabla_z H_{\mathcal{C}}\big((x,y),t\big)|^2dz\leq C.
$$
Thus, combining H\"{o}lder's inequality gives that
\begin{equation}\label{pohozaev5}
\begin{split}
&\int_{\mathcal{E}_{\rho}}y^{1-2s}\Big|\int_{\Omega\times\{0\}}\nabla_zH_{\mathcal{C}}\big((x,y),t\big)\big(|x|^{-\mu}\ast w^{p}\big)w^{p-1}dt\Big|^2dz\\&
\leq C\Big(\int_{\mathcal{Q}(\Omega,2\rho)}\Big|\big(\int_{\Omega}\frac{w^p(t,0)}{|x-t|^{\mu}}dt\big)w^{p-1}(x,0)\Big|^qdx\Big)^{\frac{2}{q}}+\Big(\int_{\mathcal{M}(\Omega,\rho/2)}\Big|\big(\int_{\Omega}\frac{w^p(t,0)}{|x-t|^{\mu}}dt\big)w^{p-1}(x,0)\Big|dx\Big)^{2}.
\end{split}
\end{equation}
We split the integral as
$$
\int_{\Omega\times\{0\}}\nabla_z\frac{\alpha_{n,s}}{|(x-t,y)|^{n-2s}}\big(|x|^{-\mu}\ast w^{p}\big)w^{p-1}dt=\Big(\int_{\mathcal{Q}(\Omega,2\rho)\times\{0\}}+\int_{\mathcal{M}(\Omega,2\rho)\times\{0\}}\Big)\nabla_z\frac{\alpha_{n,s}\big(|x|^{-\mu}\ast w^{p}\big)}{|(x-t,y)|^{n-2s}}w^{p-1}dt.
$$
By H\"{o}lder inequality, we obtain
\begin{equation}\label{A1A21}
\begin{split}
&\int_{\mathcal{Q}(\Omega,2\rho)\times\{0\}}\nabla_z\frac{\alpha_{n,s}\big(|x|^{-\mu}\ast w^{p}\big)}{|(x-t,y)|^{n-2s}}w^{p-1}dt\\&\leq\Big\|\nabla_z\frac{\alpha_{n,s}}{|(x-t,y)|^{n-2s}}\Big\|_{L^{q^{\prime}}(\mathcal{Q}(\Omega,2\rho))}
\Big\|\big(\int_{\Omega}\frac{w^{p}(t,0)}{|x-t|^{\mu}}dt\big)w^{p-1}(x,0)\Big\|_{L^{q}(\mathcal{Q}(\Omega,2\rho))},
\end{split}
\end{equation}
where $q^{\prime}$ is the H\"{o}lder conjugate of $q$. Moreover, using the following key estimate (see \cite{CKL}):
$$
\sup\limits_{z\in\mathcal{E}_{\rho},~t\in\mathcal{M}(\Omega,2\rho)}\Big|\frac{\alpha_{n,s}}{|(x-t,y)|^{n-2s}}\Big|\leq C,
$$
we have
$$
\int_{\mathcal{M}(\Omega,2\rho)\times\{0\}}\nabla_z\frac{\alpha_{n,s}\big(|x|^{-\mu}\ast w^{p}\big)}{|(x-t,y)|^{n-2s}}w^{p-1}dt\leq C\int_{\mathcal{M}(\Omega,2\rho)\times\{0\}}\big|\big(|x|^{-\mu}\ast w^{p}\big)w^{p-1}\big|dt.
$$
Therefore, this bound together with the estimate of $\nabla_{z}|(x-t,y)|^{-(n-2s)}$ in $L^{q^{\prime}}$-norm and \eqref{A1A21}, one can easily compute
\begin{equation*}
\begin{split}
\int_{\mathcal{E}_{\rho}}y^{1-2s}|\nabla w|^2dz\leq& C\Big\|\big(\int_{\Omega}\frac{w^{p}(t,0)}{|x-t|^{\mu}}dt\big)w^{p-1}(x,0)\Big\|^2_{L^{q}(\mathcal{Q}(\Omega,2\rho))}
+C\Big\|\int_{\Omega}\frac{w^{p}(t,0)}{|x-t|^{\mu}}w^{p-1}(x,0)\Big\|^2_{L^1(\mathcal{M}(\Omega,2\rho))}.
\end{split}
\end{equation*}
Combining this bound with \eqref{A1A22}, we conclude that \eqref{idengity0}, as desired.
\end{proof}

Now we are in a position to prove Lemma \ref{thm:uniquenessofweaksolution}.
\begin{proof}[Proof of Lemma \ref{thm:uniquenessofweaksolution}]
We take $p=2_{s}^{\sharp}-1-\epsilon$ and $\mu=n-2s$ in \eqref{idengity0},
and let $L_\epsilon$ and $R_\epsilon$
be the left-hand and right-hand sides of the result \eqref{idengity0}, respectively, so that $L_\epsilon=R_\epsilon$. In the following, we shall estimate each term in the right-hand side of $L_\epsilon$ and $R_\epsilon$.

\textbf{Estimate of $L_{\epsilon}$.}
If $|x|\leq R$ and $|y|\leq R$ for some $R>0$, then $|x-y|\leq2R$, $1+|y|^2\leq1+R^2$ and $1+|x|^2\leq1+R^2$. As a result
\begin{equation*}
\begin{split}
&\Big|\int_{\mathbb{R}^n}\Big[\int_{\mathbb{R}^n}\frac{1}{|x-y|^{n-2s}}\big(\frac{1}{1+|y|^2}\big)^{\frac{(n-2s)(2_{s}^{\sharp}-1-\epsilon)}{2}}dy\Big]\big(\frac{1}{1+|x|^2}\big)^{\frac{(n-2s)(2_{s}^{\sharp}-1-\epsilon)}{2}}dx\Big|\\&
\geq\int_{B_{R}(0)}\int_{B_{R}(0)}\frac{1}{|x-y|^{n-2s}}\big(\frac{1}{1+|y|^2}\big)^{\frac{n+2s}{2}}\big(\frac{1}{1+|x|^2}\big)^{\frac{n+2s}{2}}dxdy\geq\int_{B_{R}(0)}\int_{B_{R}(0)}\frac{dxdy}{(2R)^{n-2s}(1+R^2)^{n+2s}}\geq C.
\end{split}
\end{equation*}
Then,
a direct computation shows that
\begin{equation}\label{Left}
\begin{split}
&\min\limits_{r\in[\rho,2\rho]}\Big|\Big(\frac{n}{2_{s}^{\sharp}-1-\epsilon}-\frac{n-2s}{2}\Big)\int_{\mathcal{M}(\Omega,\frac{r}{2})\times\{0\}}\big(\int_{\Omega} \frac{w_\epsilon^{2_{s}^{\sharp}-1-\epsilon}}{|x-t|^{n-2s}}dt\big)w_\epsilon^{2_{s}^{\sharp}-1-\epsilon}dx\Big|\\
&=\mu_{\epsilon}^{\frac{(n-2s)\epsilon}{2s}}\Big|\frac{(\epsilon+1)(n-2s)^2}{2[n+2s-\epsilon(n-2s)]}\int_{\mu_{\epsilon}^{\frac{2_{s}^{\sharp}-2-\epsilon}{2s}}\big(\mathcal{M}(\Omega,r)-x_{\epsilon}\big)}\Big(\int_{\Omega}\frac{v_\epsilon^{2_{s}^{\sharp}-1-\epsilon}}{|x-y|^{n-2s}}dy\Big)v_\epsilon^{2_{s}^{\sharp}-1-\epsilon}dx\Big|\\&
\geq\epsilon\mu_{\epsilon}^{\frac{(n-2s)\epsilon}{2s}}\frac{(n-2s)^2}{2[n+2s-\epsilon(n-2s)]}\Big|\int_{B_n(0,1)}\int_{B_n(0,1)}\frac{v_\epsilon^{2_{s}^{\sharp}-1-\epsilon}(x)v_\epsilon^{2_{s}^{\sharp}-1-\epsilon}(y)}{|x-y|^{n-2s}}dxdy\Big|\\&
\geq
C\epsilon\mu_{\epsilon}^{\frac{(n-2s)\epsilon}{2s}}\Big|\int_{B_n(0,1)}\Big[\int_{B_n(0,1)}\frac{1}{|x-y|^{n-2s}}\big(\frac{1}{1+|y|^2}\big)^{\frac{(n-2s)(2_{s}^{\sharp}-1-\epsilon)}{2}}dy\Big]\big(\frac{1}{1+|x|^2}\big)^{\frac{(n-2s)(2_{s}^{\sharp}-1-\epsilon)}{2}}dx\Big|\\&
\geq C\epsilon\mu_{\epsilon}^{\frac{(n-2s)\epsilon}{2s}},
\end{split}
\end{equation}
where we used that $v_{\epsilon}\rightarrow W$ uniformly on any compact set.

\textbf{Estimate of $R_{\epsilon}$.} From Lemma \ref{cWU} and Hardy-Littlewood-Sobolev inequality, we see that
\begin{equation}\label{right1}
\begin{split}
&\Big(\int_{\mathcal{M}(\Omega,\rho/2)}\Big|\big(\int_{\Omega}\frac{u_{\epsilon}^{2_{s}^{\sharp}-1-\epsilon}(t)}{|x-t|^{n-2s}}dt\big)u_{\epsilon}^{2_{s}^{\sharp}-2-\epsilon}(x)\Big|dx\Big)^{2}
\\&
\leq C \Big(\int_{\mathcal{M}(\Omega,\rho/2)}\Big|\big(\int_{\Omega}\frac{W^{2_{s}^{\sharp}-1-\epsilon}[\mu_{\epsilon}^{\frac{2_{s}^{\sharp}-2-\epsilon}{2s}},x_{\epsilon}](t)}{|x-t|^{n-2s}}dt\big)W^{2_{s}^{\sharp}-2-\epsilon}[\mu_{\epsilon}^{\frac{2_{s}^{\sharp}-2-\epsilon}{2s}},x_{\epsilon}](x)\Big|dx\Big)^{2}
\\&\leq C \Big\|W[\mu_{\epsilon}^{\frac{2_{n,s}^{\ast}-1-\epsilon}{2s}},x_{\epsilon}]\Big\|^{2(2_{n,s}^{\ast}-\epsilon)}_{L^{\frac{2n}{n+2s}(2_{n,s}^{\ast}-\epsilon)}(\mathbb{R}^n)}\Big\|W[\mu_{\epsilon}^{\frac{2_{n,s}^{\ast}-1-\epsilon}{2s}},x_{\epsilon}]\Big\|^{2(2_{n,s}^{\ast}-1-\epsilon)}_{L^{\frac{2n}{n+2s}(2_{n,s}^{\ast}-1-\epsilon)}(\mathbb{R}^n)}\\&
\leq C\big(\frac{1}{\mu_{\epsilon}}\big)^{\frac{[4s-(n-2s)\epsilon](1+2\epsilon)}{2s}}\Big\|\frac{1}{1+|x|^2}\Big\|^{(n-2s)(2_{s}^{\sharp}-1-\epsilon)}_{L^{\frac{n(2_{s}^{\sharp}-1-\epsilon)}{2_{s}^{\sharp}-1}}(\mathbb{R}^n)}\Big\|\frac{1}{1+|x|^2}\Big\|^{(n-2s)(2_{s}^{\sharp}-2-\epsilon)}_{L^{\frac{n(2_{s}^{\sharp}-2-\epsilon)}{2_{s}^{\sharp}-1}}(\mathbb{R}^n)}\\&
\leq C\big(\frac{1}{\mu_{\epsilon}}\big)^{\frac{[4s-(n-2s)\epsilon](1+2\epsilon)}{2s}},
\end{split}
\end{equation}
where $n\in(2s,\min\{6s,n+2s\})$. Due to $x_0\not\in\mathcal{Q}(\Omega,2\rho)$, we obtain $W[\mu_{\epsilon},x_{\epsilon}](x)\leq C\mu_{\epsilon}^{-\frac{n-2s}{2}}$ for $x\in \mathcal{Q}(\Omega,2\rho)$.
Therefore, for a fixed $q>0$ sufficiently large and a number $m>1$, using H\"{o}lder and Hardy-Littlewood-Sobolev inequalities, we obtain
\begin{equation}\label{right2}
\begin{split}
&\Big(\int_{\mathcal{Q}(\Omega,2\rho)}\Big|\big(\int_{\Omega}\frac{u_{\epsilon}^{2_{s}^{\sharp}-1-\epsilon}(t)}{|x-t|^{n-2s}}dt\big)u_{\epsilon}^{2_{s}^{\sharp}-2-\epsilon}(x)\Big|^qdx\Big)^{\frac{2}{q}}\\&
\leq C \Big\|\frac{1}{|x|^{n-2s}}\ast u_{\epsilon}^{2_{s}^{\sharp}-1-\epsilon}\Big\|^2_{L^{qm}(\mathcal{Q}(\Omega,2\rho))}\Big\| u_{\epsilon}^{2_{s}^{\sharp}-2-\epsilon}\Big\|^2_{L^{\frac{qm}{m-1}}(\mathcal{Q}(\Omega,2\rho))}\\&
\leq C \Big\|u_{\epsilon}^{2_{s}^{\sharp}-1-\epsilon}\Big\|^2_{L^{\frac{nqm}{sqm-n}}(\mathcal{Q}(\Omega,2\rho))}\Big\| u_{\epsilon}^{2_{s}^{\sharp}-2-\epsilon}\Big\|^2_{L^{\frac{qm}{m-1}}(\mathcal{Q}(\Omega,2\rho))}\\&
\leq C\big(\frac{1}{\mu_{\epsilon}}\big)^{\frac{[4s-(n-2s)\epsilon][2(2_{s}^{\sharp}-1-\epsilon)-1]}{2s}}.
\end{split}
\end{equation}
Analogously to \eqref{right1} and \eqref{right2}, we obtain
\begin{align}\label{iy}
\begin{split}
&\int_{\mathcal{Q}(\Omega,2\rho)}\Big|\big(\int_{\Omega}\frac{u_{\epsilon}^{2_{s}^{\sharp}-1-\epsilon}(t)}{|x-t|^{n-2s}}dt\big)u_{\epsilon}^{2_{s}^{\sharp}-1-\epsilon}(x)\Big|dx
\\&\leq C \big(\frac{1}{\mu_{\epsilon}}\big)^{\frac{[4s-(n-2s)\epsilon]\epsilon}{4s}}\big(\frac{1}{\mu_{\epsilon}}\big)^{\frac{[4s-(n-2s)\epsilon](2_{s}^{\sharp}-1-\epsilon)}{2s}}=\big(\frac{1}{\mu_{\epsilon}}\big)^{\frac{[4s-(n-2s)\epsilon](2(2_{s}^{\sharp}-1)-\epsilon)}{4s}}.
\end{split}
\end{align}
In addition, applying the oddness of
the integrand and Lemma \ref{cWU},
we see that the last integral is handled as
\begin{equation*}
\begin{split}
&\int_{\mathcal{M}(\Omega,r/2)}\Big(\int_{\Omega}x(x-t)\frac{u_{\epsilon}^{2_{s}^{\sharp}-1-\epsilon}(t)}{|x-t|^{n-2s+2}} dt\Big)u_{\epsilon}^{2_{s}^{\sharp}-1-\epsilon}(x)dx\\&\leq C\int_{\mathbb{R}^n}\int_{\mathbb{R}^n}\frac{xW^{2_{s}^{\sharp}-1-\epsilon}[\mu_{\epsilon}^{\frac{2_{s}^{\sharp}-2-\epsilon}{2s}},x_{\epsilon}](t)W^{2_{s}^{\sharp}-1-\epsilon}[\mu_{\epsilon}^{\frac{2_{s}^{\sharp}-2-\epsilon}{2s}},x_{\epsilon}](x)}{|x-t|^{n-2s+1}} dxdt\\&
=C\big(\frac{1}{\mu_{\epsilon}}\big)^{\frac{[4s-(n-2s)\epsilon]\epsilon}{2s}}\int_{\mathbb{R}^n}\int_{\mathbb{R}^n}\frac{x}{|x-t|^{n-2s+1}}\frac{1}{(1+|t|)^{n+2s-\epsilon(n-2s)}}\frac{1}{(1+|x|)^{n+2s-\epsilon(n-2s)}}dxdt=0.
\end{split}
\end{equation*}
Combining these estimates, we conclude that
\begin{equation}\label{Right}
\begin{split}
R_\epsilon&\leq C\big[\big(\frac{1}{\mu_{\epsilon}}\big)^{\frac{[4s-(n-2s)\epsilon](1+2\epsilon)}{2s}}+\big(\frac{1}{\mu_{\epsilon}}\big)^{\frac{[4s-(n-2s)\epsilon][2(2_{s}^{\sharp}-1-\epsilon)-1]}{2s}}
+\big(\frac{1}{\mu_{\epsilon}}\big)^{\frac{[4s-(n-2s)\epsilon](2(2_{s}^{\sharp}-1)-\epsilon)}{4s}}]\\&
\leq C\big(\frac{1}{\mu_{\epsilon}}\big)^{\frac{[4s-(n-2s)\epsilon](1+2\epsilon)}{2s}}.
\end{split}
\end{equation}
Finally, by putting \eqref{Left} and \eqref{Right} in the statement of Proposition \ref{prosition-identity}, we obtain
$$
\epsilon\mu_{\epsilon}^{\frac{(n-2s)\epsilon}{2s}}\leq C\big(\frac{1}{\mu_{\epsilon}}\big)^{\frac{[4s-(n-2s)\epsilon](1+2\epsilon)}{2s}},
$$
which leads to
$$\epsilon\leq C\big(\frac{1}{\mu_{\epsilon}}\big)^{2+\frac{[4s-(n-2s)\epsilon]\epsilon}{s}},$$
the result follows.
\end{proof}

\begin{lem}\label{thm:existenceofweaksolution}
We have the following limit
\begin{equation}\label{limit}
\lim\limits_{\epsilon\rightarrow0} \mu_{\epsilon}^{\epsilon}=1.
\end{equation}
\end{lem}
\begin{proof}
By the theorem of the mean and Lemma \ref{thm:uniquenessofweaksolution}, we have
\begin{equation}\label{lamta}
\big|\mu_{\epsilon}^{\epsilon}-1\big|=\big|\epsilon\mu_{\epsilon}^{t\epsilon}\log\mu_{\epsilon}\big|\leq M\mu_{\epsilon}^{-1-\frac{[4s-(n-2s)\epsilon]\epsilon}{s}}\log\mu_{\epsilon}
\end{equation}
for some $0<t<1$.
Hence \eqref{limit} can be derived by a direct computation as $\epsilon\rightarrow0$.
\end{proof}

\subsection{Proof of Theorem \ref{prondgr} and Theorem \ref{consequence}}

We shall give the proof of Theorems \ref{prondgr}-\ref{consequence}.

The following important identity will be used later to conclude the proofs of Theorem \ref{prondgr} and Theorem \ref{consequence}.
\begin{lem}\label{p1-00}
		We have that
		\begin{equation*}
			|x|^{-\mu}\ast W_i^{p_s}
			=\int_{\mathbb{R}^n}\frac{W_i^{p_s}(y)}{|x-y|^{\mu}}dy
			=\widetilde{\beta}_{n,\mu,s}W_i^{2_s^{\sharp}-2_{\mu,s}^{\ast}}(x),
		\end{equation*}
		where  $\widetilde{\beta}_{n,\mu,s}=\frac{\pi^{n/2}\Gamma\big(\frac{n-\mu}{2}\big)}{\Gamma(\frac{2n-\mu}{2})}\Big(\frac{2^{2s}\Gamma(\frac{n+2s}{2})\Gamma(\frac{2n-\mu}{2})}{\pi^{n/2}\Gamma\big(\frac{n-2s}{2}\big)\Gamma\big(\frac{n-\mu}{2}\big)}\Big)^{\frac{n-\mu}{n+2s-\mu}}$, $s\in(0,\frac{n}{2})$ and $\mu\in(0,n)$.
	\end{lem}
	\begin{proof}
		The conclusion follows by the Fourier transforms of the kernels of Riesz and Bessel potentials (see (37) in \cite{DHQWF} for example).
	\end{proof}
\begin{proof}[Proof of Theorems \ref{prondgr} and \ref{consequence}]
We have
\begin{equation}\label{laplacian}
A_{s}(\|u_{\epsilon}\|_{\infty}u_{\epsilon})=\alpha_{n,s}\mu_{\epsilon}\big(|x|^{-{(n-2s)}}\ast u_\epsilon^{2_{s}^{\sharp}-1-\epsilon}\big)u_\epsilon^{2_{s}^{\sharp}-2-\epsilon}\quad\mbox{in}\quad \Omega.
\end{equation}
We integrate the right-hand side of \eqref{laplacian}
$$\int_{\Omega}\int_{\Omega}\alpha_{n,s}\mu_{\epsilon}\frac{u_\epsilon^{2_{s}^{\sharp}-1-\epsilon}(y)u_\epsilon^{2_{s}^{\sharp}-2-\epsilon}(x)}{|x-y|^{n-2s}} dxdy=\int_{\Omega_{\epsilon}}\int_{\Omega_{\epsilon}}\alpha_{n,s}\mu_{\epsilon}^{\frac{n-2s}{2s}\epsilon}\frac{v_\epsilon^{2_{s}^{\sharp}-1-\epsilon}(y)v_\epsilon^{2_{s}^{\sharp}-2-\epsilon}(x)}{|x-y|^{n-2s}} dxdy.$$
Thus, combining \eqref{cU} by dominated convergence, Lemma \ref{finite}, Lemma \ref{thm:existenceofweaksolution} and Lemma \ref{p1-00}, we obtain
\begin{equation*}
\begin{split}
\lim\limits_{\epsilon\rightarrow0^{+}}\int_{\Omega}\alpha_{n,s}\mu_{\epsilon}\big(|x|^{-{(n-2s)}}\ast u_\epsilon^{2_{s}^{\sharp}-1-\epsilon}\big)u_\epsilon^{2_{s}^{\sharp}-2-\epsilon}&
=\alpha_{n,s}\int_{\mathbb{R}^n}\int_{\mathbb{R}^n}\frac{W^{2_{s}^{\sharp}-1}(y)W^{2_{s}^{\sharp}-2}(x)}{|x-y|^{n-2s}}dxdy\\&
=\alpha_{n,s}\widetilde{\beta}_{n,s}\int_{\mathbb{R}^n}W^{2_{s}^{\sharp}-1}(x)dx=b_{n,s},
\end{split}
\end{equation*}
where $\widetilde{\beta}_{n,s}:=\widetilde{\beta}_{n,n-2s,s}$ and
\begin{equation}\label{bns}
b_{n,s}:=\frac{|S^{n-1}|}{2}\frac{\Gamma(s)\Gamma(n/2)}{\Gamma((n+2s)/2)}\alpha^{2_{s}^{\sharp}}_{n,s}\widetilde{\beta}_{n,s}.
\end{equation}
Moreover, by H\"{o}lder inequality and the estimates of integration, we have proved that
\begin{equation}\label{linfinity}
\int_{\mathbb{R}^n}\frac{1}{|x-y|^{n-2s}}\frac{1}{(1+|y|)^{(n-2s)(2_{s}^{\sharp}-1-\epsilon)}}dy\in L^{\infty}(\mathbb{R}^n).
\end{equation}
Also using the bound \eqref{00cU} and \eqref{linfinity}, we find
\begin{equation*}
\begin{split}
\mu_{\epsilon}\big(\int_{\Omega}\frac{u_\epsilon^{2_{s}^{\sharp}-1-\epsilon}}{|x-y|^{n-2s}} \big)u_\epsilon^{2_{s}^{\sharp}-2-\epsilon}\leq
C_0\big(\frac{1}{\mu_{\epsilon}}\big)^{\frac{2_{s}^{\sharp}-2-\epsilon}{2s}[\frac{(n-2s)(2(2_{s}^{\sharp}-1-\epsilon)-1)}{2}-2s]}\big(\frac{1}{|x-x_{0}|}\big)^{(n-2s)(2_{s}^{\sharp}-2-\epsilon)}
\end{split}
\end{equation*}
for $x\neq x_0$ and some $C_0>0$. It is noticing that
$(n-2s)(2(2_{n,s}^{\ast}-\epsilon)-1)>4s$,
hence, we deduce that
\begin{equation}\label{LWUQ}
\alpha_{n,s}\mu_{\epsilon}\big(|x|^{-{(n-2s)}}\ast u_\epsilon^{2_{s}^{\sharp}-1-\epsilon}\big)u_\epsilon^{2_{s}^{\sharp}-2-\epsilon}\rightarrow0\quad\mbox{for}\quad x\neq x_0.
\end{equation}
From here we have
$$
A_{s}(\|u_{\epsilon}\|_{L^{\infty}(\Omega)}u_{\epsilon})\rightarrow b_{n,s}\delta_{x=x_0}\hspace{3mm}\mbox{as}\hspace{3mm}\epsilon\rightarrow0
$$
in the sense of distributions in $\Omega$. If we denote $\tilde{u}_{\epsilon}:=v(\|u_{\epsilon}\|_{L^{\infty}(\Omega)}u_{\epsilon})$,  then we have that $\tilde{u}_{\epsilon}$ converges to zero uniformly on any neighborhood $\omega$ of $\partial\Omega$, not containing $x_0$.
Observe that, we get from the integral representation of $w_{\epsilon}$ the s-harmonic extension of $u_{\epsilon}$ that
$$
\|u_{\epsilon}\|_{L^{\infty}(\Omega)}w_{\epsilon}(z)=\int_{\Omega}G_{\mathcal{C}}(z,t)\tilde{u}_{\epsilon}(t)dt=\int_{\Omega}[G_{\mathbb{R}^{n+1}_{+}}\big((x,y),t\big)-H_{\mathcal{C}}\big((x,y),t\big)]\tilde{u}_{\epsilon}(t)dt.
$$
On the other hand we have $H_{\mathcal{C}}((x,y),\cdot)$ is in $C_{loc}^{\infty}$ and $\|H_{\mathcal{C}}((x,y),\cdot)\|_{L^{2_{s}^{\ast}}(\Omega)}\leq C$ which holds uniformly on any neighborhood $\omega$ of $\partial\Omega$, not containing $x_0$.
In conclusion, we have
$$ \|u_{\epsilon}\|_{L^{\infty}(\Omega)}w_{\epsilon}(z)\rightarrow b_{n,s}G_{\mathcal{C}}(z,x_0)
\hspace{2mm}\mbox{in}\hspace{2mm}C_{loc}^{0}(\bar{\mathcal{C}}\setminus\{(x_0,0)\}).
$$
Furthermore, by elliptic regularity theory, pointwise convergence in $\mathcal{C}$ holds for all derivatives of $\|u_{\epsilon}\|_{L^{\infty}(\Omega)}w_{\epsilon}(z)$. In particular,
combining the regularity property of the function $H_{\mathcal{C}}$ for $t=0$ (see \cite{Kenig}) yields that
\begin{equation}\label{ox}
\begin{split}
\lim\limits_{\epsilon\rightarrow0^{+}}&\|u_{\epsilon}\|_{L^{\infty}(\Omega)}u_{\epsilon}(x)=b_{n,s}G(x,x_{0})\\&
\hspace{2mm}\mbox{in}\hspace{2mm}
\left\lbrace
\begin{aligned}
&C_{loc}^\alpha(\Omega\setminus\{x_0\})\hspace{2mm}\mbox{for all}\hspace{2mm}\alpha\in(0,2s)\hspace{2mm}\mbox{if}\hspace{2mm}s\in(0,1/2],\hspace{2mm}\mbox{as}\hspace{2mm}\epsilon\rightarrow0,
\\&C_{loc}^{1,\alpha}(\Omega\setminus\{x_0\})\hspace{2mm}\mbox{for all}\hspace{2mm}\alpha\in(0,2s-1)\hspace{2mm}\mbox{if}\hspace{2mm}s\in(1/2,1),\hspace{2mm}\mbox{as}\hspace{2mm}\epsilon\rightarrow0,
   \end{aligned}
\right.
\end{split}
\end{equation}
Concluding the proof.
\end{proof}

\section{Proof of Theorem \ref{prondgr-1} and Theorem \ref{remainder terms}}\label{remainder}
In this section, we are position to prove Theorems \ref{prondgr-1} and \ref{remainder terms}.
For the simplicity of notations,
we write $X_0$ instead of $(x_0,0)$ in the sequel. Let us define the sets
$$
B_r=B_{n+1}(X_0,r)\cap\mathbb{R}^{n+1}_{+},\hspace{2mm}\partial B_r^{+}=\partial B_r\cap\mathbb{R}^{n+1}_{+}\hspace{2mm}\mbox{and}\hspace{2mm} \mathcal{Z}_r=B_{n}(x_0,r) \hspace{2mm}\mbox{for}\hspace{2mm}r>0\hspace{2mm}\mbox{small}.
$$
\begin{lem}\label{GXX}
Let $\mathcal{O}_x(x):=\|w_{\epsilon}(\cdot,0)\|_{\infty}w_{\epsilon}(\cdot,0)$, it hold that
\begin{equation}\label{GXX-1}
\lim\limits_{\epsilon\rightarrow0}\int_{\delta}^{2\delta}\int_{\partial B_r^{+}}y^{1-2s}|\nabla \mathcal{O}_x|^2\nu_jdSdr=
b^2_{n,s}\int_{\delta}^{2\delta}\int_{\partial B_r^{+}}y^{1-2s}|\nabla G|^2\nu_jdSdr,
\end{equation}
and
\begin{equation}\label{GXX-2}
\lim\limits_{\epsilon\rightarrow0}\int_{\delta}^{2\delta}\int_{\partial B_r^{+}}y^{1-2s}\langle\nabla \mathcal{O}_x,\nu\rangle \partial_j\mathcal{O}_xdSdr
=b^2_{n,s}\int_{\delta}^{2\delta}\int_{\partial B_r^{+}}y^{1-2s}\langle\nabla G,\nu\rangle \partial_jGdSdr
\end{equation}
for $r>0$ small, where $\nu_j$ is the $j$-th component of $\nu$ and $\partial_j$ is the partial derivative with respect to the $k$-th variable.
\end{lem}
\begin{proof}
By the Green's representation, we have
$$
\mathcal{O}_x(z)=\int_{\Omega}G_{\mathcal{C}}(z,t)\tilde{u}_{\epsilon}(t)dt,
$$
where $\tilde{u}_{\epsilon}:=A_{s}(\|u_{\epsilon}\|_{L^{\infty}(\Omega)}u_{\epsilon})$ in section \ref{consequence00}.
A direct differentiation yields
\begin{equation*}
\begin{split}
|\nabla \mathcal{O}_x(z)|^2=&\Big|\int_{\Omega}\nabla_{z}G_{\mathcal{C}}(z,t)\tilde{u}_{\epsilon}(t)dt\Big|^2\leq C\big(\sup\limits_{t\in\mathcal{Z}_{r/2}}\big|\nabla_{z}G_{\mathcal{C}}(z,t)\big|\big)^2
+\Big|\int_{\Omega\setminus \mathcal{Z}_{r/2}}\nabla_{z}G_{\mathcal{C}}(z,t)\tilde{u}_{\epsilon}(t)dt\Big|^2.
\end{split}
\end{equation*}
We next proceed similarly to \cite{CKL}. We find
$$
\int_{\delta}^{2\delta}y^{1-2s}\big(\sup\limits_{t\in\mathcal{Z}_{r/2}}\big|\nabla_{z}G_{\mathcal{C}}(z,t)\big)^2dz\leq C
\hspace{2mm}\mbox{and}\hspace{2mm} \lim\limits_{\epsilon\rightarrow0}\int_{\delta}^{2\delta}\big(\int_{\Omega\setminus \mathcal{Z}_{r/2}}\nabla_{z}G_{\mathcal{C}}(z,t)\tilde{u}_{\epsilon}(t)dt\big)^2dz=0.
$$
By combining this with the dominated convergence theorem, and $\mathcal{O}_x$ converges to $b_{n,s}G$ uniformly in on
any neighborhood $\omega$ of $\Omega$, not containing $x_0$, we finally derive \eqref{GXX-1} and \eqref{GXX-2} by taking limit.

\end{proof}

\begin{proof}[Proof of Theorem \ref{prondgr-1}].
We note that
$$
I_0:=\int_{\mathcal{Z}_{r}}\Big(\int_{\mathcal{Z}_{r}}\frac{x_j-t_j}{|x-t|^{\mu+2}}F(w)dt\Big)F(w)dx=-\int_{\mathcal{Z}_{r}}\Big(\int_{\mathcal{Z}_{r}}\frac{x_j-t_j}{|x-t|^{\mu+2}}F(w)dt\Big)F(w)dx,$$
which implies $I_0=0$.
If $w_{\epsilon}$ is a solution to \eqref{CF}, for each $1\leq j\leq n$, and using the following identity
$$y^{1-2s}\nabla w_{\epsilon}\cdot\nabla\partial_j w_{\epsilon}=div(y^{1-2s}\partial_jw_{\epsilon}\nabla w_{\epsilon})-\partial_j\cdot div(y^{1-2s}\nabla w_{\epsilon}),$$
the divergence theorem and \eqref{CFLL-00}, we obtain
\begin{equation*}
\begin{split}
\int_{\partial B_r^{+}}y^{1-2s}|\nabla w_{\epsilon}|^2\nu_jdS&=\int_{B_{r}}y^{1-2s}\partial_j|\nabla w_{\epsilon}|^2dz=2\int_{B_{r}}y^{1-2s}\nabla w_{\epsilon}\cdot\nabla \partial_jw_{\epsilon}dz\\&
=2k_{s}\int_{\partial B_r^{+}}y^{1-2s}\langle\nabla w_{\epsilon},\nu\rangle \partial_jw_{\epsilon}dS+2k_s\int_{\mathcal{Z}_{r}}\partial_{\nu}^{s}w_{\epsilon}\partial_jw_{\epsilon}dz\\&
=2\int_{\partial B_r^{+}}y^{1-2s}\langle\nabla w_{\epsilon},\nu\rangle \partial_jw_{\epsilon}dS+\frac{2k_{s}}{2_{s}^{\sharp}-1-\epsilon}\int_{\partial\mathcal{Z}_{r}}(\int_{\Omega\times\{0\}}\frac{w_{\epsilon}^{2_{s}^{\sharp}-1-\epsilon}}{|x-t|^{n-2s}}dt\Big)w_{\epsilon}^{2_{s}^{\sharp}-1-\epsilon}\nu_jdS_x\\&
\hspace{3mm}+\frac{2k_{s}}{2_{s}^{\sharp}-1-\epsilon}(n-2s)\int_{\mathcal{Z}_{r}}\Big(\int_{\Omega\setminus\mathcal{Z}_{r}}\frac{x_j-t_j}{|x-t|^{n-2s+2}}w_{\epsilon}^{2_{s}^{\sharp}-1-\epsilon}(t,0)dt\Big)w_{\epsilon}^{2_{s}^{\sharp}-1-\epsilon}dx.
\end{split}
\end{equation*}
Multiplying the above identity by the function $\|w_{\epsilon}(\cdot,0)\|_{\infty}^2$, and integrating in the interval $[\delta,2\delta]$ with respect to $r$, we see that
\begin{equation}\label{epsilonfrac}
\begin{split}
\int_{\delta}^{2\delta}\int_{\partial B_r^{+}}y^{1-2s}&|\nabla \mathcal{O}_x|^2\nu_jdSdr=
2\int_{\delta}^{2\delta}\int_{\partial B_r^{+}}y^{1-2s}\langle\nabla \mathcal{O}_x,\nu\rangle \partial_j\mathcal{O}_xdSdr\\&
+\frac{2k_{s}}{2_{s}^{\sharp}-1-\epsilon}\big\|w_{\epsilon}\big\|_{\infty}^2\int_{\delta}^{2\delta}\int_{\partial\mathcal{Z}_{r}}(\int_{\Omega\times\{0\}}\frac{w_{\epsilon}^{2_{s}^{\sharp}-1-\epsilon}}{|x-t|^{n-2s}}dt\Big)w_{\epsilon}^{2_{s}^{\sharp}-1-\epsilon}\nu_jdS_xdr\\&
+\frac{2k_{s}}{2_{s}^{\sharp}-1-\epsilon}(n-2s)\big\|w_{\epsilon}\big\|_{\infty}^2\int_{\delta}^{2\delta}\int_{\mathcal{Z}_{r}}\Big(\int_{\Omega\setminus\mathcal{Z}_{r}}\frac{(x_j-t_j)w_{\epsilon}^{2_{s}^{\sharp}-1-\epsilon}}{|x-t|^{n-2s+2}}dt\Big)w_{\epsilon}^{2_{s}^{\sharp}-1-\epsilon}dxdr
\end{split}
\end{equation}
for $r>0$ small. Due to $x_0\not\in\partial\mathcal{Z}_{r}$, and using Lemma \ref{cWU}, Lemma \ref{p1-00} and \eqref{VEPUSILONG}, we continue to estimate
\begin{equation}\label{fanshu}
\begin{split}
\big\|w_{\epsilon}(x,0)\big\|^{2}_{\infty}\int_{\partial\mathcal{Z}_{r}}(\int_{\Omega\times\{0\}}\frac{w_{\epsilon}^{2_{s}^{\sharp}-1-\epsilon}(t,0)}{|x-t|^{n-2s}}dt\Big)w_{\epsilon}^{2_{s}^{\sharp}-1-\epsilon}(x,0)\nu_jdS_x&
\leq C\big\|w_{\epsilon}(x,0)\big\|^{2}_{\infty}\int_{\partial\mathcal{Z}_{r}}w_{\epsilon}^{2_{s}^{\sharp}}dS_{x}\\&
\leq C\big\|w_{\epsilon}(x,0)\big\|^{2}_{\infty}\big\|w_{\epsilon}(x,0)\big\|^{-\frac{[4s-\epsilon(n-2s)]n}{2s(n-2s)}}_{\infty}\\&
\leq C\big\|w_{\epsilon}(x,0)\big\|^{-\frac{[8s^2-\epsilon (n-2s)n]}{2s(n-2s)}}_{\infty}.
\end{split}
\end{equation}
By taking the limit in $[\delta,2\delta]$, we find
\begin{equation}\label{G-1}
\begin{split}
\lim\limits_{\epsilon\rightarrow0}\big\|w_{\epsilon}(x,0)\big\|^{2}_{\infty}\int_{\delta}^{2\delta}\int_{\partial\mathcal{Z}_{r}}(\int_{\Omega\times\{0\}}\frac{w_{\epsilon}^{2_{s}^{\sharp}-1-\epsilon}}{|x-t|^{n-2s}}dt\Big)w_{\epsilon}^{2_{s}^{\sharp}-1-\epsilon}\nu_jdS_xdr
\leq  \lim\limits_{\epsilon\rightarrow0}\int_{\delta}^{2\delta}\big\|w_{\epsilon}\big\|^{-\frac{[8s^2-\epsilon (n-2s)n]}{2s(n-2s)}}_{\infty}dr=0
\end{split}
\end{equation}
for small $r>0$.
Since $x_0\not\in\Omega\setminus\mathcal{Z}_{r}$, we note that
\begin{equation*}
\begin{split}
&\big\|w_{\epsilon}(x,0)\big\|^{2}_{\infty}\int_{\mathcal{Z}_{r}}\Big(\int_{\Omega\setminus\mathcal{Z}_{r}}\frac{(x_j-t_j)w_\epsilon^{2_{s}^{\sharp}-1-\epsilon}}{|x-t|^{n-2s+2}}dt\Big)w_\epsilon^{2_{s}^{\sharp}-1-\epsilon}dx\\&\leq C
\frac{\big\|w_{\epsilon}(x,0)\big\|^{2}_{\infty}}{\big\|w_{\epsilon}(x,0)\big\|^{\frac{(n-2s)(2_{s}^{\sharp}-1-\epsilon)(2_{s}^{\sharp}-2-\epsilon)}{4s}}_{\infty}}\Big(\int_{\mathbb{R}^{n}}w_{\epsilon}^{2_{s}^{\sharp}-1-\epsilon}dx\Big)\Big(\int_{\Omega\setminus\mathcal{Z}_{r}}\frac{1}{|x-t|^{n-2s+1}}dt\Big)\\&
\leq
\begin{cases}
Cr^{2s-1}\big\|w_{\epsilon}(x,0)\big\|^{-\frac{8s^2-n(n-2s)\epsilon}{2s(n-2s)}}_{\infty} \quad\quad\quad\hspace{6mm}\mbox{ if } 0<s<\frac{1}{2},\\
C|\log r|\big\|w_{\epsilon}(x,0)\big\|^{-\frac{8s^2-n(n-2s)\epsilon}{2s(n-2s)}}_{\infty} \hspace{4.8mm}\quad\quad\quad\mbox{ if } s=\frac{1}{2},\\
C\big\|w_{\epsilon}(x,0)\big\|^{-\frac{8s^2-n(n-2s)\epsilon}{2s(n-2s)}}_{\infty}\quad\quad\hspace{7.2mm}\quad\quad\quad \mbox{ if } \frac{1}{2}<s<1,
\end{cases}
\end{split}
\end{equation*}
Hence we have
\begin{equation*}
\begin{split}
\lim\limits_{\epsilon\rightarrow0}&\big\|w_{\epsilon}(x,0)\big\|^{2}_{\infty}\int_{\delta}^{2\delta}\int_{\mathcal{Z}_{r}}\Big(\int_{\Omega\setminus\mathcal{Z}_{r}}\frac{(x_j-t_j)w_\epsilon^{2_{s}^{\sharp}-1-\epsilon}}{|x-t|^{n-2s+2}}dt\Big)w_\epsilon^{2_{s}^{\sharp}-1-\epsilon}dxdr
\\&\leq C\lim\limits_{\epsilon\rightarrow0}\int_{\delta}^{2\delta}\big\|w_{\epsilon}(x,0)\big\|^{-\frac{8s^2-n(n-2s)\epsilon}{2s(n-2s)}}_{\infty}\max\big\{r^{2s-1}, |\log r|, 1\big\}dr=0.
\end{split}
\end{equation*}
Combining this identity and \eqref{epsilonfrac}-\eqref{G-1} with Lemma \ref{GXX}, we conclude that
\begin{equation}\label{GX0}
\int_{\delta}^{2\delta}\int_{\partial B_r^{+}}y^{1-2s}|\nabla G|^2\nu_jdSdr=2\int_{\delta}^{2\delta}\int_{\partial B_r^{+}}y^{1-2s}\langle\nabla G,\nu\rangle \partial_jGdSdr
\end{equation}
for sufficiently small $\epsilon>0$. As the calculous of Theorem 1.2 in \cite{CKL}, we can show that
\begin{equation}\label{GX0-1}
\lim\limits_{\delta\rightarrow0}\frac{1}{\delta}\int_{\delta}^{2\delta}\int_{\partial B_r^{+}}y^{1-2s}|\nabla G|^2\nu_jdSdr=
2(n-2s)\gamma_{n,s}\frac{|S^{n-1}|}{2n}B(1-s,\frac{n+2}{2})\partial_j\phi(x_0),
\end{equation}
and
\begin{equation}\label{GX0-2}
\lim\limits_{\delta\rightarrow0}\frac{1}{\delta}2\int_{\delta}^{2\delta}\int_{\partial B_r^{+}}y^{1-2s}|\nabla G|^2\nu_jdSdr=2(n-2s+3)(n-2s)\gamma_{n,s}\frac{|S^{n-1}|}{2n}B(1-s,\frac{n+2}{2})\partial_j\phi(x_0).
\end{equation}
Consequently, the result of theorem follows by the identities \eqref{GX0}, \eqref{GX0-1} and \eqref{GX0-2}.
\end{proof}

\subsection{Proof of Theorem \ref{remainder terms}}
By exploiting the local form of the Pohozaev identity and blow-up analysis, we can prove the following result.
\begin{lem}\label{remainder-1}
It holds that
\begin{equation}\label{remainder-2}
\begin{split}
\lim\limits_{\epsilon\rightarrow0}&\kappa_{s}\bigg(\frac{n+2s}{2(2_{s}^{\sharp}-1-\epsilon)}-\frac{n-2s}{2}\bigg)\alpha^2_{n,s}\widetilde{\beta}_{n,n-2s,s}\mu_{\epsilon}^{\frac{4s+(n-2s)\epsilon}{2s}}\delta\int_{\mathbb{R}^n}W^{2_{s}^{\sharp}}(x)dx\\&
=b_{n,s}^2\int_{\delta}^{2\delta}\int_{\partial B_r^{+}}\bigg[y^{1-2s}\Big\langle\langle z-X_0, \nabla G\rangle\nabla G-\frac{|\nabla G|^2}{2}(z-X_0),\nu\Big\rangle+\frac{n-2s}{2}y^{1-2s}G \frac{\partial G}{\partial\nu}\bigg]dS_xdr.
\end{split}
\end{equation}
\end{lem}
\begin{proof}
Recalling the proof of Theorem \ref{prondgr}, we find
$$
\lim\limits_{\epsilon\rightarrow0}\int_{\Omega}\tilde{u}_{\epsilon}(x)dx=b_{n,s}\hspace{2mm}\mbox{and}\hspace{2mm}
\lim\limits_{\epsilon\rightarrow0}\tilde{u}_{\epsilon}(t)=0\hspace{2mm}\mbox{in}\hspace{2mm}C_{loc}^{0}(\Omega\setminus\{x_0\}),
$$
where $b_{n,s}$ is defined in \eqref{bns}. Similarly to \eqref{local-po}, there exists some subset $\tilde{X}_0\in\mathbb{R}_{+}^{n+1}$ such that for a function $\bar{W}\in C^1(D)$ satisfying div$(y^{1-2s}\nabla\bar{W})=0$ in $X$, there holds
\begin{equation}\label{local-po-22}
\mbox{div}\Big\{2y^{1-2s}\langle z-X_0, \nabla\bar{W} \rangle\nabla\bar{W} -y^{1-2s}|\nabla\bar{W}|^2(z-X_0)\Big\}+(n-2s)y^{1-2s}|\nabla \bar{W}|^2=0\hspace{2mm}\mbox{in}\hspace{2mm}\tilde{X}_0.
\end{equation}
It is then natural to expect that we can establish a similar identity analogous to \eqref{Qx} $\mathcal{M}(\Omega,r/2)$ and $\mathcal{D}_r^{+}$ replaced by $\mathcal{Z}_r$ and $B_{r}^{+}$, respectively.
However, such an identity is still insufficient for proving Theorem \ref{remainder terms}, and we therefore need to develop a more refined estimate.
According to the symmetry of integration, we have
\begin{equation*}
\int_{\mathcal{Z}_r}w_\epsilon(x,0)\nabla\Big(x \int_{\mathcal{Z}_r}\frac{w_\epsilon^p(t,0)}{|x-t|^{\mu}}dtw_\epsilon^{p-1}(x,0)\Big)dx=
-\frac{2n-\mu}{2p}\int_{\mathcal{Z}_r}\int_{\mathcal{Z}_r}\frac{w_\epsilon^p(t,0)}{|x-t|^{\mu}}dtw_\epsilon^{p}(x,0)dx.
\end{equation*}
And by a direct computation, we have
\begin{equation*}
\int_{\mathcal{Z}_r}w_\epsilon(x,0)\nabla\Big(x \int_{\Omega\setminus\mathcal{Z}_r}\frac{w_\epsilon^p(t,0)}{|x-t|^{\mu}}dtw_\epsilon^{p-1}(x,0)\Big)dx=
\int_{\mathcal{Z}_r}\int_{\Omega\setminus\mathcal{Z}_r}\Big[-\frac{n}{p}\frac{w_\epsilon^p(t,0)}{|x-t|^{\mu}}+\frac{\mu}{p}x(x-t)\frac{w_\epsilon^p(t,0)}{|x-t|^{\mu+2}}\Big]w_\epsilon^{p}(x,0)dtdx.
\end{equation*}
Thus combining the above two inequalities, we get
\begin{equation*}
\begin{split}
\int_{\mathcal{Z}_r}\langle x, \nabla w_{\epsilon}\rangle\partial_{\nu}^{s}w_{\epsilon}dx&=
-\frac{2n-\mu}{2p}\int_{\mathcal{Z}_r}\int_{\mathcal{Z}_r}\frac{w^p(t,0)}{|x-t|^{\mu}}dtw^{p}(x,0)dx\\&
+\int_{\mathcal{Z}_r}\int_{\Omega\setminus\mathcal{Z}_r}\Big[-\frac{n}{p}\frac{w^p(t,0)}{|x-t|^{\mu}}+\frac{\mu}{p}x(x-t)\frac{w^p(t,0)}{|x-t|^{\mu+2}}\Big]w^{p}(x,0)dtdx.
\end{split}
\end{equation*}
On this basis, we can establish an identity analogous to \eqref{Qx}, and integrating both sides of this identity from $\delta$ to $2\delta$ in $r$, we obtain
\begin{equation}\label{Gx-1}
\begin{split}
&\Big(\frac{2n-\mu}{2p}-\frac{n-2s}{2}\Big)\big\|w_{\epsilon}(x,0)\big\|^2_{\infty}\kappa_{s}\int_{\delta}^{2\delta}
\int_{\mathcal{Z}_r} \int_{\mathcal{Z}_r}\frac{w_\epsilon^p(t,0)}{|x-t|^{\mu}}dtw_\epsilon^{p}(x,0)dxdr\\&\hspace{3.5mm}
+\Big(\frac{n}{p}-\frac{n-2s}{2}\Big)\big\|w_{\epsilon}(x,0)\big\|^2_{\infty}\kappa_{s}\int_{\delta}^{2\delta}\int_{\mathcal{Z}_r} \int_{\Omega\setminus\mathcal{Z}_r}\frac{w_\epsilon^p(t,0)}{|x-t|^{\mu}}dtw_\epsilon^{p}(x,0)dxdr\\&
=\int_{\delta}^{2\delta}\int_{\partial B_r^{+}}y^{1-2s}\bigg[\Big\langle\langle z-X_0, \nabla \mathcal{O}_x\rangle\nabla \mathcal{O}_x-\frac{|\nabla \mathcal{O}_x|^2}{2}(z-X_0),\nu\Big\rangle+\frac{n-2s}{2}\mathcal{O}_x \frac{\partial \mathcal{O}_x}{\partial\nu}\bigg]dS_xdr\\&
\hspace{3.5mm}+\frac{\mu}{p}\big\|w_{\epsilon}(x,0)\big\|^2_{\infty}\int_{\delta}^{2\delta}\int_{\mathcal{Z}_r} \int_{\Omega}x(x-t)\frac{w_\epsilon^p(t,0)}{|x-t|^{\mu+2}}dtw_\epsilon^{p}(x,0)dxdr
\\&\hspace{3.5mm}+\frac{1}{p}\big\|w_{\epsilon}(x,0)\big\|^2_{\infty}\int_{\delta}^{2\delta}\int_{\partial\mathcal{Z}_r} \int_{\Omega}\frac{w_\epsilon^p(t,0)}{|x-t|^{\mu}}dtw_\epsilon^{p}(x,0)\langle x-x_0,\nu\rangle dS_xdr.
\end{split}
\end{equation}
Let us estimate each term on the RHS and the second term of the LHS.
We note that
$$
\Big|\Big\langle\langle z-X_0, \nabla \mathcal{O}_x\rangle\nabla \mathcal{O}_x-\frac{|\nabla \mathcal{O}_x|^2}{2}(z-X_0),\nu\Big\rangle\Big|\leq C\big|\nabla\mathcal{O}_x\big|^2
$$
on the interval $[\delta,2\delta]$. By the dominated convergence theorem,
similar to \eqref{GXX-1} in the proof of Lemma \ref{GXX}, one has
\begin{equation}\label{Gx-2}
\begin{split}
&\lim\limits_{\epsilon\rightarrow0}\int_{\delta}^{2\delta}\int_{\partial B_r^{+}}y^{1-2s}\Big\langle\langle z-X_0, \nabla \mathcal{O}_x\rangle\nabla \mathcal{O}_x-\frac{|\nabla \mathcal{O}_x|^2}{2}(z-X_0),\nu\Big\rangle dS_xdr\\&
=b_{n,s}^2\int_{\delta}^{2\delta}\int_{\partial B_r^{+}}y^{1-2s}\Big\langle\langle z-X_0, \nabla G\rangle\nabla G-\frac{|\nabla G|^2}{2}(z-X_0),\nu\Big\rangle dS_xdr,
\end{split}
\end{equation}
where $b_{n,s}$ is defined in \eqref{bns}. Similarly, we have the following estimate
\begin{equation}\label{Gx-3}
\lim\limits_{\epsilon\rightarrow0}\int_{\delta}^{2\delta}\int_{\partial B_r^{+}}y^{1-2s}\mathcal{O}_x \frac{\partial \mathcal{O}_x}{\partial\nu}dS_xdr=
b_{n,s}^2\int_{\delta}^{2\delta}\int_{\partial B_r^{+}}y^{1-2s}G \frac{\partial G}{\partial\nu}dS_xdr.
\end{equation}
In additional, applying the symmetry, we have
\begin{equation*}
\int_{\Omega}\Big(\int_{\Omega}t(t-x)\frac{u_{\epsilon}^{p}(t)}{|x-t|^{n-2s+2}} dt\Big)u_{\epsilon}^{p}(x)dx=-\int_{\Omega}\Big(\int_{\Omega}(-t)(t-x)\frac{u_{\epsilon}^{p}(t)}{|x-t|^{n-2s+2}} dt\Big)u_{\epsilon}^{p}(x)dx.
\end{equation*}
which leads to
\begin{equation*}
\begin{split}
\int_{\Omega}\Big(\int_{\Omega}x(x-t)\frac{u_{\epsilon}^{p}(t)}{|x-t|^{n-2s+2}} dt\Big)u_{\epsilon}^{p}(x)dx&=\frac{1}{2}\int_{\Omega}\int_{\Omega}\frac{[x(x-t)-(-t)(t-x)]u_{\epsilon}^{p}(t)u_{\epsilon}^{p}(x)}{|x-t|^{n-2s+2}} dxdt\\&
=\frac{1}{2}\int_{\Omega}\int_{\Omega}\frac{u_{\epsilon}^{p}(t)u_{\epsilon}^{p}(x)}{|x-t|^{n-2s}} dxdt.
\end{split}
\end{equation*}
Thus, by choosing $p=2_{s}^{\sharp}-1-\epsilon$ and $\mu=n-2s$, together with Lemma \ref{cWU} and Hardy-Littlewood-Sobolev inequality, we see that the second term on the RHS of \eqref{Gx-1} is handled as
\begin{equation}\label{Gx-4}
\begin{split}
&\big\|w_{\epsilon}(x,0)\big\|^2_{\infty}\int_{\mathcal{Z}_r} \int_{\Omega}x(x-t)\frac{w_\epsilon^{2_{s}^{\sharp}-1-\epsilon}}{|x-t|^{n-2s+2}}dtw_\epsilon^{2_{s}^{\sharp}-1-\epsilon}dx
\\&\leq C\big\|w_{\epsilon}(x,0)\big\|^{2}_{\infty}\big\|w_{\epsilon}(x,0)\big\|^{-\frac{[4s-\epsilon(n-2s)](1+\epsilon)}{2s}}_{\infty}\Big(\int_{\Omega}\big(\frac{1}{1+|x|}dx\big)^{\frac{2n[n+2s-\epsilon(n-2s)]}{n+2s}}\Big)^{(n+2s)/n}
\\&\leq C\big\|w_{\epsilon}(x,0)\big\|^{-\frac{\epsilon[6s-n-\epsilon(n-2s)]}{2s}}_{\infty}.
\end{split}
\end{equation}
For the third term, similar to \eqref{fanshu}, we have
 \begin{equation}\label{Gx-5}
\begin{split}
\Big\|w_{\epsilon}(\cdot,0)\Big\|^2_{\infty}\int_{\partial\mathcal{Z}_r} \int_{\Omega}\frac{w_\epsilon^{2_{s}^{\sharp}-1-\epsilon}(t,0)}{|x-t|^{n-2s}}dtw_\epsilon^{2_{s}^{\sharp}-1-\epsilon}(x,0)\langle x-x_0,\nu\rangle dS_x\leq C\big\|w_{\epsilon}(x,0)\big\|^{-\frac{[8s^2-\epsilon (n-2s)n]}{2s(n-2s)}}_{\infty}.
\end{split}
\end{equation}
Taking the limit in \eqref{Gx-4} and \eqref{Gx-5}, we obtain
\begin{equation}\label{Gx-6}
\begin{split}
&\lim\limits_{\epsilon\rightarrow0}\Big\|w_{\epsilon}(x,0)\Big\|^2_{\infty}\int_{\delta}^{2\delta}\bigg[\frac{n-2s}{2_{s}^{\sharp}-1-\epsilon}\int_{\mathcal{Z}_r} \int_{\Omega}x(x-t)\frac{w_\epsilon^{2_{s}^{\sharp}-1-\epsilon}(t,0)}{|x-t|^{n-2s+2}}dtw_\epsilon^{2_{s}^{\sharp}-1-\epsilon}(x,0)dx
\\&\hspace{3.5mm}+\frac{1}{2_{s}^{\sharp}-1-\epsilon}\int_{\partial\mathcal{Z}_r} \int_{\Omega}\frac{w_\epsilon^{2_{s}^{\sharp}-1-\epsilon}(t,0)}{|x-t|^{n-2s}}dtw_\epsilon^{2_{s}^{\sharp}-1-\epsilon}(x,0)\langle x-x_0,\nu\rangle dS_x\bigg]dr=0.
\end{split}
\end{equation}
For the second term of the LHS, similar to \eqref{Gx-5} and \eqref{Gx-6}, we also have
 \begin{equation}\label{Gx-5-1}
\begin{split}
\lim\limits_{\epsilon\rightarrow0}\Big\|w_{\epsilon}(x,0)\Big\|^2_{\infty}\int_{\delta}^{2\delta}\int_{\mathcal{Z}_r} \int_{\Omega\setminus\mathcal{Z}_r}\frac{w_\epsilon^{2_{s}^{\sharp}-1-\epsilon}(t,0)}{|x-t|^{n-2s}}dtw_\epsilon^{2_{s}^{\sharp}-1-\epsilon}(x,0)dx=0.
\end{split}
\end{equation}
Combining this bound with \eqref{Gx-1}-\eqref{Gx-3} and \eqref{Gx-6}-\eqref{Gx-5-1}, we see clearly that
\begin{equation}\label{Gx-7}
\begin{split}
&\lim\limits_{\epsilon\rightarrow0}\Big\|w_{\epsilon}(x,0)\Big\|^2_{\infty}\kappa_{s}\bigg(\frac{n+2s}{2(2_{s}^{\sharp}-1-\epsilon)}-\frac{n-2s}{2}\bigg)\int_{\delta}^{2\delta}\int_{\mathcal{Z}_r} \int_{\mathcal{Z}_r}\frac{w_\epsilon^{2_{s}^{\sharp}-1-\epsilon}(t,0)}{|x-t|^{n-2s}}dtw_\epsilon^{2_{s}^{\sharp}-1-\epsilon}(x,0)dxdr.
\\&=b_{n,s}^2\int_{\delta}^{2\delta}\int_{\partial B_r^{+}}\bigg[y^{1-2s}\Big\langle\langle z-X_0, \nabla G\rangle\nabla G-\frac{|\nabla G|^2}{2}(z-X_0),\nu\Big\rangle+\frac{n-2s}{2}y^{1-2s}G \frac{\partial G}{\partial\nu}\bigg]dS_xdr.
\end{split}
\end{equation}

On the other hand, we continue to eatimate
 \begin{equation}\label{Gx-8}
\begin{split}
 &\Big\|w_{\epsilon}(x,0)\Big\|^2_{\infty}\int_{\mathcal{Z}_r}\int_{\mathcal{Z}_r}\frac{w_\epsilon^{2_{s}^{\sharp}-1-\epsilon}(t,0)}{|x-t|^{n-2s}}dtw_\epsilon^{2_{s}^{\sharp}-1-\epsilon}(x,0)dx\\&
 =\alpha^2_{n,s}\mu_{\epsilon}^2\int_{\frac{2_{s}^{\sharp}-2-\epsilon}{2s}(\mathcal{Z}_r-x_\epsilon)}\int_{\frac{2_{s}^{\sharp}-2-\epsilon}{2s}(\mathcal{Z}_r-x_\epsilon)}\mu_{\epsilon}^{\frac{(n-2s)\epsilon}{2s}}\frac{v_\epsilon^{2_{s}^{\sharp}-1-\epsilon}(y)v_\epsilon^{2_{s}^{\sharp}-1-\epsilon}(x)}{|x-t|^{n-2s}} dtdx
 \\&=\alpha^2_{n,s}\mu_{\epsilon}^{\frac{4s+(n-2s)\epsilon}{2s}}\int_{\frac{2_{s}^{\sharp}-2-\epsilon}{2s}(\mathcal{Z}_r-x_\epsilon)}\int_{\frac{2_{s}^{\sharp}-2-\epsilon}{2s}(\mathcal{Z}_r-x_\epsilon)}\frac{v_\epsilon^{2_{s}^{\sharp}-1-\epsilon}(y)v_\epsilon^{2_{s}^{\sharp}-1-\epsilon}(x)}{|x-t|^{n-2s}} dtdx.
 \end{split}
\end{equation}
and by Lemma \ref{p1-00} we note that
\begin{equation}\label{Gx-9}
\begin{split}
\int_{\frac{2_{s}^{\sharp}-2-\epsilon}{2s}(\mathcal{Z}_r-x_\epsilon)}\int_{\frac{2_{s}^{\sharp}-2-\epsilon}{2s}(\mathcal{Z}_r-x_\epsilon)}\frac{v_\epsilon^{2_{s}^{\sharp}-1-\epsilon}(y)v_\epsilon^{2_{s}^{\sharp}-1-\epsilon}(x)}{|x-t|^{n-2s}}dtdx
&=\int_{\mathbb{R}^n}\int_{\mathbb{R}^n}\frac{W^{2_{s}^{\sharp}-1}(y)W^{2_{s}^{\sharp}-1}(x)}{|x-y|^{n-2s}}dtdx
\\&=\widetilde{\beta}_{n,s}\int_{\mathbb{R}^n}W^{2_{s}^{\sharp}}(x)dx=B_{n,s}\widetilde{\beta}_{n,s},
\end{split}
\end{equation}
where
$$
B_{n,s}=\sigma_n\int_{0}^{\infty}\frac{r^{n-1}}{(1+r^2)^n}dr.
$$
Coupling \eqref{Gx-6}, \eqref{Gx-7} and \eqref{Gx-8}, \eqref{remainder-2} is proved.
\end{proof}

We are now in position to conclude the proof of Theorem \ref{remainder terms}.
\begin{proof}[Proof of Theorem \ref{remainder terms}]
Let us consider the functions
$$G(z)=\frac{\gamma_{n,s}}{|z-X_0|^{n-2s}}-H(z)\hspace{2mm}\mbox{and}\hspace{2mm}\nabla G(z)=-\gamma_{n,s}(n-2s)\frac{z-X_0}{|z-X_0|^{n-2s+2}}-\nabla H(z).$$
Substituting the above $G$ and $\nabla G$ into the right-hand side of \eqref{remainder-2} yields
\begin{equation*}
\begin{split}
\lim\limits_{\epsilon\rightarrow0}&2\kappa_{s}\bigg(\frac{n+2s}{2(2_{s}^{\sharp}-1-\epsilon)}-\frac{n-2s}{2}\bigg)\alpha^2_{n,s}\widetilde{\beta}_{n,s}\mu_{\epsilon}^{\frac{4s+(n-2s)\epsilon}{2s}}\delta\int_{\mathbb{R}^n}W^{2_{s}^{\sharp}}(x)dx\\&
=(n-2s)^2\gamma_{n,s}b_{n,s}^2\lim\limits_{r\rightarrow0}\bigg[2\int_{\partial B_{2r}^{+}}\frac{y^{1-2s}}{(2r)^{n-2s+1}}H(z)dS-\int_{\partial B_{r}^{+}}\frac{y^{1-2s}}{r^{n-2s+1}}H(z)dS\bigg]\\&
+\lim\limits_{\delta\rightarrow0}\frac{1}{\delta}\int_{\delta}^{2\delta}\int_{\partial B_r^{+}}y^{1-2s}O\Big(\big\langle\nu,\nabla H(z)\big\rangle\Big)\Big(\frac{1}{r^{n-2s}}+H(z)\Big)+r|\nabla H(z)|^2\bigg]dSdr=\mathcal{Q}_1+\mathcal{Q}_2.
\end{split}
\end{equation*}
According to Lemma 2.9 in \cite{CDS}, we find that $\partial_iH_{\mathcal{C}}(\cdot,x_0)$ has a bounded H\"{o}lder norm over a small neighborhood of $x_0$ for any $i=1,\cdots,n$. Then we can derive $\mathcal{Q}_2=0$. As a consequence, we deduce that
\begin{equation*}
\begin{split}
\lim\limits_{\epsilon\rightarrow0}\bigg(\frac{n+2s}{2(2_{s}^{\sharp}-1-\epsilon)}-\frac{n-2s}{2}\bigg)&\mu_{\epsilon}^{\frac{4s+(n-2s)\epsilon}{2s}}\int_{\mathbb{R}^n}W^{2_{s}^{\sharp}}(x)dx
\\ \rightarrow&\frac{(n-2s)^2\gamma_{n,s}b_{n,s}^2}{2\kappa_{s}\alpha^2_{n,s}\widetilde{\beta}_{n,s}}\phi(x_0)\lim\limits_{r\rightarrow0}\int_{\partial B_{r}^{+}}\frac{y^{1-2s}}{r^{n+2s-1}}dS=\frac{(n-2s)^2\gamma_{n,s}b_{n,s}^2}{2\kappa_{s}\alpha^2_{n,s}\widetilde{\beta}_{n,s}}M_{n,s}\phi(x_0)
\end{split}
\end{equation*}
as $\delta\rightarrow0$, where
$$
M_{n,s}:=\sigma_n\int_{0}^{1}\frac{r^{n-1}}{(1-r^2)^s}dr.
$$
Combined with Lemma \ref{thm:existenceofweaksolution}, we complete the proof.
\end{proof}
\section{Proof of Theorem \ref{emm} and Theorem \ref{emm-1}}\label{section7}
In this section, we are devoted to show that Theorems \ref{emm} and \ref{emm-1}.
Set $w_{\epsilon}\in\mathcal{D}^{1,2}(\mathcal{C};y^{1-2s})$ be the $s$-harmonic extensions of the elliptic equation to $\mathcal{C}=\Omega\times(0,\infty)$, that is, $w_{\epsilon}$ satisfies
\begin{equation}\label{CF1010}
\left\lbrace
\begin{aligned}
&-\mbox{div}(y^{1-2s}\nabla w_{\epsilon})=0\hspace{4.14mm}\quad\quad \quad \quad \quad \quad \quad \quad  \mbox{in}\hspace{2mm} \mathcal{C},\\
&w_{\epsilon}=0\quad \quad \quad \quad \quad \quad \quad \hspace{12mm}\hspace{2mm}\hspace{8.9mm}\hspace{10mm} \mbox{on}\hspace{2mm}\partial_{L}\mathcal{C},\\
&w_{\epsilon}>0\quad \quad \quad \quad \quad \quad \quad \quad \hspace{10mm}\hspace{8.5mm}\hspace{10mm}\hspace{1mm}\mbox{in}\hspace{2mm}\mathcal{C},\\
&\partial_{\nu}^{s}w_{\epsilon}=\big(|x|^{-\mu}\ast w_{\epsilon}^{2_{\mu,s}^{\ast}}\big)w_{\epsilon}^{2_{\mu,s}^{\ast}-1}+\epsilon w_{\epsilon}\hspace{7.8mm}\quad \mbox{in}\hspace{2mm}\Omega\times\{y=0\},
\end{aligned}
		\right.
\end{equation}
We choose $x_\epsilon\in\Omega$ and the number $\mu_{\epsilon}>0$ such that
$\alpha_{n,\mu,s}\mu_{\epsilon}=\|u_{\epsilon}\|_{\infty}=u_\epsilon(x_\epsilon),\hspace{2mm}\mbox{where}\hspace{2mm}\alpha_{n,\mu,s}
$
is defined in \eqref{afal}.
We define a family of rescaled functions $v_{\epsilon}$ and $\widetilde{W}_{\epsilon}(z)$ as follows
$$v_{\epsilon}(x)=\mu_{\epsilon}^{-1}u_{\epsilon}(\mu_{\epsilon}^{-\frac{2}{n-2s}}x+x_{\epsilon})\quad\mbox{for}\quad x\in\Omega_{\epsilon}:=\mu_{\epsilon}^{\frac{2}{n-2s}}(\Omega-x_{\epsilon}).$$
and
$$
\widetilde{W}_{\epsilon}(z):=\mu_{\epsilon}^{-1}w_{\epsilon}(\mu_{\epsilon}^{-\frac{2}{n-2s}}z+x_{\epsilon})\quad\mbox{for}\quad z\in\mathcal{C}_{\epsilon}:=\mu_{\epsilon}^{\frac{2}{n-2s}}(\mathcal{C}-x_{\epsilon}).
$$
We set $\mathcal{C}_{\epsilon}^{\ast}:=T(\mathcal{C}_{\epsilon})$ and Kelvin transform transformation
\begin{equation*}
V_{\epsilon}(x)=\frac{1}{|x|^{n-2s}}v_{\epsilon}\big(\frac{x}{|x|^2}\big)\hspace{2mm}\mbox{and}\hspace{2mm}\widetilde{W}_{\epsilon}^{\ast}(z)=\frac{1}{|z|^{n-2s}}\widetilde{W}_{\epsilon}\big(\frac{z}{|z|^2}\big),
\end{equation*}
Then $\widetilde{W}_{\epsilon}^{\ast}$ satisfies
\begin{equation}\label{1-CF101}
\left\lbrace
\begin{aligned}
&-\mbox{div}(y^{1-2s}\nabla \widetilde{W}_{\epsilon}^{\ast})=0\hspace{4.14mm}\quad \quad \quad\quad \quad\quad\quad \quad \quad \quad\quad \quad \quad \quad  \hspace{9mm}\mbox{in}\hspace{2mm} \mathcal{C}_{\epsilon}^{\ast},\\
&\widetilde{W}_{\epsilon}^{\ast}=0\quad \quad \quad \quad \quad \quad\quad\quad \quad \quad \quad \quad \quad\hspace{12mm}\hspace{2mm}\hspace{10.9mm}\hspace{10mm}\hspace{7mm} \mbox{on}\hspace{2mm}\partial_{L}\mathcal{C}_{\epsilon}^{\ast},\\
&\widetilde{W}_{\epsilon}^{\ast}>0\quad \quad \quad \quad \quad \quad\quad\quad \quad \quad \quad \quad \quad\quad \hspace{10mm}\hspace{10.5mm}\hspace{10mm}\hspace{7mm}\hspace{1mm}\mbox{in}\hspace{2mm}\mathcal{C}_{\epsilon}^{\ast},\\
&\partial_{\nu}^{s}\widetilde{W}_{\epsilon}^{\ast}=\big(|x|^{-\mu}\ast (\widetilde{W}_{\epsilon}^{\ast})^{2_{\mu,s}^{\ast}}\big)(\widetilde{W}_{\epsilon}^{\ast})^{2_{\mu,s}^{\ast}-1}+\epsilon\frac{1}{\mu_{\epsilon}^{4s/(n-2s)}|x|^{4s}}\widetilde{W}_{\epsilon}^{\ast}\hspace{5.8mm} \mbox{in}\hspace{2mm}T(\Omega_{\epsilon}\times\{y=0\}).
\end{aligned}
		\right.
\end{equation}
As a preparation for the proof of Theorem \ref{emm}, we first consider the following result.
\begin{lem}\label{2-UV1}
Assume that $w\in H_{0,L}^{s}(\mathcal{C})$ is a solution of problem \eqref{CF1010}. Then for each $\rho>0$ and $q>\frac{n}{s}$, there is a constant $C=C(\rho,q)>0$ such that
\begin{equation}\label{1-idengity0}
\begin{split}
&\min\limits_{r\in[\rho,2\rho]}\Big|s\epsilon\int_{\mathcal{M}(\Omega,r/2)\times\{0\}}w^{2}dx\Big|\\
\leq&C\bigg[\Big(\int_{\mathcal{Q}(\Omega,2\rho)}\Big|\big(\int_{\Omega}\frac{w^p(t,0)}{|x-t|^{\mu}}dt\big)w^{p-1}(x,0)\Big|^qdx\Big)^{\frac{2}{q}}+\Big(\int_{\mathcal{M}(\Omega,\rho/2)}\Big|\big(\int_{\Omega}\frac{w^p(t,0)}{|x-t|^{\mu}}dt\big)w^{p-1}(x,0)\Big|dx\Big)^{2}\\&
+\int_{\mathcal{Q}(\Omega,2\rho)}\Big|\big(\int_{\Omega}\frac{w^p(t,0)}{|x-t|^{\mu}}dt\big)w^{p}(x,0)\Big|dx+\int_{\mathcal{M}(\Omega,r/2)}\Big(\int_{\Omega}x(x-t)\frac{w^p(t,0)}{|x-t|^{\mu+2}} dt\Big)w^p(x,0)dx\bigg]\\&+C\Big[\int_{\mathcal{Q}(\Omega,2\rho)}w^2dx\Big].
\end{split}
\end{equation}
 \end{lem}
 \begin{proof}
We note that
 \begin{equation}\label{Qx-100}
 \epsilon\int_{\mathcal{M}(\Omega,r/2)\times\{0\}}\langle x,\nabla_{x}w\rangle wdx=-\frac{n}{2}\epsilon\int_{\mathcal{M}(\Omega,r/2)\times\{0\}}w^2dx+\frac{1}{2}\epsilon\int_{\partial\mathcal{M}(\Omega,r/2)\times\{0\}}w^2\langle x,\nu\rangle dS_{x}.
 \end{equation}
With the help of \eqref{Qx-100}
and similarly to the argument of \eqref{Qx}, we obtain a Pohozaev-type
identity
 \begin{equation*}
\begin{split}
&\kappa_{s}\Big(\frac{n}{p}-\frac{n-2s}{2}\Big)\int_{\mathcal{M}(\Omega,r/2)} \int_{\Omega}\frac{w^p(t,0)}{|x-t|^{\mu}}dtw^{p}(x,0)dx+sk_s\epsilon\int_{\mathcal{M}(\Omega,r/2)\times\{0\}}w^{2}dx\\&
=\int_{\partial\mathcal{D}_r^{+}}y^{1-2s}\Big\langle\langle z, \nabla w\rangle\nabla w-\frac{|\nabla w|^2}{2}z,\nu\Big\rangle dS_x+\frac{N-2s}{2}\int_{\partial\mathcal{D}_r^{+}}y^{1-2s}w \frac{\partial w}{\partial\nu}dS_x\\&
\hspace{3.5mm}+\frac{\mu}{p}\int_{\mathcal{M}(\Omega,r/2)} \int_{\Omega}x(x-t)\frac{w^p(t,0)}{|x-t|^{\mu+2}}dtw^{p}(x,0)dx
+\frac{1}{p}\int_{\partial\mathcal{M}(\Omega,r/2)} \int_{\Omega}\frac{w^p(t,0)}{|x-t|^{\mu}}dtw^{p}(x,0)\langle x,\nu\rangle dS_x\\&\hspace{3.5mm}+
\frac{1}{2}\epsilon\int_{\partial\mathcal{M}(\Omega,r/2)\times\{0\}}w^2\langle x,\nu\rangle dS_{x}.
\end{split}
\end{equation*}
Since
$
\Big(\frac{2n-\mu}{2\cdot2_{\mu,s}^{\ast}}-\frac{n-2s}{2}\Big)\leq\Big(\frac{n}{p}-\frac{N-2s}{2}\Big),
$
and hence the rest of the proof is similar to Proposition \ref{prosition-identity}.
 \end{proof}
 Using \eqref{1-idengity0}, we find the following.
 \begin{lem}\label{3-UV1}
Assume that $0<s<1$ and $n>4s$. There exists a constant $M>0$ such that for $\epsilon>0$ small, there holds
\begin{equation}\label{eq}
\mu_{\epsilon}\leq C\epsilon^{-\frac{n-2s}{2(n-4s)}}.
\end{equation}
  \end{lem}
\begin{proof}
A direct calculation yields
\begin{equation}\label{bredecay-120}
\begin{split}
\epsilon\int_{\mathcal{M}(\Omega,r/2)\times\{0\}}w^{2}dx&=\epsilon\int_{\mathcal{M}(\Omega,r/2)\times\{0\}}\mu_{\epsilon}^2v_{\epsilon}^2(\mu_{\epsilon}^{\frac{2}{n-2s}}(x-x_{\epsilon}))dx\\&
\geq\epsilon\mu_{\epsilon}^{-\frac{4s}{n-2s}}\int_{B_{n}(0,1)}v_{\epsilon}^2(x)dx\geq C\epsilon\mu_{\epsilon}^{-\frac{4s}{n-2s}},
\end{split}
\end{equation}
and since $x_0\not\in\mathcal{Q}(\Omega,2\rho)$, we have
\begin{equation}\label{bredecay-121}
\int_{\mathcal{Q}(\Omega,2\rho)}w^2 dx\leq C(\frac{1}{\mu_{\epsilon}})^2.
\end{equation}
Consider the first term on the RHS of \eqref{1-idengity0}. Then analogous to the argument of \eqref{Right} and \eqref{bredecay-121}, we obtain
\begin{equation}\label{bredecay-122}
\mbox{RHS of}\hspace{2mm}\eqref{1-idengity0}\leq C\big(\frac{1}{\mu_{\epsilon}}\big)^{2}.
\end{equation}
Hence \eqref{1-idengity0}, \eqref{bredecay-120} and \eqref{bredecay-122} imply that \eqref{eq}.

\end{proof}
 The proof of the next results are similar to Lemma \ref{cWU} by applying Lemma \ref{3-UV1}.
\begin{lem}\label{1-UV1}
There exists a constant $C>0$ independently of $\epsilon>0$ provided $\epsilon$ is sufficiently small, such that
\begin{equation}\label{bredecay-123}
u_\varepsilon(x)\leq CW[\mu_{\epsilon}^{\frac{2}{n-2s}},x_{\epsilon}](x).
\end{equation}
\end{lem}
 \begin{proof}
 By Lemma \ref{3-UV1}, for any $\delta>0$ small, we estimate
 \begin{equation}\label{timate}
 \begin{split}
 \Big\|\epsilon\frac{1}{\mu_{\epsilon}^{4s/(n-2s)}|x|^{4s}}\Big\|_{L^{\frac{n}{2s}+\delta}(\mathcal{C}_{\epsilon}^{\ast})}
 &\leq \epsilon\Big(\int_{\{|x|\geq \frac{1}{c_0}\mu_{\epsilon}^{-\frac{2}{n-2s}}\}}\mu_{\epsilon}^{-\frac{4s}{n-2s}(\frac{n}{2s}+\delta)}|x|^{-(2n-4s\delta)}dx\Big)^{\frac{1}{n/(2s)+\delta}}
 \\&
 \leq \epsilon\mu_{\frac{8s^2\delta}{(n-2s)(n+\delta)}}\leq \epsilon\cdot\epsilon^{-\frac{n-2s}{2(n-4s)}\cdot\frac{8s^2\delta}{(n-2s)(n+\delta)}}\leq1.
 \end{split}
 \end{equation}
Estimate \eqref{bredecay-123} can be proven by a similar argument in Lemma \ref{cWU} and together with \eqref{timate}.
 \end{proof}

We can now prove Theorem \ref{emm}.
 \begin{proof}[Proof of Theorem \ref{emm}]
 Based on Lemma \ref{1-UV1}, the conclusion follows as in proof of Theorems \ref{prondgr}-\ref{consequence}.
 \end{proof}

We conclude the section by proving Theorem \ref{emm-1}.
\begin{proof}[Proof of Theorem \ref{emm-1}]
(a) The part (1) follows by a similar argument as in the proof Theorem \ref{prondgr-1}.

For part (b), by using a calculation similar to \eqref{remainder-2}, we derive
 \begin{equation*}
\begin{split}
\lim\limits_{\epsilon\rightarrow0}&\epsilon s\kappa_{s}\alpha^2_{n,\mu,s}\mu_{\epsilon}^{\frac{2(n-4s)}{n-2s}}\delta\int_{\mathbb{R}^n}W^{2}(x)dx\\&
=d_{n,s}^2\int_{\delta}^{2\delta}\int_{\partial B_r^{+}}\bigg[y^{1-2s}\Big\langle\langle z-X_0, \nabla G\rangle\nabla G-\frac{|\nabla G|^2}{2}(z-X_0),\nu\Big\rangle+\frac{n-2s}{2}y^{1-2s}G \frac{\partial G}{\partial\nu}\bigg]dS_xdr.
\end{split}
\end{equation*}
Then similar to argument of Theorem \ref{remainder terms}, we have
\begin{equation*}
\begin{split}
\lim\limits_{\epsilon\rightarrow0}&\epsilon\mu_{\epsilon}^{\frac{2(n-4s)}{n-2s}}\int_{\mathbb{R}^n}W^{2}(x)dx
\rightarrow\frac{(n-2s)^2\gamma_{n,s}d_{n,s}^2}{2sk_{s}}M_{n,s}\phi(x_0)
\end{split}
\end{equation*}
as $\delta\rightarrow0$, and concluding the proof.
 \end{proof}




\end{document}